\documentclass[10pt]{article}

\usepackage{amssymb, latexsym, amsfonts, pigpen, enumitem, mathrsfs, amsthm, bbm, stackrel, todonotes, mathtools,caption, scalerel}
\usepackage{longtable}
\usepackage[matrix,arrow,curve]{xy}
\usepackage[letterpaper,top=1.05in, bottom=1.05in, left=1.05in, right=1.05in]{geometry}
\usepackage[colorlinks=true,pagebackref=true]{hyperref}

\usepackage{epigraph}
\usepackage[polutonikogreek,english]{babel}
\usepackage[or]{teubner}

\newtheorem{theorem}{Theorem}[subsection]
\newtheorem{theorem_intr}{Theorem}
\newtheorem{lemma}[theorem]{Lemma}
\newtheorem{proposition}[theorem]{Proposition}
\newtheorem{corollary}[theorem]{Corollary}

\theoremstyle{definition}
\newtheorem{definition}[theorem]{Definition}

\newtheorem{expos}[theorem]{}

\theoremstyle{remark}
\newtheorem{remark}[theorem]{Remark}

\numberwithin{equation}{theorem}

\newcommand{\mc}[1]{\mathcal{#1}}
\newcommand{\mr}[1]{\mathrm{#1}}

\newcommand{\OGr}{{\mathrm{Gr}}}

\newcommand{\Hom}{{\mathrm{Hom}}}

\newcommand{\Z}{{\mathbb{Z}}}

\newcommand{\AAA}{\mathbb{A}}

\newcommand{\PP}{\mathbb{P}}
\newcommand{\NN}{\mathbb{N}}
\newcommand{\rank}{\operatorname{rank}}

\newcommand{\id}{\operatorname{id}}

\newcommand{\FF}{\mathbb{F}}

\newcommand{\struct}{\mathcal{O}}
\newcommand{\Gmm}{{\mathbb{G}_m}}

\newcommand{\Spec}{\operatorname{Spec}}

\newcommand{\Pic}{\operatorname{Pic}}

\newcommand{\chark}{\operatorname{char}}

\newcommand{\Smk}{\mathcal{S}m_k}
\newcommand{\SmF}{\mathcal{S}m_F}

\newcommand{\Ch}{\mathrm{Ch}}
\newcommand{\CH}{\mathrm{CH}}
\newcommand{\Res}{\operatorname{Res}}
\newcommand{\Image}{\operatorname{Im}}

\newcommand{\Sp}{\operatorname{\mathrm{Sp}}}

\newcommand{\Spin}{\operatorname{\mathrm{Spin}}}

\newcommand{\SU}{\operatorname{\mathrm{SU}}}

\newcommand{\SL}{\operatorname{\mathrm{SL}}}

\newcommand{\Gal}{\operatorname{\mathrm{Gal}}}

\newcommand{\mut}[2]{
  \mathchoice{\raisebox{-1.5pt}{$\displaystyle {}^{#1}\mkern-2mu$}}
             {\raisebox{-1.5pt}{${}^{#1}\mkern-2mu$}}
             {\raisebox{-1.5pt}{$\scriptstyle {}^{#1}\mkern-3mu$}}
             {\raisebox{-0.2pt}{$\scriptscriptstyle {}^{#1}\mkern-2mu$}}
             \mu_{#2}}

\newcommand{\Att}{{}^{2}\mkern-2mu \mr{A}}
\newcommand{\refbr}[1]{{\hyperref[#1]{(\ref*{#1})}}}

\newcommand{\bigslant}[2]{{\left.\raisebox{.2em}{$#1$}\middle/\raisebox{-.2em}{$#2$}\right.}}

\newcounter{counter:enumi}

    \newcommand{\Addresses}{{
		\bigskip
		\footnotesize
		
		Alexey~Ananyevskiy, Mathematisches Institut der Universit\"at M\"unchen, Theresienstr. 39, D-80333 M\"unchen, Germany. \textit{E-mail address: alseang@gmail.com}
		\medskip
		
		Nikita~Geldhauser, Mathematisches Institut der Universit\"at M\"unchen, Theresienstr. 39, D-80333 M\"unchen, Germany. \textit{E-mail address: geldhauser@math.lmu.de}
}}

\begin{document}
	
	\title{\bf Chow rings of quasi-split geometrically almost simple algebraic groups}
	
	\author{Alexey Ananyevskiy and Nikita Geldhauser}

\date{}

	\maketitle
 
\epigraph{
\hspace*{\fill}Dedicated to Professor Nikolai Vavilov (1952 -- 2023)\\
\hspace*{\fill}who taught us the numerology of algebraic groups.\\
\hspace*{\fill}\textgreek{\Cs{v}Hjos \s{a}njr\A{w}p\iS{w} da\A{i}mwn.}
}

	\thanks{}
	
	\begin{abstract}
	We compute the Chow ring $\mathrm{Ch}^*(G)$ of a quasi-split geometrically almost simple algebraic group $G$ assuming the coefficients to be a field. This extends the classical computation for split groups done by Kac to the non-split quasi-split case. For the proof we introduce and study equivariant conormed Chow rings, which are well adapted to the study of quasi-split groups and their homogeneous varieties.
	\end{abstract}

        \tableofcontents
 
	\section{Introduction}

 Given a linear algebraic group $G$ one can associate with it two rings of cohomological origin, $\CH^*(BG)$ and $\CH^*(G)$, both of which carry important information about possible twisted forms of $G$, its torsors, homogeneous varieties and related motives. Classical topological analogues of these rings are the singular cohomology ring $\mathrm{H}^*(B\mathcal{G})$ of the classifying space of a Lie group $\mathcal{G}$ and the singular cohomology ring $\mathrm{H}^*(\mathcal{G})$ of the Lie group itself, and these rings are intimately related by the Leray--Serre spectral sequence for the fibration $E\mathcal{G}\to B\mathcal{G}$. 
 
 The ring $\CH^*(BG)$ is the Chow ring of algebraic cycles of the classifying space $BG$ introduced by Totaro~\cite{To99}, see also \cite[Section~4]{MV99} for a motivic homotopy point of view, and it hosts the universal characteristic classes of principal homogeneous spaces (torsors) under $G$. This ring is in general rather hard to compute even for a split simple group except the special cases of $\SL_n$ and $\Sp_{2n}$, see \cite{To99,Pa98,MRV06,KM22} for the known answers.
 
 The ring $\CH^*(G)$ is the Chow ring of algebraic cycles of $G$ itself, considered as an algebraic variety forgetting the multiplicative structure.  The problem of computing $\CH^*(G)$ for a split semisimple group $G$, and the closely related problem of computing the singular cohomology of a compact Lie group attracted a lot of attention in the 1960s and was done case by case in the works of Miller, Borel, Baum, Browder, Araki, Shikata, Ishitoya, Kono, Toda and others, culminating in the celebrated article \cite{Kac85} of Victor Kac, who summarized the previously known computations and gave a uniform conceptual answer for $\CH^*(G)$ assuming the coefficients to be a finite field (see \cite[Theorem~6]{Kac85} and especially \cite[Table~2]{Kac85}).


The computations by Victor Kac were used  20 years later by Geldhauser, Petrov, and Zainoulline to establish the structure of Chow motives of generically split twisted flag varieties and to introduce a discrete motivic invariant of algebraic groups of inner type, called the $J$-invariant (see \cite{PSZ08, PS10, PS12}).
In the case of quadratic forms an equivalent invariant was introduced previously by Vishik in \cite{Vi05}.
The $J$-invariant was an important tool to solve several long-standing problems, for example, it played an important role in the progress on the Kaplansky problem about possible values of the $u$-invariant of fields by Vishik \cite{Vi07} and in the solution of a problem of Serre about groups of type $\mathrm{E}_8$ and their finite subgroups \cite{S16, GS10}.
Besides, Garibaldi, Geldhauser, and Petrov used the $J$-invariant to relate the rationality of some parabolic subgroups of groups of type $\mathrm{E}_7$ with the Rost invariant, proving a conjecture of Rost and solving a question of Springer in \cite{GPS16}.
The concept of the $J$-invariant is built on the first three columns of \cite[Table~2]{Kac85} encoding the Chow rings of split semisimple algebraic groups.
In fact, an important step in the construction of the $J$-invariant was an observation that the parameters in the Kac's table coincide in certain cases with the parameters related to the generalized Rost motives, which were used by Rost and Voevodsky in the proofs of the Milnor and Bloch--Kato conjectures. The last column of Kac's table also has a motivic interpretation in terms of $G$-equivariant motives given in \cite{PS17}.

In order to generalize the $J$-invariant to arbitrary linear algebraic groups (i.e. possibly of outer type) one has to study the Chow rings of non-split quasi-split semisimple groups. Progress in this direction was recently obtained in \cite{SZ22}. However, there are several substantial difficulties when dealing with non-split quasi-split groups. First
of all, quasi-split groups $G$ are equipped with a finite Galois field extension $K/k$, which is the minimal splitting field of $G$, i.e., come together with some arithmetic information. A priori it is not clear whether their Chow rings depend on $K$, and we show in the current article that it is not the case. Besides, contrary to the split situation, our computations show that the Chow rings of non-split quasi-split groups do not have, in general, the structure of Hopf algebras. For example, the lacunary relation appearing additionally to the nilpotency of generators in the non-split quasi-split adjoint group of type ${}^2\mr{D}_{2r}$ (case 2 with $l = 2$ of Theorem~\ref{thm:main} below) is not possible in any Hopf algebra over a finite field. Moreover, contrary to the split situation, the characteristic sequence of Grothendieck (see \cite[Remark~$2^\circ$]{Gro58}), which is a useful tool, is not right exact in general for quasi-split groups.

To overcome these difficulties we develop a theory of equivariant conormed Chow rings and apply it to quasi-split groups. The non-equivariant version of these, i.e. Chow rings modulo pushforwards (norms) coming from a fixed field extension, have already been used in the study of varieties homogeneous under an action of a quasi-split group in e.g. \cite{Ka12, KZ13, KM22} and were studied explicitly in \cite{Fi19,SZ22}. We would also like to note that the category of conormed Chow motives shares some properties with the category of isotropic motives introduced and studied by Vishik \cite{Vi24}, in particular, in the latter category motives of anisotropic varieties are zero while in the former category the motive of an irreducible $k$-variety is zero when the variety becomes reducible over the (Galois) field extension $K/k$.

The main result of this article is the following theorem:
 
	\begin{theorem_intr} \label{thm:main}
	Let $G$ be a quasi-split geometrically almost simple group over a field $k$. For a field $\FF$ put $p:=\chark \FF$ and $\Ch^*(-):=\CH^*(-)\otimes \FF$. Then, depending on the type $\Delta(G)$, the fundamental group $\pi_1(G)$ (see Sections~\ref{sec:multiplicative_groups} and \ref{sec:quasi-split_simple} for the notation) and~$p$ the following holds.
        \begin{enumerate}
            \item If the triple $(\Delta(G),\pi_1(G),p)$ is in the Table~\ref{tab:main}, then
        \[
		\Ch^*(G) \cong \FF[e_1,e_2,\hdots,e_s]/(e^{p^{k_1}}_1,e^{p^{k_2}}_2,\hdots, e^{p^{k_s}}_s),\quad \deg e_i=d_i,
		\]
		with the parameters $s,d_i,k_i$ given in the table.
		\begin{center}
			\def\arraystretch{1.5}
			\begin{longtable}{l|l|l|l|l|l}
				\caption{$\Ch^*(G)$ for a quasi-split simple algebraic group $G$}
				\label{tab:main}\\
				$\Delta(G)$ & $\pi_1(G)$ & $p$ & $s$ & $d_i$, $i=1,2,\ldots, s$ & $k_i$, $i=1,2,\ldots, s$ \\
				\hline
				${}^{\hphantom{2}}\mr{A}_{n}$ & $\mut{\hphantom{2}}{l}$, $l\,|\, (n+1)$ & $p\mid l$ & $1$ & $1$ & $v_p(n+1)$ \\
				
				$\Att_{n}$, $n\ge 2$ & $\mut{2}{l}$, $l\,|\, (n+1)$, $l$ is odd & $2$ & $[\frac{n}{2}]$ & $2i+1$ & $1$ \\
				
				$\Att_{2r-1}$, $r\ge 2$ & $\mut{2}{2m}$, $m\mid r$, $m$ is odd & $2$ & $r$ & $2i-1$ & $v_2(r)+1$, $i=1$ \\
				& & & & & $1$, $i\ge 2$ \\
				
				$\Att_{2r-1}$, $r\ge 2$ & $\mut{2}{2m}$, $m\mid r$, $m$ is even & $2$ & $r+1$ & $2$, $i=1$ & $v_2(r)$, $i=1$ \\
				& & & & $2i-3$, $i\ge 2$ & $1$, $i\ge 2$ \\		
				
				\hline
				${}^{\hphantom{2}}\mr{B}_{n}$ & $1$ & $2$ & $[\frac{n-1}{2}]$ & $2i+1$ & $[\log_2 \frac{2n}{2i+1}]$ \\
				
				& $\mu_2$& $2$ & $[\frac{n+1}{2}]$ & $2i-1$ & $[\log_2 \frac{2n}{2i-1}]$ \\
				\hline
				${}^{\hphantom{2}}\mr{C}_{n}$ & $\mu_2$ & $2$ & $1$ & $1$ & $v_2(n)+1$ \\
				\hline
				${}^{\hphantom{2}}\mr{D}_{n}$, $n\ge 3$ & $1$ & $2$ & $[\frac{n}{2}]-1$ & $2i+1$ & $[\log_2 \frac{2n-1}{2i+1}]$ \\						
				
				& $\mu_2^{so}$ & $2$ & $[\frac{n}{2}]$ & $2i-1$ & $[\log_2 \frac{2n-1}{2i-1}]$ \\

				& $\mu_4$ or $\mu_2\times \mu_2$ & $2$ & $[\frac{n}{2}]+1$ & $1$, $i=1$& $v_2(n)$, $i=1$\\						
				&  & & & $2i-3$, $i\ge 2$ & $[\log_2 \frac{2n-1}{2i-3}]$, $i\ge 2$ \\						
				
				${}^{\hphantom{2}}\mr{D}_{2r}$, $r\ge 2$	& $\mu_2^{hs}$ & $2$ & $r$ & $1$, $i=1$ & $v_2(r)+1$, $i=1$ \\			
				&  & & & $2i-1$, $i\ge 2$ & $[\log_2 \frac{4r-1}{2i-1}]$, $i\ge 2$ \\					
				${}^2\mr{D}_{n}$, $n\ge 3$ & $1$ & $2$ & $[\frac{n+1}{2}]-1$ & $2i+1$ & $[\log_2 \frac{2n}{2i+1}]$ \\						
				
				& $\mu_2$ & $2$ & $[\frac{n+1}{2}]$ & $1$, $i=1$ & $[\log_2 n]+1$ \\						
				&  &  &  & $2i-1$, $i\ge 2$ & $[\log_2 \frac{2n}{2i-1}]$ \\								
				
				${}^2\mr{D}_{2r+1}$, $r\ge 1$	& $\mut{2}{4}$ & $2$ & $r+2$ & $1$, $i=1$ & $1$, $i=1$ \\			
				& & & & $2$, $i=2$ & $[\log_2 (2r+1)]$, $i=2$ \\					
				&  & & & $2i-3$, $i\ge 3$ & $[\log_2 \frac{4r+2}{2i-3}]$, $i\ge 3$ 		\\
                ${}^3\mr{D}_4$, ${}^6\mr{D}_4$ & $1$ & $2$ & $1$ & $3$ & $1$\\
    
				                & $1$, $\mut{3}{2,2}$, $\mut{6}{2,2}$& $3$ & $1$ & $4$ & $1$\\

				\hline
				${}^{\hphantom{2}}\mr{E}_6$ & $1,\mu_3$& $2$ & $1$ & $3$ & $1$ \\			
				& $1$ & $3$ & $1$ & $4$ & $1$ \\	
				& $\mu_3$ & $3$ & $2$ & $1,4$ & $2,1$ \\	
                ${}^2\mr{E}_6$ & $1,\,\mut{2}{3}$& $2$ & $3$ & $3,5,9$ & $1,1,1$    \\
                 & $1$ & $3$ & $1$ & $4$ & $1$    \\
				\hline
				${}^{\hphantom{2}}\mr{E}_7$ & $1$ & $2$ & $3$ & $3,5,9$ & $1,1,1$ \\
				
				& $\mu_2$ & $2$ & $4$ & $1,3,5,9$ & $1,1,1,1$ \\
				
				& $1,\mu_2$ & $3$ & $1$ & $4$ & $1$ \\
				\hline
				${}^{\hphantom{2}}\mr{E}_8$& $1$  & $2$ & $4$ & $3,5,9,15$ & $3,2,1,1$  \\
				
				& & $3$ & $2$ & $4,10$ & $1,1$ \\
				
				& & $5$ & $1$ & $6$ & $1$ \\
				\hline
				${}^{\hphantom{2}}\mr{F}_4$ & $1$ &  $2$ & $1$ & $3$ & $1$ \\
				& &  $3$ & $1$ & $4$ & $1$ \\			
				\hline
				
				${}^{\hphantom{2}}\mr{G}_2$ & $1$ & $2$ & $1$ & $3$ & $1$ \\
				
			\end{longtable}
		\end{center}
        Here and below $v_p$ is the $p$-adic valuation.
        
        \item If $(\Delta(G),\pi_1(G),p)=({}^2\mr{D}_{2r},\mut{2}{2,2},2)$, then 
        \[
            \Ch^*(G) \cong
            \begin{cases}
            \FF[e_1,\hdots,e_s]/\left(e_1^{2^{k_1}}, \hdots, e_s^{2^{k_s}}, e_1\cdot \left(\sum_{j=0}^{k_1-1} e_1^{2^{k_1}-2^{j+1}} \cdot e_2^{2^{j}}\right)\right), & l=2,\\
            \FF[e_1,\hdots,e_s]/(e_1^{2^{k_1}}, \hdots, e_s^{2^{k_s}}, e_1\cdot e_l^{2^{k_l-1}}), & l\ge 3,
            \end{cases}
                \]
            where $l\in \NN$ is such that $2r=2^m(2l-3)$ and
            \begin{gather*}
            s=r+1,\, \deg e_1=1,\, \deg e_2=2,\, k_1=v_2(2r), \, k_2=[\log_2 2r],\\
            \deg e_i=2i-3,\, k_i=\left[\log_2 \frac{4r}{2i-3}\right],\, 3\le i\le s.
            \end{gather*}
        \item
        If $(\Delta(G),\pi_1(G),p)$ is in the table below, then $\Ch^*(G)$ is as in the table.
            \begin{center}
			\def\arraystretch{1.5}
			\begin{longtable*}{l|l|l|l|l}
    			$\Delta(G)$ & $\pi_1(G)$ & $p$ & $\Ch^*(G)$ & $\deg e_i$ \\
				\hline
                    $\Att_n$  & $\mut{2}{l}$, $l\mid n+1$ & $2\neq p\mid l$ & $\FF[e_1]/(e_1^m)$, $m=\frac{p^{v_p(n+1)}+1}{2}$ & $2$\\
				${}^3\mr{D}_4$  & $\mut{3}{2,2}$ & $2$ & $\FF[e_1,e_2,e_3,e_4]/(e_1^4,e_2^2,e_3^2,e_4^2,e_1e_2,e_1e_3,e_1^3+e_2e_3)$ & $2,3,3,3$\\
                ${}^6\mr{D}_4$  & $\mut{6}{2,2}$ & $2$ & $\FF[e_1,e_2,e_3]/(e_1^4,e_2^2,e_1e_2,e_3^2)$ & $2,3,3$\\
                ${}^2\mr{E}_6$ & $\mut{2}{3}$ & $3$ & $\FF[e_1,e_2]/(e_1^5,e_2^3)$ & $2,4$ 
		\end{longtable*}
		\end{center}        

        \item For all the remaining cases $(\Delta(G),\pi_1(G),p)$ one has $\Ch^*(G)\cong \FF$. 
        \end{enumerate}

    \end{theorem_intr}
    \noindent
    It is clear that one can assume $\FF$ to be a prime field, that is $\mathbb{Q}$ or a finite field $\FF_p$. If $\FF=\mathbb{Q}$, then it is well-known and goes back to \cite[Remark~2 on p.~21]{Gro58} that $\Ch^*(G)\cong \mathbb{Q}$ (see also Section~\ref{expos:Chow_rational} for the details). The case of a split group $G$, i.e. $\Delta(G)=\mr{A}_n$, $\mr{B}_n$, $\mr{C}_n$, $\mr{D}_n$, $\mr{E}_6$, $\mr{E}_7$, $\mr{E}_8$, $\mr{F}_4$, or $\mr{G}_2$, is also classical \cite[Theorem~6]{Kac85} (see also Section~\ref{expos:Chow_ring_split} for a recollection).
    All the remaining cases, that is when $G$ is a non-split quasi-split simple group, are, to the best of our knowledge, completely new. Theorem~\ref{thm:main} is obtained as a combination of the above classical results and Theorems~\ref{thm:D2r_ad_answer},~\ref{thm:answer_at_p},~\ref{thm:answer_away_p}, and~\ref{thm:6D4} of the present article.


    \textbf{Outline of the proof.} We assume $\FF$ to be a prime field. Furthermore, as it was said above, using the classical results (see Section~\ref{sec:Chow_ring_classic} for references) we can assume that $G$ is a non-split quasi-split geometrically almost simple group. Let $K/k$ be the splitting field of $G$. The cases are naturally divided into three groups: (1) $p=[K:k]$, (2) $p$ is coprime to $[K:k]$ and (3) $\Delta(G)={}^6\mr{D}_4$ with $p=2$ or $3$. The first case is the most interesting one, while the latter two are rather straightforward.

    \textbf{Case (1)}, $p=[K:k]$. Let $\pi\colon G_K\to G$ be the projection and consider the exact sequence
    \begin{equation} \label{eq:intro_conormed}
        \Ch^*(G_K)\xrightarrow{\pi_*} \Ch^*(G) \xrightarrow{q} \CH^*_K(G) \to 0 \tag{*}
    \end{equation}
    with the last algebra being the cokernel of $\pi_*$. First we analyse the algebra $\CH^*_K(G)$, which fits into the setting of \textit{conormed Chow groups} $\CH^*_K(-)$, i.e. Chow groups modulo the image of the norm map. Such groups proved to be a valuable invariant in the study of quasi-split homogeneous varieties and implicitly occurred in e.g. \cite{Ka12, KZ13, KM22} and were studied explicitly in \cite{Fi19,SZ22}. As it was shown in \cite{SZ22} one can define $\CH^*_K(-)$ for an arbitrary separable field extension $K/k$ (see Definition~\ref{def:conormed}), yielding an oriented cohomology theory on $\Smk$ in the sense of \cite{LM07} whose basic properties we study in Section~\ref{sec:conormed}. The most important property is that $\CH^*_K(X)=0$ for a connected variety $X$ that becomes disconnected over an intermediate Galois extension $K/L/k$ (see Corollary~\ref{cor:triv_normed}). This, in particular, means that Artin--Tate motives (motives of quasi-split projective homogeneous varieties are among the examples \cite[Lemma~29]{CM06}) behave like Tate motives from the conormed point of view.
    
    Following the approach by Totaro \cite{To99} and Edidin-Graham \cite{EG98} for the ordinary Chow groups we extend the conormed Chow groups to the equivariant setting (Section~\ref{sec:conormed_equiv}) defining $\CH^*_{K,T}(X)$ for an algebraic group $T$ over $k$ and a $T$-variety $X$. If $T$ is a torus which is (a) \textit{quasi-trivial} (also known as \textit{induced}), i.e. $T\cong R_{L_1/k} \Gmm \times \hdots \times R_{L_r/k} \Gmm$ for Weil restrictions of split tori, and such that (b) $T_K$ is split, then $\CH^*_{K,T}(-)$ behave very much like equivariant Chow groups for split tori, in particular, one can compute $\CH^*_{K,T}(X)$ for a smooth variety $X$ with the trivial $T$-action (see Proposition~\ref{prop:classifying_quasi-trivial}) as
    \[
    \CH^*_{K,T}(X)\cong \CH^{*}_K(X)[b_1,b_2,\hdots,b_r],\quad \deg b_i=[L_i:k], 
    \]
    and for a smooth variety $X$ with an arbitrary $T$-action one has an isomorphism (Corollary~\ref{cor:EM_qt})
    \begin{equation} \label{eq:intro}
        \CH_{K,T}^*(X) \otimes_{\CH_{K,T}^*(\Spec k)} \CH_{K}^*(\Spec k)\cong \CH_{K}^*(X). \tag{**}
    \end{equation}
    A similar result was also obtained in \cite[Theorem~3.1]{Sal22}, but in a slightly different language.
    
    Now let $T\le B\le G$ be a maximal torus and a Borel subgroup, and $K$ be the splitting field of $G$ as before. If $T$ is quasi-trivial, which is the case if $G$ is simply connected or adjoint (Lemma~\ref{lem:quasi-split_quasi-trivial}), then the isomorphism~\refbr{eq:intro} for $X:=G$ yields that the characteristic sequence (Definition~\ref{def:characteristic_map})
    \[
    \CH^*_{K,T}(\Spec k) \xrightarrow{c} \CH_K^*(G/B) \xrightarrow{\phi^*} \CH^*_K(G) \to 0
    \]
    is an exact sequence of graded algebras in the sense that $\phi^*$ is surjective and $c(\CH^{>0}_{K,T}(\Spec k))$ generates the kernel of $\phi^*$ as an ideal. This is the conormed version of the classical characteristic sequence from \cite{Gro58} which since its introduction has been one of the main tools for the computations related to the Chow rings of split reductive groups. The cokernel of the conormed characteristic map
    \[
    \mc{C}^*_K(G):= \CH^*_K(G/B) /c(\CH^{>0}_{K,T}(\Spec k))\cdot \CH^*_K(G/B)
    \]
    was computed in the adjoint case in \cite[Section~8]{SZ22}. This gives an answer for $\CH^*_K(G)$ in the adjoint case, since the exactness of the characteristic sequence yields $\CH^*_K(G)\cong \mc{C}^*_K(G)$.
    
    If $G$ is neither adjoint nor simply connected, then in general the characteristic sequence is not exact, since $\phi^*$ fails to be surjective already for $\CH^1_K(-)$, i.e. for the conormed Picard groups. In order to circumvent this issue we construct an \textit{extended characteristic sequence} in the following way. Let $\tilde{G}\to G$ be the simply connected cover and $\tilde{T}\le \tilde{B}\le \tilde{G}$ be the maximal torus and the Borel subgroup over $T\le B$. Then $\tilde{T}$ acts on $G$ with $\pi_1(G)$ being the kernel of the action and we have isomorphisms
    \[
    \CH^*_{K,\tilde{T}}(G) \cong \CH^*_{K,\pi_1(G)}(\tilde{G}/\tilde{T})  \cong \CH^*_{K}(\tilde{G}/\tilde{B}) \otimes \CH^*_{K,\pi_1(G)}(\Spec k),
    \]
    where the first one arises from the span of torsors $G\leftarrow \tilde{G} \to \tilde{G}/\tilde{T}$ and the second one is a combination of a K\"unneth property (Lemma~\ref{lem:Kunneth_for_pi1}) and homotopy invariance for $\tilde{G}/\tilde{T}\to \tilde{G}/\tilde{B}$. Using this and the isomorphism~\refbr{eq:intro} for $X:=G$ and the quasi-trivial torus $\tilde{T}$ we obtain an extended characteristic sequence
    \[
    \CH^*_{K,\tilde{T}}(\Spec k) \xrightarrow{\hat{c}} \CH_K^*(\tilde{G}/\tilde{B}) \otimes \CH^*_{K,\pi_1(G)}(\Spec k) \to \CH^*_K(G) \to 0
    \]
    which is an exact sequence of graded algebras (Theorem~\ref{thm:char_qs}). The algebras $\CH^*_{K,\pi_1(G)}(\Spec k)$ arising here can be computed (Proposition~\ref{prop:equivariant_coef_nl}) using resolutions by quasi-trivial tori, yielding an explicit presentation of $\CH^*_K(G)$ as a quotient of $\CH_K^*(\tilde{G}/\tilde{B})$ or $\CH_K^*(\tilde{G}/\tilde{B})[x]$ (Theorem~\ref{thm:group_via_cover}). In particular, this gives a precise relation between $\CH^*_K(G)$ and already known $\CH^*_K(\bar{G})$ for the adjoint quotient $G\to \bar{G}$, allowing us to compute $\CH^*_K(G)$ (Theorem~\ref{thm:conormed_group_answer}).
    
    We also give a presentation of $\CH^*_K(G)$ as a quotient of $\mc{C}^*_K(G)[x]$ for the cokernel $\mc{C}^*_K(G)$ of the conormed characteristic map (Theorem~\ref{thm:cokernel}), and use it to compute $\mc{C}^*_K(G)$ explicitly (Theorem~\ref{thm:cokernel_answer}). Moreover, as a by-product of the isomorphism~\refbr{eq:intro} we show that $G$ enjoys the conormed K\"unneth property (Proposition~\ref{prop:Kunneth_conormed}), endowing $\CH^*_K(G)$ with a natural structure of a Hopf algebra, which in turn explains the structure of the answer for $\CH^*_K(G)$ obtained in Theorem~\ref{thm:conormed_group_answer}. Note that in general the K\"unneth homomorphism
    \[
    \CH^*(X)\otimes \CH^*(G) \to \CH^*(X\times G)
    \]
    fails to be an isomorphism for a quasi-split group $G$, which could be seen from Table~\ref{tab:main} taking $X:=\Spec K$ with $K$ being the splitting field.

    Having computed $\CH^*_K(G)$ in the sequence~\refbr{eq:intro_conormed} we turn to the study of $\pi_*$. It turns out (Proposition~\ref{prop:norm_zero}) that $\pi_*=0$ for a simple group $G$ except when $G$ is an adjoint group of type $\Delta(G)={}^2\mr{D}_{2r}$. This is proved on a case-by-case basis. The case of $\Delta(G)={}^3\mr{D}_4$ is immediate, since the algebra $\Ch^*(G_K)$ is trivial. The case of $\Delta(G)=\Att_n$ is also straightforward, since $\Ch^*(G_K)$ is rather simple. For $\Delta(G)={}^2\mr{D}_n$ we compute the algebra 
    \[
    \Ch^*(\OGr(n-1;q)) \cong \Ch^*(\OGr(n-1;q)_K)^{\Gal(K/k)}
    \]
    for a quasi-split submaximal isotropic Grassmannian (Proposition~\ref{prop:OGr_Chow_ring}) using the answer for the ring ${\Ch^*(\OGr(n-1;q)_K)}$ given by \cite{Vi07} and \cite[Section~86]{EKM08} and calculating explicitly the action of the Galois group. Since the pullback
    \[
    \Ch^*(\OGr(n-1;q)_K) \to \Ch^*(G_K)
    \]
    is well-known to be surjective, this allows us to show that $\pi_*=0$ when $\Delta(G)={}^2\mr{D}_n$ and $G$ is not an adjoint group of type $\Delta(G)={}^2\mr{D}_{2r}$. In the last case $\Delta(G)={}^2\mr{E}_6$ we have $\Ch^*(G_K)\cong \FF_2[e_1]/(e_1^2)$ with $\deg e_1 =3$ by \cite[Theorem~6]{Kac85}, so we need only to show $\pi_*(e_1)=0$. For this we lift the problem to $G/B$, showing that for an explicit Schubert cycle $Z\in \Ch^3((G/B)_K)$ which goes to $e_1$ on $G_K$ its pushforward $\rho_*(Z)\in \Ch^*(G/B)$ is in the image of the characteristic map. This is done via explicit computer-assisted calculations with Schubert cycles.

    The vanishing of $\pi_*$ in sequence~\refbr{eq:intro_conormed} yields an answer for $\Ch^*(G)$ (Theorem~\ref{thm:answer_at_p}), since we already know $\CH^*_K(G)$. In the last remaining case of the adjoint group of type $\Delta(G)={}^2\mr{D}_{2r}$ a more detailed analysis of $\Ch^*(\OGr(2r-1;q))$ gives an answer for $\Ch^*(G)$ (Theorem~\ref{thm:D2r_ad_answer}), with a crucial ingredient being $c_1(\mc{L}(\bar{\varpi}_{2r-1}+\bar{\varpi}_{2r}))^{2^{v_2(2r)}}=0$ for the first Chern class as in the split case (Proposition~\ref{prop:OGr_Chow_ring}, see also Section~\ref{expos:line_over_group} for the notation).

    \textbf{Case (2)}, $p$ is coprime to $[K:k]$. In this case we show that $\Ch^*(G)=\Ch^*(G_K)^{\Gal(K/k)}$ (Lemma~\ref{lem:pull_inj}). If $\Delta(G)\neq {}^2\mr{E}_6$, then a straightforward description of the $\Gal(K/k)$-action on $\Pic(G_K)$ provides enough information to compute $\Ch^*(G_K)^{\Gal(K/k)}$, while in the case of $\Delta(G)= {}^2\mr{E}_6$ we again resort to explicit computer-assisted computations with Schubert cycles (Theorem~\ref{thm:answer_away_p}).

    \textbf{Case (3)}, $\Delta(G)={}^6\mr{D}_4$ with $p=2$ or $3$ (Theorem~\ref{thm:6D4}). This case is treated via a mixture of the methods used in the first two cases. In order to compute $\Ch^*(G)$ for $p=2$ we partially split the group via a degree $3$ subextension $K/L/k$, obtaining $\Delta(G_L)={}^2\mr{D}_4$. We already know $\Ch^*(G_L)$ by the case (1), and we further analyse which cycles descent to $\Ch^*(G)$ computing the $\Gal(K/k)$-action on $\Ch^*(G_K)$ by computer-assisted calculations with Schubert cycles. The case of $p=3$ is treated similarly, using the degree $2$ subextension $K/L/k$ such that $\Delta(G_L)={}^3\mr{D}_4$ and again computer-assisted calculations with Schubert cycles.

\medskip

    \textbf{Acknowledgements.} The work of the authors is supported by the DFG research grant AN 1545/4-1, the work of the first author is additionally supported by DFG Heisenberg grant AN 1545/1-1 and the work of the second author is additionally supported by DFG research grant SE 1721/4-1 and by the Center for Advanced Studies (LMU, Munich). The authors would like to thank Philippe Gille, Fabien Morel, Andrei Lavrenov, Burt Totaro, Alois Wohlschlager, Maksim Zhykhovich and Egor Zolotarev for valuable discussions.  Besides, we would like to thank an anonymous referee for comments and suggestions.

\medskip

    \textbf{Notation.}
	Throughout the article we employ the following assumptions and notations.
    \begin{itemize} \itemsep0pt
        \item A \textit{variety} over a field $k$ is a reduced scheme separated and of finite type over $k$.
        \item  An \textit{algebraic group} over a field $k$ is an algebraic group in the sense of \cite[Definition~1.1]{Mi17}, i.e. a group object in the category of schemes of finite type over $k$, in particular, an algebraic group is not necessary reduced.
        \item A \textit{group action} is a left group action. At the same time, somehow inconsistently, for a group $G$ acting on $X$ we denote by $X/G$ the quotient.
        
        \item A \textit{simple group} is a geometrically almost-simple group in the sense of \cite[Definition~19.7]{Mi17}.

        \item By {\it the splitting field} of a group $G$ we always mean the minimal splitting field. Note that the minimal splitting field is always a Galois extension.

        \item We use some standard terminology and notation related to quasi-split algebraic groups and algebraic groups in general, and also to vector bundles on (projective) homogeneous varieties, see Appendix~\ref{sec:appendix} for a partial recollection.
    \end{itemize}		 
  
	\begin{tabular}{l|l}
            $v_p$ & $p$-adic valuation\\		
            $k$ & a field \\
            $k^{sep}$ & separable closure of $k$ \\
            $\FF_p$ & finite field with $p$ elements\\
            $\Ch^*(-)$ & $\CH^*(-)\otimes \FF$ for a field $\FF$\\
            $\Smk$ & the category of smooth varieties over $k$\\
		$X_F$ & $X\times_{\Spec k} \Spec F$ for a field extension $F/k$ and $X\in \Smk$\\
            $\mc{X}^*(S)$ & the $\Gal(k^{sep}/k)$-module $\Hom (S_{k^{sep}}, \Gmm)$ of characters of an algebraic group $S$ over $k$\\
            $\pi_1(G)$ & the fundamental group of $G$\\
            $\Delta(G)$ & Dynkin diagram of a quasi-split semisimple group $G$ with the action of $\Gal(k^{sep}/k)$\\ 
            ${}^{\hphantom{2}}\mu_l$ & the algebraic group of $l$-th roots of unity \\
            $\mut{2}{l}$ & a non-trivial (if $l\ge 3$) form of $\mu_l$ split by a quadratic Galois extension \\
            $\mut{2}{2,2}$, $\mut{3}{2,2}$ & a non-trivial form of $\mu_2\times \mu_2$ split by a degree $2$ (degree $3$) Galois extension \\
            $\mut{6}{2,2}$ & a form of $\mu_2\times \mu_2$ with the splitting field $K/k$ such that $\Gal(K/k)\cong S_3$ \\
            $\Smk^T$ & the category of smooth $T$-varieties over $k$ for an algebraic group $T$ over $k$ \\
		$\CH_{K,T}^*(X)$ & equivariant conormed Chow ring of $X\in\Smk^T$ (Definition~\ref{def:eq_conormed}) \\
		$\Phi_\rho, \Phi_T, \Phi_T^c$ & restriction homomorphisms on equivariant conormed Chow rings (Definition~\ref{def:restriction})\\
		$\Res_\rho$ & restriction homomorphism $\Res_\rho \colon \mathrm{Rep}(T)\to \mathrm{Rep}(S)$ on the representation rings for\\
		& a homomorphism $\rho\colon S\to T$ of algebraic groups over a field $k$\\
       
	\end{tabular}
	
    \section{Conormed Chow ring of a quasi-split simple group}
    
    In this part of the article we develop the theory of \textit{conormed Chow rings} and its equivariant version, compute some equivariant conormed Chow rings of a point and compute the conormed Chow rings of quasi-split simple groups. Conormed Chow rings form an oriented cohomology theory, which is especially well adapted to the study of cohomological properties of quasi-split algebraic groups and their homogeneous varieties. It turns out that from the conormed point of view many non-split algebraic groups (e.g. quasi-trivial tori, twisted forms of $\mu_l$ and others) behave very much like their split counterparts.
    
    \subsection{Conormed Chow groups} \label{sec:conormed}
    In this section we recall the definition of conormed Chow rings and explore their basic properties. In particular, we show that for a Galois field extension $K/k$ the conormed Chow group $\CH^*_K(-)$ is generated by irreducible cycles that remain irreducible over $K$.
	
	\begin{definition}[{see also \cite[\S~5]{SZ22} and \cite[\S~2]{Fi19}}] \label{def:conormed}
		Let $K/k$ be a separable extension of fields. For $X\in \Smk$ the \textit{conormed Chow ring of $X$} is
		\[
		\CH^*_{K}(X) := \bigslant{\CH^*(X)}{\Image\left(\bigoplus\limits_{\stackrel{k\subsetneq F\subseteq K}{F/k\,\mathrm{finite}}} \CH^*(X_F)\xrightarrow{(\pi_{F/k})_*} \CH^*(X) \right)},
		\]
		where $\CH^*(X)$ is the Chow ring of $X$, $\pi_{F/k}\colon X_F\to X$ is the projection and $(\pi_{F/k})_*$ is the pushforward homomorphism. By the projection formula the above image is an ideal in $\CH^*(X)$. This ideal is referred to as the \textit{norm ideal}. Note that in the case of $K=k$ one has $\CH^*_{K}(X) = \CH^*(X)$, since there are no nontrivial intermediate field extensions.
		
		One can easily check that since the Chow rings form an oriented cohomology theory on $\Smk$ in the sense of \cite[Definition~1.1.2]{LM07}, the conormed Chow rings $\CH^*_{K}(-)$ for a given separable extension of fields $K/k$ also form an oriented cohomology theory on $\Smk$. In particular, following \cite{Ma68} (see also \cite[Chapter~12]{EKM08}), for a separable extension of fields $K/k$ one can consider the category of conormed Chow motives $\mc{M}_{K}$ built from the category of conormed correspondences, i.e. from the category with the objects given by smooth proper varieties over $k$ and the morphisms between irreducible proper varieties given by $\CH^{\dim Y}_{K}(X\times Y)$. The transformation of oriented cohomology theories $\CH^*(-)\to \CH^*_{K}(-)$ induces a functor $\mc{M} \to\mc{M}_{K}$ from the category of Chow motives to the category of conormed Chow motives.
	\end{definition}
	
	\begin{lemma} \label{lem:normed_coef}
		Let $K/k$ be a separable field extension. Suppose that there exists an intermediate field extension $K\supseteq L\supsetneq k$ with $L/k$ being finite Galois. Then the following holds.
		\begin{itemize}
			\item 
			If $[L:k]$ is not a power of a prime, then $\CH^*_{K}(\Spec k)=0$.
			\item
			If $[L:k]=p^n$ for a prime $p$ and for every intermediate extension $K/F/k$ the degree $[F:k]$ is either infinite or a power of $p$, then $\CH^*_{K}(\Spec k)\cong \FF_p$.
		\end{itemize}
	\end{lemma}	
	\begin{proof}
		It is straightforward to see from the definition that $\CH^*_{K}(\Spec k)\cong \Z/m\Z$ with
		\[
		m:=m(K,k):=\gcd\limits_{\stackrel{k\subsetneq F\subseteq K}{F/k\,\mathrm{finite}}} ([F:k]).
		\]
		Put $\mc{G}:=\Gal(L/k)$. Suppose that $|\mc{G}|=[L:k]=l$ is not a power of a prime. For $p\mid l$ consider a Sylow $p$-subgroup $\mc{G}'_p\le \mc{G}$, then for $F:=L^{\mc{G}'_p}$ one has $[F:k]=l'_p$ for the prime to $p$ part $l'_p$ of $l$. Since the set of all $l_p'$ is coprime, it follows that $m(K,k)=1$ and $\CH^*_{K}(\Spec k)=0$.
		
		If $|\mc{G}|=p^n$ for a prime $p\in \NN$, then there exists a subgroup $\mc{G}'\le \mc{G}$ such that $[\mc{G}:\mc{G}']=p$, thus there is an intermediate field extension $F:=L^{\mc{G}'}$ such that $[F:k]=p$, yielding $m(K,k)\mid p$. On the other hand, it follows from the assumptions that $p\mid m(K,k)$, thus $m(K,k)=p$ and $\CH^*_{K}(\Spec k)\cong \FF_p$.
	\end{proof}
	
	\begin{remark}
		If $K/k$ has no intermediate extensions $k\subsetneq L\subseteq K$ with $L/k$ being finite Galois, then $\CH^*_{K}(\Spec k)$ may fail to be a field, e.g. for a separable extension $K$ with $[K:k]=n$ and without nontrivial intermediate extensions one has $\CH^*_{K}(\Spec k)\cong \Z/n\Z$.
	\end{remark}
	
	\begin{lemma} \label{lem:trivial_normed_corestriction}
		Let $K/F/k$ be separable field extensions with $k\subsetneq F$ being finite. Then for $X\in \SmF$ viewed as a variety over $k$ one has $\CH^*_{K}(X)=0$.
	\end{lemma}
	\begin{proof}
		The projection onto the first factor
		\[
		\Spec F \times_{\Spec k} \Spec F \to \Spec F
		\]
		has a section. Taking the base change along the morphism $X\to \Spec F$ we see that the projection 
		\[
		X_F=X \times_{\Spec k} \Spec F \xrightarrow{\pi_{F/k}} X
		\]
		also has a section. Let $s\colon X\to X_F$ be such section which means that $s$ satisfies $\pi_{F/k}\circ s = \id_X$. Then 
		\[
		1=(\id_X)_*1=(\pi_{F/k})_*(s_*1)\in (\pi_{F/k})_*(\CH^*(X_F))\subseteq \CH^*(X)
		\]
		and the norm ideal coincides with $\CH^*(X)$. The claim follows.
	\end{proof}
	
	\begin{lemma} \label{lem:relative_field}
		Let $X$ be a normal irreducible variety over a field $k$ and $L/k$ be a finite Galois field extension. Suppose that $X_L$ is reducible. Then there exists an intermediate field extension $L\supseteq F\supsetneq k$ such that the structure morphism $X\to \Spec k$ factors as $X\to \Spec F\to \Spec k$.
	\end{lemma}
	\begin{proof}
		Let $k'$ be the separable closure of $k$ in the function field $k(X)$. Since $X$ is normal, it follows that for every open affine subscheme $\Spec A\subseteq X$ we have $k'\subseteq A\subseteq k(X)$, thus $k'\subseteq H^0(X,\mc{O}_X)$ yielding a factorization $X\to \Spec k'\to \Spec k$. 
		
		It suffices to show that there exists a subfield $k' \supseteq F\supsetneq k$ isomorphic to a subfield of $L$. Suppose that no such $F$ exists. We claim that then $k'\otimes_k L$ is a field. Choose embeddings $k'\subseteq k^{sep}$ and $L\subseteq k^{sep}$, and put
		\[
		\mc{G}:=\Gal(k^{sep}/k),\quad \mc{G}_L:=\Gal(k^{sep}/L),\quad \mc{G}_{k'}:=\Gal(k^{sep}/k'), \quad \mc{G}_{Lk'}:=\Gal(k^{sep}/Lk'),
		\]
		where $Lk'$ is the compositum of $L$ and $k'$ inside $k^{sep}$. These groups have the following properties:
		\begin{itemize}
			\item $\mc{G}_L$ is a normal subgroup of $\mc{G}$, since $L/k$ is Galois,
			\item the subgroup $\langle \mc{G}_L,\mc{G}_{k'}  \rangle \le \mc{G}$ generated by $\mc{G}_L$ and $\mc{G}_{k'}$ is $\mc{G}$, since  $k'$ and $L$ have no common subfields except $k$,
			\item $\mc{G}_{Lk'}=\mc{G}_L\cap\mc{G}_{k'}$ since $Lk'$ is the compositum of $L$ and $k'$ inside $k^{sep}$.
		\end{itemize}
		The first two properties yield that $\mc{G} = \mc{G}_{k'}\cdot \mc{G}_L$, thus the monomorphism 
		\[
		\mc{G}_{k'}/(\mc{G}_L\cap\mc{G}_{k'})\to \mc{G}/\mc{G}_L
		\]
		induced by the embedding is an isomorphism. Note that $k'/k$ is an extension of finite degree, being an algebraic subextension of a finitely generated extension $k(X)/k$. Thus we have
		\[
		[Lk':k]=[\mc{G}:\mc{G}_{Lk'}]=[\mc{G}:\mc{G}_L\cap\mc{G}_{k'}]=[\mc{G}:\mc{G}_{k'}]\cdot[\mc{G}_{k'}:\mc{G}_L\cap\mc{G}_{k'}] =[\mc{G}:\mc{G}_{k'}]\cdot[\mc{G}:\mc{G}_L] = [k':k]\cdot [L:k].
		\]
		By the dimension count it follows that the homomorphism $k'\otimes_k L \to Lk'$ is an isomorphism, and $k'\otimes_k L$ is a field. Then $k(X)\otimes_k L\cong k(X)\otimes_{k'} (k' \otimes_k L)$ is also a field, since $k(X)$ is primary over $k'$ (the algebraic closure of $k'$ in $k(X)$ is purely inseparable over $k'$) and $k' \otimes_k L$ is separable over $k'$. Thus $X_L$ is irreducible which contradicts the assumptions.
	\end{proof}
	
	\begin{remark}
		Some normality assumptions on $X$ and $L/k$ in Lemma~\ref{lem:relative_field} are indeed necessary, as the following examples show:
		\begin{itemize}
			\item Let $X:=\Spec \mathbb{R}[x,y]/(x^2+y^2)$ and $L/k = \mathbb{C}/\mathbb{R}$. Then $X$ is an irreducible variety over $\mathbb{R}$ and $X_\mathbb{C}$ is reducible being the union of two lines, but the structure morphism $X\to \Spec \mathbb{R}$ does not factor through $\Spec \mathbb{C}$, since $X$ has a closed point  with the residue field $\mathbb{R}$.
			\item Let $K/k$ be a Galois field extension with the Galois group being $\Gal(K/k)=A_4$, and $L_1,L_2\subseteq K$ be the subfields with $\Gal(K/L_1)=\langle (12)(34) \rangle$ and $\Gal(K/L_2)=\langle (123) \rangle$. Put $X:=\Spec L_1$ and $L:=L_2$. For the compositum $L_1L_2$ we have 
			\[
			\dim_k L_1\otimes_k L_2 = 6 \cdot 4 > 12 = [K:k]\ge [L_1L_2:k],
			\]
			hence $L_1\otimes_k L_2$ is not a field and $X_{L}$ is not connected. Suppose that $X\to \Spec k$ factors through $\Spec F\to \Spec k$ for a subfield $k\subseteq F\subseteq L=L_2$. Then $F$ is isomorphic to a subfield of $L_1$ and there exists $\sigma\in \Gal(K/k)$ such that $F\subseteq \sigma(L_1)$. We have $\Gal(K/\sigma(L_1))\cong C_2$. A pair of order $2$ and order $3$ elements generate $A_4$ yielding $\langle \Gal(K/\sigma(L_1)), \Gal(K/L_2)\rangle = \Gal(K/k)$. It follows that ${\sigma(L_1)\cap L_2=k}$ and $F=k$.
		\end{itemize}
	\end{remark}
	
	\begin{corollary} \label{cor:triv_normed}
		Let $X\in \Smk$ and $K/L/k$ be separable field extensions with $L/k$ being a finite Galois extension. Suppose that $X$ is connected and $X_L$ is not connected. Then $		\CH_{K}^*(X)=0$.
	\end{corollary}
	\begin{proof}
		Follows from Lemma~\ref{lem:trivial_normed_corestriction} and  Lemma~\ref{lem:relative_field}.
	\end{proof}
	
	\begin{remark} \label{rem:conormed_connected}
		Using normalization and an argument similar to the above one can show that for a (not necessary smooth) variety $X$ and a Galois field extension $K/k$ the group
		\[
		\CH_n^{K}(X) := \bigslant{\CH_n(X)}{\Image\left(\bigoplus\limits_{\stackrel{k\subsetneq F\subseteq K}{F/k\,\mathrm{finite}}} \CH_n(X_F)\xrightarrow{(\pi_{F/k})_*} \CH_n(X) \right)},
		\]
		is generated by irreducible cycles that remain to be irreducible over $K$.
	\end{remark}
	
	\subsection{Equivariant conormed Chow groups} \label{sec:conormed_equiv}

        In this section we use the approach by Totaro \cite{To99} and Edidin-Graham \cite{EG98} to extend the conormed Chow groups to the equivariant setting, obtaining a $T$-equivariant ring cohomology theory $\CH^*_{K,T}(-)$.
	
	\begin{definition} \label{def:eq_conormed}
		Let $T$ be an affine algebraic group over $k$ and $K/k$ be a separable field extension. Following \cite{EG98,To99} we introduce the \textit{$T$-equivariant conormed Chow ring} $\CH^*_{K,T}(X)$ of $X\in\Smk^T$ as follows. Let $n\in \NN_0$, let $V$ be a representation of $T$ and let $U\subseteq V$ be an open subset such that
		\begin{enumerate}
			\item 
			$V\setminus U$ is of codimension greater than $n$ in $V$,
			\item
			$T$ acts freely on $U$,
			\item 
			the quotient $U/T$ exists as a scheme.
		\end{enumerate}
		Such $V$ and $U\subseteq V$ always exist, see e.g. \cite[Lemma~9]{EG98}. We refer to $U$ as an \textit{$n$-th approximation to the universal torsor $ET$} and usually denote such $U$ as $E_nT$. We put 
		\[
		\CH^n_{K,T}(X):= \CH^n_{K}((X \times E_nT)/T),
		\]
		where $(X \times E_nT)/T$ is considered as an algebraic space \cite[Proposition~22]{EG98} and the conormed Chow groups on the right are the Chow groups of an algebraic space as in \cite[Section~6]{EG98} modulo the respective norm ideal. If the quotient $(X \times E_nT)/T$ is represented by a smooth variety, e.g. if $X$ and $T$ satisfy the assumptions of \cite[Proposition~23]{EG98}, then these are the conormed Chow groups of a variety introduced in Definition~\ref{def:conormed}. It is straightforward to check (see \cite[Definition-Proposition~1]{EG98}) that, up to a canonical isomorphism, the group $\CH^n_{K,T}(X)$ does not depend on the choice of $V$ and $U=E_nT$. Moreover, one clearly has
		\[
		\CH^*_{K,T}(X) = \bigslant{\CH^*_T(X)}{\Image\left(\bigoplus\limits_{\stackrel{k\subsetneq F\subseteq K}{F/k\,\mathrm{finite}}} \CH^*_T(X_F)\xrightarrow{(\pi_{F/k})_*} \CH^*_T(X) \right)},
		\]
		where $\CH^*_T(-)$ are the equivariant Chow groups of \cite{EG98}. This gives rise to an oriented $T$-equivariant ring cohomology theory on $\Smk^T$, in particular, $\CH^*_{K,T}(X)$ is a ring for $X\in \Smk^T$ and one has
		\begin{itemize}
			\item
			\textit{pullback} ring homomorphisms
			\[
			f^*\colon \CH^*_{K,T}(X) \to \CH^*_{K,T}(Y)
			\] 
			for a $T$-equivariant morphism $f\colon Y\to X$, $X,Y\in\Smk^T$,
			\item
			\textit{pushforward} homomorphism of $\CH^*_{K,T}(X)$-modules
			\[
			f_*\colon \CH^{*-n}_{K,T}(Y) \to \CH^*_{K,T}(X)
			\]
			for a projective $T$-equivariant morphism $Y\to X$ of codimension $n$ with $X,Y\in\Smk^T$,
			\item
			an exact \textit{localization sequence}
			\[
			\CH^{*-n}_{K,T}(Z) \xrightarrow{i_*} \CH^*_{K,T}(X) \xrightarrow{j^*} \CH^*_{K,T}(X-Z) \to 0
			\]
			for a closed codimension $n$ embedding of smooth $T$-varieties $Z\to X$,
			\item
			\textit{homotopy invariance} isomorphism 
			\[
			f^*\colon \CH^*_{K,T}(X) \xrightarrow{\simeq} \CH^*_{K,T}(\mc{V})
			\]
			for a  $T$-equivariant vector bundle $f\colon \mc{V}\to X$,
			\item
			\textit{Chern class} $c_n(\mc{V})\in \CH^n_{K,T}(X)$ for a $T$-equivariant vector bundle $\mc{V}$ over $X\in\Smk^T$ and $n\in \NN$, in particular, $c_n(V)\in  \CH^n_{K,T}(\Spec k)$ for a representation $V$ of $T$ and $n\in \NN$,
			\item
			\textit{normalization}
			\[
			c_n(\mc{V})=z^*z_*(1)\in \CH^n_{K,T}(X)
			\]
			for the zero section $z\colon X\to \mc{V}$ of a rank $n$ $T$-equivariant vector bundle $\mc{V}\to X$.
		\end{itemize}
		Combining the localization sequence for the zero section $z\colon X\to \mc{V}$ of a rank $n$ $T$-equivariant vector bundle $\mc{V}\to X$, the homotopy invariance isomorphism and the normalization property one obtains the \textit{Gysin exact sequence}
		\[
		\CH^{*-n}_{K,T}(X) \xrightarrow{c_n(\mc{V})} \CH^*_{K,T}(X) \xrightarrow{j^*} \CH^*_{K,T}(\mc{V}-z(X)) \to 0.
		\]
	\end{definition}
	
	\begin{lemma} \label{lem:trivial_normed}
		Let $T$ be an affine connected algebraic group over $k$, $X\in \Smk^T$ and let $K/L/k$ be separable extensions of fields with $L/k$ being a finite Galois extension. Suppose that $X$ is connected and $X_L$ is not connected. Then
		\[
		\CH^*_{K,T}(X)=0.
		\]
	\end{lemma}
	\begin{proof}
		Let $L\supseteq F\supsetneq k$ and $X\to \Spec F\to \Spec k$ be a factorization of the structure morphism given by Lemma~\ref{lem:relative_field}. The factorization arises from an embedding $F\subseteq H^0(X,\struct_X)$, thus the projection $X\to \Spec F$ is $T$-equivariant with the $T$-action on $\Spec F$ induced by the one on $H^0(X,\struct_X)$. Since the automorphism group of $F$ over $k$ is finite and $T$ is connected, it follows that the $T$-action on $\Spec F$ is trivial. It follows that the splitting of the projection $X_F\to X$ constructed in the proof of Lemma~\ref{lem:trivial_normed_corestriction} is $T$-equivariant, so the claim follows as in loc. cit.	
	\end{proof}
	
	\begin{definition} \label{def:restriction}
		Let $\rho\colon S\to T$ be a homomorphism of affine algebraic groups over a field $k$. For $n\in\NN_0$ choose $n$-th approximations $E_nS$ and $E_nT$ to the universal torsors $ES$ and $ET$ and a morphism $f\colon E_nS\to E_nT$ equivariant with respect to $\rho$, i.e. such that the diagram
		\[
		\xymatrix{
			S\times E_nS \ar[r] \ar[d]_{\rho\times f} & E_nS \ar[d]^{f}\\
			T\times E_nT \ar[r]  & E_nT,
		}
		\]
		commutes. Here the horizontal maps are the respective actions. Then for $X\in \Smk^T$ we have a morphism
		\[
		\bar{f}\colon (X\times E_nS)/S \to (X\times E_nT)/T
		\]
		induced by $\id_X\times f$ and for a separable field extension $K/k$ we obtain a pullback homomorphism
		\[
		\bar{f}^*\colon \CH_{K}^n((X\times E_nT)/T) \to \CH_{K}^n((X\times E_nS)/S).
		\]
		As in \cite[Definition-Proposition~1]{EG98}, one can show that this homomorphism does not depend on the choices of approximations and on $f$ yielding a \textit{restriction homomorphism}
		\[
		\Phi_{\rho} \colon \CH_{K,T}^n (X) \to \CH_{K,S}^n (X)
		\]
		with the action of $S$ on $X$ given by the composition of $\rho$ and the action of $T$. If $S=1$ is the trivial group, then $\CH^*_{K,S} (-) \cong \CH^*_{K} (-)$ and we put
		\[
		\Phi_T:=\Phi_{\rho} \colon \CH_{K,T}^n (X) \to  \CH^n_{K,S} (X) \cong \CH_{K}^n (X).
		\]
		If $T=1$ is the trivial group, then $\CH^*_{K,T} (-) \cong \CH^*_{K} (-)$ and we put
		\[
		\Phi^c_S:=\Phi_{\rho} \colon \CH_{K}^n (X) \cong \CH_{K,T}^n (X) \to  \CH^n_{K,S} (X),
		\]
		where $X\in\Smk$ is equipped with the trivial action of $S$.
	\end{definition}
	
	\begin{definition} \label{def:Eilenberg--Moore_homo}
		Let $K/k$ be a separable field extension.
		
		Let $\rho\colon S\to T$ be a homomorphism of affine algebraic groups over a field $k$ and $X\in\Smk^{T}$. For $n\in \NN_0$ choose $n$-th approximations $E_nS$ and $E_nT$ to the universal torsors $ES$ and $ET$ and consider the following cartesian square induced by the projections:
		\[
		\xymatrix{
			(X\times E_nS \times E_nT)/S \ar[r] \ar[d] & (E_nS \times E_nT)/S \ar[d]\\
			(X\times E_nT)/T \ar[r] & E_nT/T
		}
		\]
		Here the actions of $S$ and $T$ on the respective products are the diagonal ones, and the action of $S$ on $E_nT$ is induced by the homomorphism $\rho$. Note that $E_nS \times E_nT$ with the diagonal action of $S$ may also be considered as an $n$-th approximation to the universal torsor $ES$. Thus this gives rise to a homomorphism
		\begin{equation} \label{eq:EM1}
			\CH^*_{K,T} (X) \otimes_{\CH^*_{K,T}(\Spec k)}  \CH^*_{K,S} (\Spec k) \to \CH^*_{K,S} (X),\quad (x,\alpha)\mapsto \alpha \cdot \Phi_{\rho}(x),
		\end{equation}
		which we refer to as \textit{Eilenberg--Moore homomorphism}. One can check that it does not depend on the choice of $E_nS$ and $E_nT$. In particular, for $S=1$ being the trivial group the Eilenberg--Moore homomorphism looks as
		\begin{equation} \label{eq:EM2}
			\CH^*_{K,T} (X) \otimes_{\CH^*_{K,T}(\Spec k)}  \CH^*_{K} (\Spec k) \to \CH^*_{K} (X),\quad (x,\alpha)\mapsto \alpha \cdot \Phi_{T}(x).
		\end{equation}
		
		Let $S\xrightarrow{\rho} T\xrightarrow{\rho'} R$ be a short exact sequence of affine algebraic groups over a field $k$ and $X\in\Smk^{T}$.  For $n\in \NN_0$ choose $n$-th approximations $E_nT$ and $E_nR$ to the universal torsors $ET$ and $ER$ and similarly to the above consider the following cartesian squares:
		\[
		\xymatrix{
			(X\times E_nT \times E_nR)/S \ar[r] \ar[d] & (E_nT \times E_nR)/S \ar[d] \ar[r] & \ar[d] E_nR \\
			(X\times E_nT \times E_nR)/T \ar[r] & (E_nT\times E_nR)/T \ar[r] & E_nR/R
		}
		\]
		Here the action of $S$ on $E_nT$ is restricted from $T$, the action of $S$ on $E_nR$ is trivial and  the action of $T$ on $E_nR$ is induced by the homomorphism $\rho'$. With the specified actions $E_nT\times E_nR$ can be viewed as $n$-th approximations both to the universal torsors $ES$ and $ET$. Thus we obtain the following version of the Eilenberg--Moore homomorphism:
		\begin{equation} \label{eq:EM3}
			\CH^*_{K,T} (X) \otimes_{\CH^*_{K,R}(\Spec k)}  \CH^*_{K} (\Spec k) \to \CH^*_{K,S} (X),\quad (x,\alpha)\mapsto \alpha \cdot \Phi_{\rho}(x),
		\end{equation}
		with the structure of $\CH^*_{K,R}(\Spec k)$-module on $\CH^*_{K,T} (X)$ induced by the restriction homomorphism $\Phi_{\rho'}\colon\CH^*_{K,R}(\Spec k)\to \CH^*_{K,T}(\Spec k)$.
	\end{definition}
	
	\begin{remark}
		The name for the homomorphisms in Definition~\ref{def:Eilenberg--Moore_homo} refers to the classical Eilenberg--Moore spectral sequence which computes the singular cohomology groups of a pullback over a fibration.
	\end{remark}
	
	\begin{lemma}[{see \cite[Proposition~8(a)]{EG98}}] \label{lem:torsor}
		Let $K/k$ be a separable field extension, let $T$ be an affine algebraic group over $k$, $X\in \Smk^{T}$ and suppose that the action of $T$ on $X$ is free and $X/T$ is represented by a scheme. Then there is a canonical isomorphism
		\[
		\CH^*_{K, T}(X) \cong \CH^*_{K}(X/T).
		\]
	\end{lemma}
	\begin{proof}
		The same reasoning as in the proof of \cite[Proposition~8(a)]{EG98} applies, but since we will make computations using this isomorphism later in the article, we provide the construction for a completeness.
		
		Let $n\in \NN_0$ and let $E_nT\subseteq V$ be an $n$-th approximation to the universal torsor $ET$. Since $X\to X/T$ a $T$-torsor, it follows that 
		\[
		(X\times V)/T \to  X/T
		\]
		is a vector bundle with 
		\[
		(X\times E_nT)/T \subseteq (X\times V)/T
		\]
		being an open subset with the complement of codimension greater then $n$. Thus 
		\[
		\CH^n_{K, T}(X)= \CH^n_{K}((X\times E_nT)/T)
		\cong \CH^n_{K}((X\times V )/T)\cong \CH^n_{K}(X/T).
		\]
	\end{proof}
	
	\begin{lemma} \label{lem:trivial_coprime}
		Let $K/k$ be a separable field extension, let $N$ be a finite group scheme over $k$ and let $X\in\Smk$ be equipped with the trivial action of $N$. Suppose that the order of $N$ is invertible in $\CH^*_{K}(\Spec k)$. Then the restriction homomorphism
		\[
		\Phi_N\colon \CH^*_{K,N}(X)\to \CH^*_{K}(X)
		\]
		is an isomorphism.
	\end{lemma}
	\begin{proof} 
		For $n\in \NN_0$ let $E_nN$ be an $n$-th approximation to the universal torsor $EN$. Then 
		\[
		f\colon X\times E_nN\to (X\times E_nN)/N \cong X\times (E_nN)/N
		\]
		is a finite morphism of degree equal to the order of $N$ and the composition 
		\[
		f_*\circ f^*\colon \CH^n_{K}((X\times E_nN)/N) \to \CH^n_{K}(X\times E_nN) \to \CH^n_{K}((X\times E_nN)/N),
		\]
		is the multiplication by the order of $N$ \cite[Example~1.7.4]{Fu98}. It follows that $f_*\circ f^*$ an isomorphism. Recall that
		\[
		\CH^n_{K,N}(X)=\CH^n_{K}((X\times E_nN)/N),\quad \CH^n_{K}(X)=\CH^n_{K}(X\times E_nN),
		\]
		and under this identifications we have $\Phi_{N}=f^*$. Thus the composition
		\[
		\CH^n_{K,N}(X) \xrightarrow{\Phi_{N}} \CH^n_{K}(X) \xrightarrow{f_*} \CH^n_{K,N}(X)
		\]
		is an isomorphism. On the other hand, since $X$ has a trivial action of $N$, we have a homomorphism 
        \[
        \Phi_N^c\colon \CH^n_{K}(X) \to \CH^n_{K,N}(X)
        \]
        induced by the projection $g\colon (X\times E_nN)/N \to X$, and the composition
		\[
		\CH^n_{K}(X) \xrightarrow{\Phi_N^c} \CH^n_{K,N}(X) \xrightarrow{\Phi_N} \CH^n_{K}(X) 
		\]	
		is also an isomorphism. The claim follows.
	\end{proof}

	\subsection{Eilenberg--Moore isomorphism for quasi-trivial tori}

        In this section we investigate the relations between equivariant and non-equivariant conormed Chow groups for quasi-trivial tori, with the main results being the computation of the equivariant conormed Chow ring of the point and a recipe for recovering the non-equivariant conormed Chow groups from the equivariant ones. Here and below we use some standard notation and well-known facts about quasi-trivial tori, see Section~\ref{sec:multiplicative_groups} for a recollection.

	\begin{proposition} \label{prop:modified_Gysin}
		Let $K/L/k$ be separable extensions of fields with $L/k$ being finite, let $T$ be a connected algebraic group over $k$, let $R:=R_{L/k} \Gmm$ be the Weil restriction of the one-dimensional split torus and let $\rho\colon T\to R$ be a homomorphism. Suppose that $R_K$ is split. Then for $X\in \Smk^T$ the sequence
		\[
		\CH^{*-n}_{K,T}(X)\xrightarrow{c_n(V_R)}\CH^*_{K,T}(X)\xrightarrow{j^*}\CH^*_{K,T}(X\times R)\to 0
		\]
		is exact. Here
		\begin{itemize}
			\item
			$n=[L:k]$,
			\item 
			the first homomorphism is given by the multiplication with the Chern class $c_n(V_R)$ where the standard vector representation $V_R$ of $R$ is viewed as a representation of $T$ via $\rho$,
			\item
			$T$ acts on $X\times R$ diagonally,
			\item 
			$j\colon X\times R\to X$ is the projection.
		\end{itemize}
	\end{proposition}
	\begin{proof}
		We have a canonical open embedding $R\subseteq V_R\setminus \{0\} \subseteqq V_R$ (see~\refbr{expos:tori}) and the Gysin exact sequence for the vector bundle $X\times V_R\to X$ implies that it is sufficient to check that the pullback homomorphism
		\[
		\CH^*_{K,T}(X\times (V_R\setminus \{0\}))\xrightarrow{} \CH^*_{K,T}(X\times R)
		\]
		is an isomorphism. Below we will omit $X$ from the notation.
		
		Since $R_K$ is split by the assumption, it follows that $K$ contains the Galois closure $\bar{L}/k$ of $L/k$ which is the splitting field of $R$. Choose a basis $\{e_1,e_2,\hdots, e_n\}$ for $L$ over $k$ and fix the corresponding isomorphism $V_R\cong \AAA^n_k=\Spec k[x_1,x_2,\hdots, x_n]$. Denote by
		\[
		\mc{H}=Z(x_1e_1+x_2e_2+\hdots + x_ne_n)\leq \AAA^n_{\bar{L}} \cong (V_R)_{\bar{L}}
		\]
		the corresponding hyperplane. Put $\mathcal{G}:=\Gal(\bar{L}/k)$, $\mathcal{G}':=\Gal(\bar{L}/L)$, $\mathcal{S}:=\mathcal{G}/\mathcal{G}'$ and consider the following filtration of $\AAA^n_{\bar{L}} \cong (V_R)_{\bar{L}}$ by closed subsets.
		\[
		\{0\}=\bar{V}_n\leq \bar{V}_{n-1} \leq \hdots \leq \bar{V}_0=\AAA^n_{\bar{L}}\cong (V_R)_{\bar{L}}, \quad \bar{V}_i:=\bigcup_{\stackrel{I\subseteq \mathcal{S}}{|I|=i}} \mc{H}_I, \quad \mc{H}_I:=\bigcap_{\sigma\in I} \sigma(\mc{H}), \quad \mc{H}_\emptyset := \AAA^n_{\bar{L}}.
		\]
		Note that $\mathcal{H}$ is stable under the action of $\mathcal{G}'$, so we may define $\sigma(\mc{H})$ for $\sigma\in \mathcal{S}$ using an arbitrary representative of the corresponding coset. Since for every $i$ the variety $\bar{V}_i$ is clearly $\mc{G}$-stable, it follows that there exist closed subvarieties 
		\[
		\{0\}={V}_n\leq {V}_{n-1} \leq \hdots \leq {V}_0=\AAA^n_k\cong V_R
		\]
		such that $(V_i)_{\bar{L}}= \bar{V}_i$.	For $0\le i\le n-1$ we have
		\[
		\bar{V}_i\setminus \bar{V}_{i+1} = \bigsqcup_{\stackrel{I\subseteq \mathcal{S}}{|I|=i}} \mc{H}_I^o =   \bigsqcup_{\mc{A}\in Orb_\mathcal{G} {\mathcal{S}\choose {i}}} \left(	\bigsqcup_{I \in \mc{A}} 
		\mc{H}_I^o\right), \quad \mc{H}_I^o := \mc{H}_I \setminus \bigcup_{\tau \in \mathcal{S}\setminus I} \left( \mc{H}_I \cap \tau(\mc{H}) \right).
		\]
		Here $Orb_\mathcal{G} {\mathcal{S}\choose {i}}$ is the set of $\mathcal{G}$-orbits under the natural action of $\mathcal{G}$ on the set $\left\{I\subseteq \mathcal{S}\, |\, |I|=i\right\}$ induced by the action of $\mathcal{G}$ on $\mathcal{G}/\mathcal{G}'= \mathcal{S}$. For $0\le i\le n-1$ and $\mc{A}\in Orb_\mathcal{G} {\mathcal{S}\choose {i}}$ put
		\[
		\overline{W}_\mc{A}:=	\bigsqcup_{I \in \mc{A}} 
		\mc{H}_I^o.
		\]
		Note that every $\overline{W}_{\mc{A}}$ is $\mathcal{G}$-stable and smooth being a disjoint union of open subsets of different $\mc{H}_I\cong \AAA^{n-i}_{\bar{L}}$, hence it corresponds to some smooth locally closed subvariety $W_\mc{A}\subseteq \AAA^n_k\cong V_R$ satisfying $(W_\mc{A})_{\bar{L}}=\overline{W}_\mc{A}$. Since $\mathcal{G}$ acts transitively on $\mathcal{A}$, it follows that $W_\mc{A}$ is connected, and since $|\mathcal{A}|>1$ for $1\le i\le n-1$, it follows that $(W_\mc{A})_{\bar{L}}=\overline{W}_\mc{A}$ is not connected for such $i$. 
		Applying Lemma~\ref{lem:trivial_normed} to $W_{\mc{A}}$ and $K/\bar{L}/k$ we get
		\[
		\CH^*_{K,T}(W_{\mc{A}})=0
		\]
		for $1\le i\le n-1$. Successively applying localization sequences
		\[
		\CH^{*-i}_{K,T}(\bigsqcup_{\mc{A}\in Orb_\mathcal{G} {\mathcal{S}\choose {i}}} W_\mc{A}) \xrightarrow{} \CH^*_{K,T}(V_R\setminus V_{i+1}) \xrightarrow{} \CH^*_{K,T}(V_R\setminus V_{i}) \to 0
		\]
		for $i=n-1,\hdots,1$ we obtain isomorphisms
		\[
		\CH^*_{K,T}(V_R\setminus V_{n}) \xrightarrow{\simeq} \CH^*_{K,T}(V_R\setminus V_{n-1})  \xrightarrow{\simeq} \hdots  \xrightarrow{\simeq} 
		\CH^*_{K,T}(V_R\setminus V_{1}).
		\]
		Recall that $V_R\setminus V_{n}=V_R\setminus \{0\}$ while recollection~\refbr{expos:tori} yields $V_R\setminus V_{1}=R$. The claim follows.
	\end{proof}
	
	\begin{proposition}[{see also \cite[\S~3]{Ka12}}] \label{prop:classifying_quasi-trivial}
		Let $K/k$ be a separable field extension and let $T\cong R_{L_1/k} \Gmm \times \hdots \times R_{L_r/k} \Gmm$ be a quasi-trivial torus over $k$. Suppose that $T_K$ is split. Then for $X\in \Smk$ equipped with the trivial action of $T$ one has an isomorphism
		\[
		\CH^*_{K,T}(X) \cong \CH^*_{K}(X)[b_1,b_2,\dots,b_r],\quad b_i\mapsto c_{n_i}(V_i),\quad n_i :=[L_i:k],
		\]
        where $V_i:=V_{R_i}$ is the standard vector representation of $R_i:=R_{L_i/k} \Gmm$ viewed as a representation of $T$ via the projection $T\to R_i$.
	\end{proposition}
	\begin{proof}
		The Weil restrictions $R_{L_i/k}\PP^n_{L_i}$ are models for the $n$-th approximations $(E_n R_i) /R_i$ to the classifying spaces $(ER_i) /R_i$ by \cite[\S~3]{Ka12} (see also \cite[Lemma~3.1]{KM22}), in particular, one has
		\[
		\CH^*_{K,T}(X)\cong \varprojlim_n \CH^*_{K}(X\times R_{L_1/k}\PP^n_{L_1}\times\hdots \times R_{L_r/k}\PP^n_{L_r}).
		\]
		It follows from \cite[Proposition~5.6]{Ka00} that the (integral) Chow motive $\mc{M}(R_{L_i/k}\PP^n_{L_i})$ splits as a direct sum of motives $\mc{M}(\Spec F,l)$ for some $l\in \NN_0$ and intermediate fields $\bar{L}_i/F/k$ with $\bar{L}_i/k$ being a Galois closure of $L_i/k$. Since $T_K$ is split and $\bar{L}_i$ is the splitting field of $R_i$, it follows that $\bar{L}_i/k$ embeds into $K/k$. Lemma~\ref{lem:trivial_normed_corestriction}  yields that in the conormed motives one has $\mc{M}_{K}(\Spec F,l) =0$ for $K\supseteq F\supsetneq k$, thus 
		\[
		\mc{M}_{K}(R_{L_i/k}\PP^n_{L_i})\cong  \mc{M}_{K}(\Spec k, l_0) \oplus \hdots \oplus \mc{M}_{K}(\Spec k, l_d)
		\]
		for some integers $l_0,\hdots,l_d$. Then the K\"unneth formula
		\[
		\CH^*_{K}(X\times R_{L_1/k}\PP^n_{L_1}\times\hdots \times R_{L_r/k}\PP^n_{L_r}) 
		\cong \CH^*_{K}(X)\otimes \CH^*_{K}(R_{L_1/k}\PP^n_{L_1})\otimes\hdots \otimes \CH^*_{K}(R_{L_r/k}\PP^n_{L_r})
		\]
		holds and it suffices to show that for every $i$ there is an isomorphism
		\[
		\CH^*_{K,R_i}(\Spec k)\cong  \CH^*_{K}(\Spec k) [b_i],\quad b_i\mapsto c_{n_i}(V_i).
		\]
		If $L_i=k$, then
		\[
		\CH^*_{K,R_i}(\Spec k) = \CH^*_{K,\Gmm}(\Spec k) \cong \CH^*_{K}(\Spec k)[x],\quad x\mapsto c_1(V_{\Gmm}),
		\]
		by the usual projective bundle theorem.	So we can assume that $L_i\supsetneq k$. Proposition~\ref{prop:modified_Gysin} together with the isomorphism $\CH^*_{K,R_i}(R_i)\cong \CH^*_{K}(\Spec k)$ of Lemma~\ref{lem:torsor} yields that $\CH^{>0}_{K,R_i}(\Spec k)$ is generated by $c_{n_i}(V_i)$. Since $\CH^0_{K,R_i}(\Spec k)=\CH^0_{K}(\Spec k)$ is either zero or a field by Lemma~\ref{lem:normed_coef}, it follows that $b_i\mapsto c_{n_i}(V_i)$ induces an isomorphism $\CH^*_{K,R_i}(\Spec k)\cong\CH^*_{K}(\Spec k)[b_i]$ or $\CH^*_{K,R_i}(\Spec k)\cong \CH^*_{K}(\Spec k)[b_i]/(b_i^m)$ for some $m\in \NN$. In order to see that $c_{n_i}(V_i)$ is not nilpotent consider the homomorphism $\rho\colon \Gmm\to R_i$ given by the unit of the adjunction between the extension of scalars and the Weil restriction. This gives rise to a homomorphism 
		\[
		\Phi_\rho\colon \CH^*_{K,R_i}(\Spec k) \to \CH^*_{K,\Gmm}(\Spec k) \cong \CH^*_{K}(\Spec k)[x].
		\]
		The representation $V_{i}$ of $R_i$ restricted to $\Gmm$ is isomorphic to a direct sum of $n_i$ copies of the standard linear representation of $\Gmm$. Thus 
		\[
		\Phi_\rho(c_{n_i}(V_i))= c_{n_i}(V_{\Gmm}^{\oplus n_i})=c_1(\struct(-1))^{n_i},
		\]
		so $\Phi_\rho(c_{n_i}(V_i))$ is not nilpotent and it follows that $c_{n_i}(V_i)$ is not nilpotent as well.
	\end{proof}
	
	\begin{corollary} \label{cor:EM_qt}
		Let $K/k$ be a separable field extension and let $T$ be a quasi-trivial torus over a field $k$. Suppose that $T_K$ is split. Then for $X\in \Smk^T$ the Eilenberg--Moore homomorphism~\refbr{eq:EM2}
		\[
		\CH_{K,T}^*(X) \otimes_{\CH_{K,T}^*(\Spec k)} \CH_{K}^*(\Spec k)\to \CH_{K}^*(X)
		\]
		is an isomorphism.
	\end{corollary}
	\begin{proof}
		Choose an isomorphism $T\cong R_{L_1/k} \Gmm \times \hdots \times R_{L_r/k} \Gmm$ with $L_i/k$ being finite separable field extensions. Put $n_i:=[L_i:k]$ and $R_i:= R_{L_i/k} \Gmm$. Inductively applying Proposition~\ref{prop:modified_Gysin} we obtain
		\[
		\CH^*_{K,T}(X\times T) \cong \CH^*_{K,T}(X)/(c_{n_1}(V_{R_1}), \hdots, c_{n_r}(V_{R_r}))
		\]
		where $V_{R_i}$ is viewed as a representation of $T$ via the projection $T\to R_i$.	Lemma~\ref{lem:torsor} yields an isomorphism $\CH^*_{K,T}(X\times T)\cong 	\CH^*_{K}(X)$ while Proposition~\ref{prop:classifying_quasi-trivial} yields an isomorphism  $\CH_{K,T}^*(\Spec k)\cong  \CH_{K}^*(\Spec k)[b_1, \hdots, b_r]$, $b_i\mapsto c_{n_i}(V_{R_i})$, so the claim follows.
	\end{proof}
	
	\begin{remark}
		A similar result was obtained in \cite[Theorem~3.1]{Sal22} claiming $\CH^*(X)$ to be the quotient of $\CH_{T}^*(X)$ modulo the pushforwards of characteristic classes defined over field extensions $F/k$.
	\end{remark}
	
	\begin{definition} \label{def:Kunneth_formula_BN}
		Let $N$ be an affine algebraic group over a field $k$ and let $K/k$ be a separable field extension. We say that $BN$ \textit{satisfies the K\"unneth formula for $K/k$}, if for every $X\in\Smk$ equipped with the trivial action of $N$ the K\"unneth homomorphism
		\[
		\CH_{K}^*(X) \otimes \CH_{K,N}^*(\Spec k) \xrightarrow{} \CH_{K,N}^*(X),\quad (y,\alpha) \mapsto \alpha\cdot \Phi_{N}^c(y),
		\]
		is an isomorphism.
	\end{definition}
	
	\begin{proposition} \label{prop:EM_qt_extended}
		Let $K/k$ be a separable field extension, let $T$ be a quasi-trivial torus over $k$, let $N\le T$ be a subgroup and $X\in\Smk^T$. Suppose that $T_K$ is split and that $BN$ satisfies the K\"unneth formula for $K/k$. Then the homomorphism
		\begin{gather*}
			(\CH_{K,T}^*(X) \otimes \CH_{K,N}^*(\Spec k)) \otimes_{\CH_{K,T}^* (\Spec k)} \CH_{K}^* (\Spec k) \xrightarrow{} \CH_{K,N}^* (X),\\
			(x,\alpha,\beta) \mapsto \alpha\cdot\beta\cdot \Phi_{\rho} (x),
		\end{gather*}
		is an isomorphism. Here $\rho\colon N\to T$ is the embedding and the structure of a $\CH_{K,T}^* (\Spec k)$-module on $\CH_{K,T}^*(X) \otimes \CH_{K,N}^*(\Spec k)$ arises from the ring homomorphism
		\begin{gather*}
			\CH_{K,T}^* (\Spec k)\xrightarrow{\Phi_{\nu}} \CH_{K,T\times N}^* (\Spec k) \cong \CH_{K,T}^*(\Spec k) \otimes \CH_{K,N}^*(\Spec k)
		\end{gather*}
		for the homomorphism $\nu \colon T\times N\to T$, $\nu(t,s)=t\cdot s^{-1}$, composed with the componentwise module structure.
	\end{proposition}
	\begin{proof}
		Let $n\in \NN_0$ and let $E_nT$ and $E_nN$ be $n$-th approximations to the universal torsors $ET$ and $EN$ respectively and equip $X\times E_nT \times E_nN$ with an action of $T\times N$ given by $(t,s)\cdot (x,u,v) := (tsx,tsu,sv)$. Let $\tilde{X}:=(X\times E_nT \times E_nN)/N$ and equip it with an action of $T$ compatible with the projection ${X\times E_nT \times E_nN\to \tilde{X}}$. Corollary~\ref{cor:EM_qt} applied to $\tilde{X}$ yields an isomorphism
		\[
		\CH_{K,T}^*(\tilde{X}) \otimes_{\CH_{K,T}^*(\Spec k)} \CH_{K}^*(\Spec k)\xrightarrow{\simeq} \CH_{K}^*(\tilde{X}).
		\]
		Since $E_nT$ and $E_nN$ are $n$-th approximations to the universal torsors $ET$ and $EN$, it follows that ${E_nT\times E_nN}$ with the diagonal action of $N$ is also an $n$-th approximation to the universal torsor $EN$, and $\CH_{K}^n(\tilde{X})\cong \CH_{K,N}^n(X)$. We have isomorphisms
		\[
		\tilde{X}/T \cong (X\times E_nT \times E_nN) / (T\times N) \xrightarrow{\simeq} (X\times E_n T)/T \times E_nN/N
		\]
		with the second one induced by the respective projections. Applying further Lemma~\ref{lem:torsor} we obtain isomorphisms
		\begin{equation} \label{eq:isos1}
			\CH_{K,T}^*(\tilde{X}) \cong \CH_{K}^*(\tilde{X}/T) \cong \CH_{K}^*((X\times E_n T)/T \times E_nN/N)
			\cong \CH_{K,T}^*(X)\otimes \CH_{K,N}^*(\Spec k)
		\end{equation}
		in degrees up to $n$. Combining all the above we obtain the desired isomorphism
		\begin{equation}\label{eq:isos2}
			(\CH_{K,T}^*(X) \otimes \CH_{K,N}^*(\Spec k)) \otimes_{\CH_{K,T}^* (\Spec k)} \CH_{K}^* (\Spec k) \xrightarrow{\simeq} \CH_{K,N}^* (X),
		\end{equation}
		It is straightforward to check that the isomorphism is given by $(x,\alpha,\beta) \mapsto \alpha\cdot\beta\cdot \Phi_{\rho} (x)$ as claimed. In order to identify the module structure consider the following diagram.
		\[
		\xymatrix{
			E_nT/T  &  (E_nT \times E_nN \times E_nT)/(T\times N) \ar[l]_(0.65){q_3} \ar[d]^{q_{12}}\\
			(X\times E_nT \times E_nN \times E_nT)/(T\times N) \ar[u]^{p_4} \ar[d]^{(p_{12},p_3)} \ar[ru]^{p_{234}} &  (E_nT \times E_nN)/(T\times N) \ar[d]_{\cong}^{(s_1,s_2)}\\
			(X\times E_nT)/T \times E_nN/N \ar[r]^{(r_{2},\id)} &  E_nT/T \times E_nN/N
		}
		\]
		Here the action in the middle of the left column is given by $(t,s)\cdot (x,u,v,w):= (tsx,tsu,sv,tw)$, in the bottom-left corner $T$ acts diagonally on $X\times E_nT$ and $N$ acts on $E_nN$ in the given way, the action in the top-right corner is given by $(t,s)\cdot (u,v,w):= (tu,sv,ts^{-1}w)$, in the middle of the right column the action is the componentwise one and in the remaining corners it is the standard one. All the numbered morphisms are induced by the projections on the respective factors, e.g. $p_4$ is induced by the projection $X\times E_nT \times E_nN \times E_nT \to E_nT$ on the last $E_nT$, and $p_{234}$ is induced by the projection $X\times E_nT \times E_nN \times E_nT \to E_nT \times E_nN \times E_nT$; it is straightforward to check that these projections descend to the quotients under the given group actions. The diagram clearly commutes. Morphism $(s_1,s_2)$ is an isomorphism, and, moreover, morphisms $(p_{12},p_3)$ and $q_{12}$ induce isomorphisms up to degree $n$ on the respective conormed Chow groups by the codimension reasons and homotopy invariance. Unwinding the definition of the Eilenberg--Moore homomorphism~\refbr{eq:EM2} and the isomorphisms~\refbr{eq:isos1} we see that the $\CH_{K,T}^* (\Spec k)$-module structure on $\CH_{K,T}^*(X) \otimes \CH_{K,N}^*(\Spec k)$ in the isomorphism~\refbr{eq:isos2} up to degree $n$ is induced by the ring homomorphism
		\[
			((p_{12},p_3)^*)^{-1}\circ p_4^*\colon \CH_{K}^* (E_nT/T) \to \CH_{K}^* ((X\times E_nT)/T \times E_nN/N)	\cong \CH_{K,T}^* (X)\otimes  \CH_{K,N}^* (\Spec k).
		\]
		Commutativity of the diagram yields that up to degree $n$ we have
		\[
		((p_{12},p_3)^*)^{-1}\circ p_4^* = (r_2,\id)^* \circ ((s_1,s_2)^*)^{-1} \circ (q_{12}^*)^{-1}\circ q_3^*.
		\]
		The pullback homomorphism $q_3^*$ is a representative (up to degree $n$) for the restriction homomorphism 
		\[
		\Phi_{\nu}\colon \CH^*_{K,T}(\Spec k)\to  \CH^*_{K,T\times N}(\Spec k),
		\]
		and $q_{12}^*$ is the canonical isomorphism between the models of the equivariant conormed Chow groups based on the different approximations to the universal torsor $E(T\times N)$, so the pullback homomorphism $(r_2,\id)^* \circ ((s_1,s_2)^*)^{-1} \circ (q_{12}^*)^{-1}\circ q_3^*$ is the ring homomorphism giving precisely the claimed module structure.
	\end{proof}
	
	\subsection{Some equivariant conormed Chow rings of a point}

        In this section we study some equivariant conormed Chow rings of a point, in particular, we compute the conormed rings of a point $\CH^*_{K,\pi_1(G)}(\Spec k)$ for the fundamental groups of non-split quasi-split simple groups, assuming that $K$ splits the group in question. From now on we use some further notation related to groups of multiplicative type introduced in Section~\ref{sec:multiplicative_groups}. In particular, for a degree $2$ Galois field extension $L/k$ we have the following groups:
        \begin{itemize} \itemsep0em 
            \item ${}^2\Gmm$ is the non-trivial form of $\Gmm$ split by $L$,
            \item $\mut{2}{l}$ is a non-trivial form of $\mu_l$ split by $L$ (if $l\ge 3$), and $\mut{2}{2}=\mu_2$,
            \item $\mut{2}{2,2}$ is the non-trivial form of $\mu_2\times\mu_2$ split by $L$.
        \end{itemize}
        These groups come with canonical embeddings into the Weil restriction $R:=R_{L/k}\Gmm$ of the split one-dimensional torus. We denote the restrictions of the standard dimension $2$ vector representation $V_R$ of $R$ to these groups by ${}^2V_{\Gmm}$, ${}^2V_{\mu_l}$ and ${}^2V_{\mu_{2,2}}$ respectively. Furthermore, if $l=2m$ is even, then there is also a nontrivial linear representation of $\mut{2}{2m}$ denoted by ${}^2\Lambda^\pm_{\mu_{2m}}$.
 
	\begin{proposition} \label{prop:equivariant_coef_nl}
		Let $K/L/k$ be separable field extensions with $[L:k]=2$ and let $X\in\Smk$. Then we have the following isomorphisms.
		\begin{enumerate}\itemsep0em 
			\item $\CH_{K,{}^2\Gmm}^*(X) \cong\CH_{K}^*(X)[b]$,
			$b\mapsto c_2({}^2V_{\Gmm})$,
			\item $\CH_{K,\mut{2}{2m+1}}^*(X)\cong \CH_{K}^*(X)$,
			\item $\CH_{K,\mut{2}{2m}}^*(X) \cong \CH_{K}^*(X)[x,b]/(x^2+m^2 \cdot b)$, $x\mapsto c_1({}^2\Lambda_{\mu_{2m}}^{\pm}), b\mapsto c_2({}^2V_{\mu_{2m}})$,
			\item $\CH_{K,\mut{2}{2,2}}^*(X) \cong \CH_{K}^*(X)[b]$, $b\mapsto c_2({}^2V_{\mu_{2,2}})$.	
		\end{enumerate}
		Here $X$ is considered as a variety with the trivial action.
	\end{proposition}
	\begin{proof}
		Put $R:=R_{L/k}\Gmm$ to be the Weil restriction of the split one-dimensional torus and let $V_R$ be its standard vector representation of dimension $2$.
		
		(1) Consider the Eilenberg--Moore homomorphism~\refbr{eq:EM3} 
		\[
		\CH^*_{K,R} (X) \otimes_{\CH^*_{K,\Gmm}(\Spec k)}  \CH^*_{K} (\Spec k) \to \CH^*_{K,{}^2\Gmm} (X)
		\]
		associated with the exact sequence of groups
		\[
		1\to {}^2\Gmm\to R\to \Gmm \to 1.
		\]
		By the construction, this homomorphism arises from a $\Gmm$-torsor
		\[
		(X\times E_nR\times E_n\Gmm)/{}^2\Gmm \to (X\times E_nR\times E_n\Gmm)/R
		\]
		with $E_nR$ and $E_n\Gmm$ being $n$-th approximations to the corresponding universal torsors. Corollary~\ref{cor:EM_qt} applied to $\Gmm$ acting on $(X\times E_nR\times E_n\Gmm)/{}^2\Gmm$ yields that this Eilenberg--Moore homomorphism is an isomorphism. Furthermore, Proposition~\ref{prop:classifying_quasi-trivial} yields isomorphisms
		\[
		\CH^*_{K,R} (X) \cong \CH^*_{K} (X) \otimes \CH^*_{K,R} (\Spec k) \cong  \CH^*_{K} (X) \otimes \CH^*_{K} (\Spec k)[b]
		\]
		with $b\mapsto c_2(V_R)$. The action of $\CH^*_{K,\Gmm}(\Spec k)$ on $\CH^*_{K,R}(X)$ in the Eilenberg--Moore homomorphism factors through the homomorphism
		\[
		\CH^*_{K}(\Spec k)[x]\cong \CH^*_{K,\Gmm}(\Spec k)\to \CH^*_{K,R}(\Spec k) \cong \CH^*_{K} (\Spec k)[b]
		\]
		which maps $x$ to $0$, since the right-hand side has no elements of degree $1$. Combining all the above we obtain
		\[
		\CH^*_{K} (X) [b]\cong \CH^*_{K,R} (X) \cong\CH^*_{K,R} (X) \otimes_{\CH^*_{K,\Gmm}(\Spec k)}  \CH^*_{K} (\Spec k) \cong \CH^*_{K,{}^2\Gmm} (X).
		\]
		
		(2) We have $\CH^*_{K}(\Spec k) = \FF_2$ or $\CH^*_{K}(\Spec k)=0$ by Lemma~\ref{lem:normed_coef}. The claim follows from Lemma~\ref{lem:trivial_coprime}.
		
		(3) Consider the short exact sequence of algebraic groups
		\[
		1\to \mut{2}{2m} \xrightarrow{f} {}^2\Gmm\times \Gmm \xrightarrow{g} R \to 1
		\]
		which is Cartier dual to the short exact sequence of $\Gal(L/k)$-modules
		\begin{gather*}
			0\to \Z\oplus \Z \xrightarrow{\hat{g}} \Z\oplus \Z \xrightarrow{\hat{f}} \Z/2m\Z \to 0, \\
			\hat{g}(x,y)=(mx-my,x+y), \quad \hat{f}(x,y)=x+my,
		\end{gather*}
		with the action of the nontrivial element $\tau\in \Gal(L/k)$ given by
		\[
		\tau(x,y)=(y,x) ,\quad \tau(x,y)=(-x,y) ,\quad \tau(x)=-x
		\]
		on the respective abelian groups. As in part (1) we obtain that the Eilenberg--Moore homomorphism~\refbr{eq:EM3}
		\begin{equation} \label{eq:is1}
			\CH^*_{K,{}^2\Gmm\times \Gmm} (X) \otimes_{\CH^*_{K,R}(\Spec k)}  \CH^*_{K} (\Spec k) \to \CH^*_{K,\mut{2}{2m}} (X)
		\end{equation}
		is an isomorphism. Proposition~\ref{prop:classifying_quasi-trivial} together with part (1) of the current proposition yields isomorphisms
		\begin{gather*}
			\CH^*_{K,{}^2\Gmm\times \Gmm} (X) \cong \CH^*_{K} (X)[x,b],\quad
			\CH^*_{K,R}(\Spec k)\cong \CH^*_{K}(\Spec k)[b'],\\ \CH^*_{K,{}^2\Gmm\times \Gmm} (\Spec k)\cong  \CH^*_{K} (\Spec k)[x,b]
		\end{gather*}
		with $b'\mapsto c_2(V_R)$, $x\mapsto c_1(V_{\Gmm})$, $b\mapsto c_2({}^2V_{\Gmm})$, and the $\CH^*_{K}(\Spec k)[b']$-module structure on $\CH^*_{K} (X)[x,b]$ factors through the homomorphism
		\[
		\CH^*_{K}(\Spec k)[b']\cong \CH^*_{K,R}(\Spec k)\xrightarrow{\Phi_g}  \CH^*_{K,{}^2\Gmm\times \Gmm} (\Spec k)\cong \CH^*_{K} (\Spec k)[x,b].
		\]
		In the representation ring $\mathrm{Rep}(R)\cong \mathbb{Z}[\mc{X}^*(R)]^{\Gal(L/k)}\cong \mathbb{Z}[\mathbb{Z}\oplus\mathbb{Z}]^{\Gal(L/k)}$ we have $[V_R]=x^{(1,0)}+x^{(0,1)}$. Then
		\[
		\Res_g([V_R])= \Res_g(x^{(1,0)}+x^{(0,1)})= x^{\hat{g}(1,0)}+ x^{\hat{g}(0,1)} = x^{(m,1)}+ x^{(-m,1)} = (x^{(m,0)}+x^{(-m,0)})\cdot x^{(0,1)}
		\]
		in $\mathrm{Rep}({}^2\Gmm\times \Gmm)\cong (\mathbb{Z}[\mc{X}^*({}^2\Gmm)]\otimes \mathbb{Z}[\mc{X}^*(\Gmm)])^{\Gal(L/k)}$ with $(1,0)$ and $(0,1)$ being the generators for $\mc{X}^*({}^2\Gmm)$ and $\mc{X}^*(\Gmm)$ respectively. Then 
		\[
		\Res_g([V_R])= \psi^m[{}^2V_{\Gmm}]\otimes [V_{\Gmm}],
		\]
		where $\psi^m$ is the $m$-th Adams operation. Using the standard properties of the Chern classes and that $c_1({}^2V_{\Gmm})=0$, since $\CH^1_{K,{}^2\Gmm}(\Spec k)=0$, we obtain
		\begin{multline*}
			\Phi_g(c_2(V_R))=c_2(\Res_g([V_R]))=c_2(\psi^m[{}^2V_{\Gmm}]\otimes [V_{\Gmm}]) =\\
			=c_1(V_{\Gmm})^2 + m\cdot c_1(V_{\Gmm})\cdot c_1({}^2V_{\Gmm})+ m^2 \cdot c_2({}^2V_{\Gmm})= c_1(V_{\Gmm})^2 + m^2 \cdot c_2({}^2V_{\Gmm}).
		\end{multline*}
		Putting everything together in isomorphism~\refbr{eq:is1} we obtain
		\[
		\CH^*_{K,\mut{2}{2m}} (X) \cong \CH^*_{K} (X)[x,b]/(x^2+m^2\cdot b),\quad x\mapsto 	\Phi_f(c_1(V_{\Gmm})),\, b\mapsto \Phi_f(c_2({}^2V_{\Gmm})).
		\]
		We have 
		\[
		\Phi_f(c_1(V_{\Gmm}))=c_1(\Res_f([V_{\Gmm}]))=c_1({}^2\Lambda_{\mu_{2m}}^\pm),\quad
		\Phi_f(c_2({}^2V_{\Gmm}))=c_2(\Res_f([{}^2V_{\Gmm}]))=c_2({}^2V_{\mu_{2m}}),
		\]
		thus $x$ maps to $c_1({}^2\Lambda_{\mu_{2m}}^\pm)$ and $b$ maps to $c_2({}^2V_{\mu_{2m}})$ as claimed.
		
		(4) Consider the standard embedding $f\colon \mut{2}{2,2}\to R$. As in part (1), the Eilenberg--Moore homomorphism~\refbr{eq:EM3}
		\begin{equation} \label{eq:eqpoint}
		\CH^*_{K,R} (X) \otimes_{\CH^*_{K,R}(\Spec k)}  \CH^*_{K} (\Spec k) \to \CH^*_{K,\mut{2}{2,2}} (X)
  		\end{equation}
		associated with the short exact sequence of groups
		\[
		1\to \mut{2}{2,2}\xrightarrow{f} R\xrightarrow{g} R \to 1,\quad g(x)=x^2,
		\]
		is an isomorphism. Proposition~\ref{prop:classifying_quasi-trivial} yields
		\[
		\CH^*_{K,R} (X)\cong \CH^*_{K} (X) \otimes  \CH^*_{K,R} (\Spec k) \cong \CH^*_{K} (X)[b], \quad b\mapsto c_2(V_R),
		\]
		and the $\CH^*_{K,R}(\Spec k)$-module structure in isomorphism~\refbr{eq:eqpoint} is induced by the homomorphism
		\[
		\Phi_g\colon \CH^*_{K,R}(\Spec k) \to \CH^*_{K,R}(\Spec k).
		\]
		Computing the respective characters we obtain $\Res_g([V_R]) = \psi^2[V_R]$
		for the second Adams operation $\psi^2$, thus
		\[
		\Phi_g(c_2(V_R)) = c_2(\Res_g([V_R])) = c_2(\psi^2[V_R])=4c_2(V_R)=0
		\]
		with the last equality arising from $2=0$ in $\CH^*_{K} (\Spec k)$ by Lemma~\ref{lem:normed_coef}. It follows that 
		\[
		\CH^*_{K} (X)[b]\cong \CH^*_{K,R} (X)\xrightarrow{\Phi_{f}} \CH^*_{K,\mut{2}{2,2}} (X)
		\]
		is an isomorphism. We have $\Phi_{f}(c_2(V_R))=c_2(\Res_f([V_{R}])) = c_2({}^2V_{\mu_{2,2}})$, so $b$ maps to $c_2({}^2V_{\mu_{2,2}})$ as claimed.
	\end{proof}
	
	\begin{lemma} \label{lem:c2_mixed_rep}
		Let $K/L/k$ be separable field extensions with $[L:k]=2$, put $R:=R_{L/k} \Gmm$, $N:=\mut{2}{2m}$ or $N:=\mut{2}{2,2}$ and $V_N:= {}^2V_{\mu_{2m}}$ or $V_N:= {}^2V_{\mu_{2,2}}$ respectively. Let $W$ be a rank $2$ representation of $R\times N$ such that for its restrictions to the factors $\rho_1\colon R\to R\times N$ and $\rho_2\colon N\to R\times N$ we have 
        \[
        \Res_{\rho_1}([W])=[V_R],\quad \Res_{\rho_2}([W])=\psi^i [V_{N}]
        \]
        for the $i$-th Adams operation $\psi^i$. Then
		\[
		c_2(W)= c_2(V_R)\otimes 1 + i^2 \cdot 1\otimes c_2(V_N) \in \CH_{K,R}^*(\Spec k)\otimes\CH_{K,N}^*(\Spec k) \cong \CH_{K,R\times N}^*(\Spec k).
		\]
	\end{lemma}
	\begin{proof}
		By Propositions~\ref{prop:classifying_quasi-trivial} and~\ref{prop:equivariant_coef_nl} the degree $2$ part of $\CH_{K,R}^*(\Spec k)\otimes\CH_{K,N}^*(\Spec k)$ is spanned by $	c_2(V_R)\otimes 1$ and $1\otimes c_2(V_N)$ and since $2=0$ in $\CH_{K}^*(\Spec k)$ by Lemma~\ref{lem:normed_coef}, it follows that the map
		\[
		(\Phi_{\rho_1},\Phi_{\rho_2})\colon \CH_{K,R\times N}^2(\Spec k) \to \CH_{K,R}^2(\Spec k)\times\CH_{K,N}^2(\Spec k) 
		\]
		is bijective. We have 
		\begin{gather*}
			\Phi_{\rho_1}(c_2(W))=c_2(\Res_{\rho_1}([W]))=c_2(V_R)=\Phi_{\rho_1}(c_2(V_R)\otimes 1 + i^2 \cdot 1\otimes c_2(V_N)),\\
			\Phi_{\rho_2}(c_2(W))=c_2(\Res_{\rho_2}([W]))=c_2(\psi^i[V_N])=i^2\cdot c_2(V_N) = \Phi_{\rho_2}(c_2(V_R)\otimes 1 + i^2 \cdot 1\otimes c_2(V_N)).
		\end{gather*}
		The claim follows.
	\end{proof}
	
	\subsection{Conormed characteristic sequence for quasi-split simple groups}

        In this section we introduce the main computational tools for the conormed Chow rings of a quasi-split simple group, namely the conormed characteristic sequence and the conormed extended characteristic sequence, and show that the latter one is exact while the former one is exact provided that the groups is either simply connected or adjoint.  We also show that $G$ satisfies the conormed K\"unneth formula for the direct product with an arbitrary $X\in \Smk$. Here and below we use some standard notation related to quasi-split algebraic groups, see Sections~\ref{sec:multiplicative_groups} and~\ref{sec:quasi-split_simple} for a recollection. 
	
	\begin{lemma} \label{lem:Kunneth_for_pi1}
		Let $G$ be a product of a finite number of quasi-split simple groups over a field $k$ and let $K/k$ be a separable field extension such that $G_K$ is split. Then $B\pi_1(G)$ satisfies the K\"unneth formula for $K/k$ in the sense of Definition~\ref{def:Kunneth_formula_BN}.
	\end{lemma}
	\begin{proof}
	We can assume that $G$ is a quasi-split simple group. The proof goes case by case on all the possible fundamental groups, see e.g. recollection~\refbr{expos:fund} for a list. 
 
 If $G$ is split, then $\pi_1(G)\cong \mu_l$ or $\pi_1(G)\cong \mu_2\times \mu_2$ and it is well known that the classifying spaces of such groups satisfy K\"unneth formula for the Chow groups (see, e.g., \cite[proof of Theorem~2.10]{To14b}), thus also for the conormed Chow groups. So we can assume that $G$ is not split.
		
		Let $L$ be the splitting field of $G$. Since $G_K$ is split, we can assume that $L\subseteq K$. The field extension $L/k$ is Galois and $[L:k]=2$, if $\Delta(G)=\Att_n$, ${}^2\mr{E}_6$ or ${}^2\mr{D}_n$, and $[L:k]=3$ or $6$, if $\Delta(G)={}^3\mr{D}_4$ or $\Delta(G)={}^6\mr{D}_4$ respectively.
		
		If $[L:k]=6$, then $\CH_{K}^*(\Spec k)=0$ by Lemma~\ref{lem:normed_coef} and there is nothing to prove.
		
		If $[L:k]=3$, then $\Delta(G)={}^3\mr{D}_4$ and $\pi_1(G)$ is either trivial or $\pi_1(G)\cong \mut{3}{2,2}$, so its order is invertible in $\CH_{K}^*(\Spec k)$ by Lemma~\ref{lem:normed_coef}. The claim follows from Lemma~\ref{lem:trivial_coprime}.
		
		If $[L:k]=2$, then $\pi_1(G)\cong \mut{2}{l}$ for some $l\in \NN$ or $\pi_1(G)\cong \mut{2}{2,2}$. The claim follows from Proposition~\ref{prop:equivariant_coef_nl}.
	\end{proof}
	
	\begin{remark}
		Let $G$ be a quasi-split semisimple group over a field $k$ and $K/k$ be a separable field extension such that $G_K$ is split. We do not know whether $B\pi_1(G)$ satisfies the K\"unneth formula for $K/k$, but is seems highly plausible.
	\end{remark}

 	\begin{proposition} \label{prop:Kunneth_conormed}
		Let $G$ be a quasi-split simple group over a field $k$ and $K/k$ be a separable field extension such that $G_K$ is split. Then for $X\in\Smk$ the K\"unneth homomorphism
		\[
		\CH_{K}^*(X)\otimes \CH_{K}^*(G) \to \CH_{K}^*(X\times G),\quad x\otimes y \mapsto p_X^*(x)\cdot p_G^*(y),
		\]
		is an isomorphism. Here $p_X\colon X\times G\to X$, $p_G\colon X\times G\to G$ are the projections. In particular, 
		\[
		\Delta\colon  \CH_{K}^*(G)\xrightarrow{m^*}  \CH_{K}^*(G\times G) \cong \CH_{K}^*(G)\otimes \CH_{K}^*(G)
		\]
		for the multiplication $m\colon G\times G\to G$ equips $\CH_{K}^*(G)$ with the structure of a Hopf algebra.
	\end{proposition}
	\begin{proof}
		Let $\tilde{G}\to G$ be the simply connected cover and $\tilde{T}\le \tilde{B}\le \tilde{G}$ be a maximal torus and a Borel subgroup. Lemma~\ref{lem:quasi-split_quasi-trivial} yields that $\tilde{T}$ is a quasi-trivial torus and Lemma~\ref{lem:Kunneth_for_pi1} yields that $B\pi_1(G)$ satisfies the K\"unneth formula for $K/k$. Then it follows from Proposition~\ref{prop:EM_qt_extended} that there are isomorphisms
		\begin{gather*}
			(\CH_{K,\tilde{T}}^*(\tilde{G}) \otimes \CH_{K,\pi_1(G)}^*(\Spec k)) \otimes_{\CH_{K,\tilde{T}}^* (\Spec k)} \CH_{K}^* (\Spec k) \xrightarrow{} \CH_{K,\pi_1(G)}^* (\tilde{G}),\\		
			(\CH_{K,\tilde{T}}^*(X\times \tilde{G}) \otimes \CH_{K,\pi_1(G)}^*(\Spec k)) \otimes_{\CH_{K,\tilde{T}}^* (\Spec k)} \CH_{K}^* (\Spec k) \xrightarrow{} \CH_{K,\pi_1(G)}^* (X\times \tilde{G})
		\end{gather*}
		with $X$ equipped with the trivial action. Lemma~\ref{lem:torsor} yields
		\[
		\CH_{K,\pi_1(G)}^* (\tilde{G})\cong \CH_{K}^* (G),\,\, \CH_{K,\pi_1(G)}^* (X\times \tilde{G})\cong \CH_{K}^* (X\times G),\,\, \CH_{K,\tilde{T}}^*(X\times \tilde{G})\cong \CH_{K}^* (X \times \tilde{G}/\tilde{T}).
		\]
		Note that
		\[
		\CH_{K}^* (\tilde{G}/\tilde{T}) \cong \CH_{K}^* (\tilde{G}/\tilde{B}),\quad
		\CH_{K}^* (X \times \tilde{G}/\tilde{T}) \cong \CH_{K}^* (X \times \tilde{G}/\tilde{B})
		\]
		by the homotopy invariance property. It follows from \cite[Lemma~29]{CM06} that the (integral) Chow motive $\mc{M}(\tilde{G}/\tilde{B})$ splits as a direct sum of $\mc{M}(\Spec F,l)$ for some $l\in\NN_0$ and field extensions $K/F/k$. Lemma~\ref{lem:trivial_normed_corestriction} yields $\mc{M}_K(\Spec F,l)=0$ for $K\supseteq F \supsetneq k$, so we obtain a decomposition
		\[
		\mc{M}_K(\tilde{G}/\tilde{B})\cong \bigoplus_{i\in I} \mc{M}_K(\Spec k,l_i)
		\]
		for some $l_i\in \NN_0$, yielding a K\"unneth isomorphism
		\[
		\CH_{K}^* (X) \otimes \CH_{K}^* (\tilde{G}/\tilde{B})\cong \CH_{K}^* (X \times \tilde{G}/\tilde{B}).
		\]
		Combining all the isomorphisms above we get the claim.
	\end{proof}
	
	\begin{definition} \label{def:exact_sequence}
		Consider a sequence 
		\[
		\mc{A}^* \xrightarrow{c} \mc{B}^* \xrightarrow{\phi^*} \mc{C}^* \to 0
		\]
		of $\Z_{\ge 0}$-graded commutative rings with $c$ and $\phi^*$ being graded degree $0$ homomorphisms. We say that this sequence is \textit{exact}, if $\phi^*$ is surjective and $c(\mc{A}^{>0})$ generates $\ker \phi^*$ as an ideal.
	\end{definition}
	
	\begin{definition} \label{def:characteristic_map}
		Let $K/k$ be a separable extension of fields, let $E$ be a $G$-torsor for a quasi-split group $G$ over $k$, and let $T\le B\le G$ be a maximal torus and a Borel subgroup. Then the \textit{characteristic map} is the composition
		\[
		c\colon \CH_{K,T}^*(\Spec k) \to \CH_{K,T}^*(E) \cong \CH_{K}^*(E/T) \cong \CH_{K}^*(E/B),
		\]
		where the first map is the pullback for the structure map, the middle isomorphism is given by Lemma~\ref{lem:torsor} and the last isomorphism is the homotopy invariance property for the vector bundle torsor $E/T \to E/B$. The \textit{characteristic sequence} is the sequence
		\[
		\CH_{K,T}^*(\Spec k) \xrightarrow{c} 	\CH_{K}^*(E/B) \xrightarrow{\phi^*} 	\CH_{K}^*(E) \to 0
		\]	
		where $\phi\colon E\to E/B$ is the projection. 
		
		Suppose that $G$ is a quasi-split semisimple group such that $B\pi_1(G)$ satisfies the K\"unneth formula for $K/k$, let $\tilde{G}\to G$ be its simply connected cover and $\tilde{T}\le \tilde{B} \le \tilde{G}$ be the maximal torus and the Borel subgroup lying over $T\le B$. Then the \textit{extended characteristic sequence} is the sequence
		\[
		\CH_{K,\tilde{T}}^*(\Spec k) \xrightarrow{\hat{c}} 	\CH_{K}^*(\tilde{G}/\tilde{B})\otimes \CH_{K,\pi_1(G)}^*(\Spec k) \xrightarrow{\hat{\phi}} 	\CH_{K}^*(G) \to 0
		\]
		where 
		\begin{itemize}
			\item 
			$\hat{c}$ is the composition
			\begin{multline*}
				\CH_{K,\tilde{T}}^*(\Spec k) \xrightarrow{\Phi_{\nu}} \CH_{K,\tilde{T}\times \pi_1(G)}^*(\Spec k)\cong \CH_{K,\tilde{T}}^*(\Spec k) \otimes \CH_{K,\pi_1(G)}^*(\Spec k) \xrightarrow{\tilde{c}\otimes \id} \\
				\xrightarrow{\tilde{c}\otimes \id} \CH_{K}^*(\tilde{G}/\tilde{B}) \otimes \CH_{K,\pi_1(G)}^*(\Spec k)
			\end{multline*}
			with $\nu\colon \tilde{T}\times \pi_1(G)\to \tilde{T}$ being the group homomorphisms given by $\nu(t,s)=t\cdot s^{-1}$, and  $\tilde{c}$ being the characteristic map for $\tilde{G}$,
			\item 
			$\hat{\phi}(x,\alpha) = \phi_G^*(x)\cdot p (\alpha)$ with $\phi_G\colon G\to G/B \cong \tilde{G}/\tilde{B}$ being the projection combined with the isomorphism $\tilde{G}/\tilde{B} \cong G/B$, and $p$ being the composition
			\[
			\CH_{K,\pi_1(G)}^*(\Spec k) \xrightarrow{p_{\tilde{G}}^*} \CH_{K,\pi_1(G)}^*(\tilde{G}) \cong \CH_{K}^*(G)
			\] 
			for the structure map $p_{\tilde{G}}\colon \tilde{G}\to \Spec k$ and the isomorphism given by Lemma~\ref{lem:torsor}.
		\end{itemize}
	\end{definition}
	
	\begin{theorem} \label{thm:char_qs}
		Let $G$ be a quasi-split algebraic group over a field $k$, let $T\le B\le G$ be a maximal torus and a Borel subgroup, and let $K/k$ be a separable extension of fields such that $T_K$ is split. 
		\begin{enumerate}
			\item If $T$ is quasi-trivial, then for a $G$-torsor $E$ the characteristic sequence
			\[
			\CH_{K,T}^*(\Spec k) \xrightarrow{c} 	\CH_{K}^*(E/B) \xrightarrow{\phi^*} 	\CH_{K}^*(E) \to 0
			\]		
			is exact. In particular, this applies to a split $G$ and, by Lemma~\ref{lem:quasi-split_quasi-trivial}, to a simply connected or an adjoint quasi-split semisimple group $G$.
			
			\item Let $G$ be a quasi-split semisimple group, $\tilde{G}\to G$ be its simply connected cover and $\tilde{T}\le \tilde{B}\le \tilde{G}$ be the maximal torus and the Borel subgroup lying over $T\le B$. Suppose that $B\pi_1(G)$ satisfies the K\"unneth formula for $K/k$. Then the extended characteristic sequence
			\[
			\CH_{K,\tilde{T}}^*(\Spec k) \xrightarrow{\hat{c}} 	\CH_{K}^*(\tilde{G}/\tilde{B})\otimes \CH_{K,\pi_1(G)}^*(\Spec k) \xrightarrow{\hat{\phi}} 	\CH_{K}^*(G) \to 0
			\]
			is exact. In particular, by Lemma~\ref{lem:Kunneth_for_pi1} this applies to $G$ being a product of quasi-split simple groups.
		\end{enumerate}
	\end{theorem}
	\begin{proof}
		(1) Corollary~\ref{cor:EM_qt} yields an exact sequence
		\[
		\CH_{K,T}^*(\Spec k) \to 	\CH_{K, T}^*(E) \xrightarrow{\Phi_T} 	\CH_{K}^*(E) \to 0
		\]
		where the first morphism is the pullback for the projection $E\to \Spec k$. The claim follows via the identification $\CH_{K,T}^*(E)\cong  \CH_{K}^*(E/B)$ as in Definition~\ref{def:characteristic_map}.
		
		(2) Proposition~\ref{prop:EM_qt_extended} applied to $X:=\tilde{G}$ with the action of $\tilde{T}$ and the subgroup $\pi_1(G)\le \tilde{T}$ (we use the notation of Definition~\ref{def:characteristic_map} for the simply connected cover $\tilde{G}$ of $G$ and for its subgroups) yields an exact sequence
		\[
		\CH_{K,\tilde{T}}^*(\Spec k) \to	\CH_{K,\tilde{T}}^*(\tilde{G})\otimes \CH_{K,\pi_1(G)}^*(\Spec k) \to \CH_{K,\pi_1(G)}^*(\tilde{G}) \to 0
		\]
		where the first homomorphism is the composition
		\begin{multline*}
			\CH_{K,\tilde{T}}^*(\Spec k) \xrightarrow{\Phi_\nu} \CH_{K,\tilde{T}\times \pi_1(G)}^*(\Spec k) \cong \CH_{K,\tilde{T}}^*(\Spec k)\otimes \CH_{K,\pi_1(G)}^*(\Spec k) \to \\ 
			\xrightarrow{ p_{\tilde{G}}^*\otimes \id} \CH_{K,\tilde{T}}^*(\tilde{G})\otimes \CH_{K,\pi_1(G)}^*(\Spec k)
		\end{multline*}
		and the second homomorphism is given by $(x,\alpha)\mapsto p_{\tilde{G}}^*(\alpha)\cdot \Phi_{\rho}(x)$. Here $\nu\colon \tilde{T}\times \pi_1(G)\to \tilde{T}$ is the homomorphism given by $\nu(t,s)=ts^{-1}$, $\rho\colon \pi_1(G)\to \tilde{T}$ is the embedding, and $p_{\tilde{G}}\colon\tilde{G}\to \Spec k$ is the structure morphism. Identifying $\CH_{K,\tilde{T}}^*(\tilde{G})\cong  \CH_{K}^*(\tilde{G}/\tilde{B})$ and $\CH_{K,\pi_1(G)}^*(\tilde{G})\cong \CH_{K}^*(G)$ as in Definition~\ref{def:characteristic_map} we obtain the claim.
	\end{proof}
	
	\begin{remark}
		If $G$ is a split group, then Theorem~\ref{thm:char_qs} is well-known and goes back to \cite{Gro58}.
	\end{remark}
	
	\subsection{Conormed Chow ring of a quasi-split group $G$ via $G/B$}
        
        In this section we use the conormed extended characteristic sequence to give a presentation for $\CH^*_K(G)$ as a quotient of $\CH^*_K(G/B)$ or $\CH^*_K(G/B)[x]$ modulo some explicit ideal. Here and below we use some standard notation for vector bundles over homogeneous spaces recalled in Section~\ref{sec:vector_bundles}, in particular, for the maximal torus $\tilde{T}$ in the simply connected cover $\tilde{G}\to G$ lying over a maximal torus $T\le B$ we have 
        \begin{itemize}  \itemsep0em 
            \item line bundle $\mc{L}(\varpi)$ over $G/B$ associated with a Galois-invariant weight $\varpi \in \mc{X}^*(\tilde{T})$,
            \item vector bundle $\mc{V}(\mc{S})$ over $G/B$ associated with a Galois-invariant finite set of weights $\mc{S} \subseteq \mc{X}^*(\tilde{T})$,
            \item line bundle $\mc{L}(\bar{\varpi})$ over $G$ associated with a Galois-invariant element $\bar{\varpi} \in \mc{X}^*(\tilde{T})/\mc{X}^*(T)$.
        \end{itemize}
        Furthermore, we use the standard numbering for simple roots $\alpha_i$ and fundamental weights $\varpi_i$ recalled in Section~\ref{sec:quasi-split_simple}, and shorten the notation as $\mc{L}_i:=\mc{L}(\varpi_i)$, $\mc{V}_{i_1,i_2,\hdots,i_j}:=\mc{V}(\{\varpi_{i_1},\varpi_{i_2},\hdots,\varpi_{i_j}\})$.
	
	\begin{lemma}[{see also \cite[Example~3.12]{KZ13}}] \label{lem:Chow_quasi-split_homogeneous}
		Let $G$ be a quasi-split reductive group over a field $k$, let $K/k$ be the splitting field of $G$ and let $P\le G$ be a parabolic subgroup. Then for every $n\in\Z$ and the canonical projection $\rho\colon (G/P)_K\to G/P$ the pullback homomorphism
		\[
		\rho^*\colon \CH^*(G/P)\otimes \Z/n\Z \to \CH^*((G/P)_K)\otimes \Z/n\Z
		\]
		is injective and $\rho^*(\CH^*(G/P)\otimes \Z/n\Z) = (\CH^*((G/P)_K)\otimes \Z/n\Z)^{\Gal(K/k)}$.
	\end{lemma}
	\begin{proof}
		It follows from \cite[Lemma~29]{CM06} that there is a motivic decomposition
		\[
		\mc{M}(G/P) \cong \bigoplus_{i\in I} \mc{M}(\Spec F_i, l_i)
		\]
		for some $l_i\in \NN_0$ and field extensions $K/F_i/k$. At the same time we have
		\[
		\mc{M}((G/P)_K) \cong \bigoplus_{i\in I} \mc{M}(\Spec (F_i \otimes_k K), l_i)\cong \bigoplus_{i\in I} \mc{M}(\Spec K, l_i)^{\oplus [F_i:k]}.
		\]	
		The pullback $\rho^*$ is given by the direct sum of diagonal embeddings
		\[
		\rho^*\colon \CH^*(G/P)\otimes \Z/n\Z \cong \bigoplus_{i\in I} \Z/n\Z \xrightarrow{\oplus \mathrm{diag}}  \bigoplus_{i\in I} (\Z/n\Z)^{\oplus [F_i:k]}  \cong  \CH^*((G/P)_K)\otimes \Z/n\Z.
		\]
		The group $\Gal(K/k)$ acts on $\bigoplus_{i\in I} (\Z/n\Z)^{\oplus [F_i:k]}$ simultaneously permuting the factors of each summand $(\Z/n\Z)^{\oplus [F_i:k]}$, so the claim follows.
	\end{proof}

	\begin{lemma} \label{lem:full_flag_codim2}
		Let $G$ be a quasi-split simple group over a field $k$ of type $\Delta(G)={}^2\mathrm{A}_{2r-1}$, let $B\le G$ be a Borel subgroup and let $K/k$ be the splitting field of $G$. Then in $\CH^*_K(G/B)$ we have the following:
		\begin{enumerate}
			\item $\CH^1_K(G/B)=\FF_2 \cdot c_1(\mc{L}_r)$,
			\item $c_1(\mc{L}_r)^2=0$,
			\item $\CH^2_K(G/B)=\FF_2 \cdot c_2(\mc{V}_{1,2r-1})\oplus \FF_2 \cdot c_2(\mc{V}_{2,2r-2})\oplus\hdots\oplus \FF_2 \cdot c_2(\mc{V}_{r-1,r+1})$.
		\end{enumerate}
	\end{lemma}
	\begin{proof}
		We have $[K:k]=2$ and $\CH^*_K(\Spec k)\cong \FF_2$ by Lemma~\ref{lem:normed_coef}. Since $G_K$ is split, $(G/B)_K$ is isomorphic to the variety of full flags in a vector space of dimension $2r$, and
		\[
		\CH^*((G/B)_K) \cong \Z [x_1,x_2,\hdots,x_{2r}]/(\sigma_1,\sigma_2,\hdots,\sigma_{2r}),\quad x_i\mapsto c_1(\mc{T}_i),
		\]
		where $\sigma_i$ is the $i$-th elementary symmetric polynomial in $x_1,x_2,\hdots,x_{2r}$ and $\mc{T}_i$ is the $i$-th tautological line bundle over the full flag variety \cite[Example~3.3.5]{Fu98}. In the notation of~\refbr{expos:simple_groups} we have $\mc{L}(\epsilon_i)=\mc{T}_i$, thus the nontrivial element $\tau\in\Gal(K/k)$ acts on $\CH^*((G/B)_K)$ via 
        \[
        \tau(\bar{x}_i)=\tau(c_1(\mc{L}(\epsilon_i)))=c_1(\mc{L}(\tau(\epsilon_i)))=c_1(\mc{L}(-\epsilon_{2r+1-i})) =-\bar{x}_{2r+1-i}
        \]
        with $\bar{x}_i$ being the image of $x_i$ in the quotient ring. 	Lemma~\ref{lem:Chow_quasi-split_homogeneous} yields that the pullback homomorphism for the morphism $(G/B)_K\to G/B$ induces an isomorphism
		\[
		\CH^*(G/B)\cong (\Z [x_1,x_2,\hdots,x_{2r}]/(\sigma_1,\sigma_2,\hdots,\sigma_{2r}))^{\Gal(K/k)}.
		\]
		Hence
		\[
		\CH^*_K(G/B) \cong \left(\Z [x_1,x_2,\hdots,x_{2r}]/(\sigma_1,\sigma_2,\hdots,\sigma_{2r})\right)^{\Gal(K/k)} / (\id + \tau)(\Z [x_1,x_2,\hdots,x_{2r}]/(\sigma_1,\sigma_2,\hdots,\sigma_{2r})).
		\]
		Furthermore, it follows from the formulae in~\refbr{expos:simple_groups} for the fundamental weights that 
		\[
		(\mc{L}_r)_K\cong\bigotimes_{i=1}^r \mc{T}_i, \quad (\mc{V}_{i,2r-i})_K\cong \bigotimes_{j=1}^i \mc{T}_j \oplus \bigotimes_{j=1}^{2r-i} \mc{T}_j
		\]
		and in $\CH^*(G/B)\cong \CH^*((G/B)_K)^{\Gal(K/k)}$ we have
		\[
		c_1(\mc{L}_r)= \bar{x}_1+\bar{x}_2+\hdots+\bar{x}_r, \quad c_2(\mc{V}_{i,2r-i}) = (\bar{x}_1+\bar{x}_2+\hdots+\bar{x}_i)\cdot (\bar{x}_1+\bar{x}_2+\hdots+\bar{x}_{2r-i}).
		\]
		The claims of the lemma follow by straightforward calculations with polynomials which we provide below. 
		
		The first claim follows from the $\Gal(K/k)$-invariant decomposition
		\[
		\CH^1((G/B)_K)\cong\left(\bigoplus_{i=1}^{2r} \Z x_i\right)/\Z\sigma_1 = \Z\cdot (\bar{x}_1+\hdots+\bar{x}_r) \oplus \left(\bigoplus_{i=1}^{r-1}[\mathbb{Z}\bar{x}_i\oplus \mathbb{Z}\cdot (-\bar{x}_{2r+1-i})] \right),
		\]
		here the first summand is a rank $1$ submodule with the trivial action of $\Gal(K/k)$, and then there are $r-1$ rank $2$ submodules with the permutation action of $\Gal(K/k)$. Then $\CH^1_K(G/B)$ is of rank $1$ with the generator given by the image of $c_1(\mc{L}_r)=\bar{x}_1+\hdots+\bar{x}_r$.
		
		For the second claim note that
		\[
		(x_1+x_2+\hdots+x_r)^2 =(x_1+x_2+\hdots+x_r) \cdot \sigma_1 - \sigma_2 + (\id+\tau)(\alpha)
		\]
		for some $\alpha\in \Z [x_1,x_2,\hdots,x_{2r}]$, thus $c_1(\mc{L}_r)^2=0$ in $\CH^*_K(G/B)$.
		
		For the last claim consider the following decomposition of the degree $2$ part of $\Z [x_1,x_2,\hdots,x_{2r}]$ into $\Gal(K/k)$-invariant submodules:
		\begin{multline*}
			\Z [x_1,x_2,\hdots,x_{2r}]^{(2)} 
			= \left(\bigoplus_{i=1}^{r}\left[\Z x_i\sigma_1 \oplus \Z x_{2r-i+1}\sigma_1\right]\right)\oplus \left[\Z\sigma_2\right]\oplus \\ \oplus \left(\bigoplus_{i=1}^{r-1} \left[\Z x_ix_{2r-i+1}\right] \right) \oplus \left(\bigoplus_{1\le i<j<2r-i+1} \left[\Z x_ix_j\oplus \Z x_{2r-i+1}x_{2r-j+1}\right]  \right),
		\end{multline*}
		here the square brackets group the $\Gal(K/k)$-invariant submodules, i.e. we have $r$ submodules of rank $2$ with the permutation action of $\Gal(K/k)$, then one of rank $1$ with the trivial action of $\Gal(K/k)$, then $r-1$ submodules of rank $1$ again with the  trivial action, and then $r(r-1)$ submodules of rank $2$ with the permutation action. It follows that 
        \[
        \CH^2((G/B)_K) \cong \left(\bigoplus_{i=1}^{r-1} \left[\Z \bar{x}_i\bar{x}_{2r-i+1}\right] \right) \oplus \left(\bigoplus_{1\le i<j<2r-i+1} \left[\Z \bar{x}_i\bar{x}_j\oplus \Z \bar{x}_{2r-i+1}\bar{x}_{2r-j+1}\right]\right),
        \]
        and $\CH^2_K(G/B)$ has an additive basis given by the images in the respective quotient ring of $\bar{x}_i\bar{x}_{2r-i+1}$ for $1\le i\le r-1$. We have 
		\begin{multline*}
			c_2(\mc{V}_{i,2r-i}) = (\bar{x}_1+\bar{x}_2+\hdots+\bar{x}_i)\cdot (\bar{x}_1+\bar{x}_2+\hdots+\bar{x}_{2r-i})=(\bar{x}_1+\bar{x}_2+\hdots+\bar{x}_i)\cdot (\bar{x}_1+\bar{x}_2+\hdots+\bar{x}_{2r-i} - \bar{\sigma}_1)= \\
			=-(\bar{x}_1+\bar{x}_2+\hdots+\bar{x}_i)\cdot (\bar{x}_{2r-i+1}+\bar{x}_{2r-i+2}+\hdots+\bar{x}_{2r})
		\end{multline*}
		in $\CH^2(G/B)$ yielding that in $\CH^2_K(G/B)$ we have $c_2(\mc{V}_{i,2r-i})=\sum_{j=1}^i \bar{x}_j\bar{x}_{2r-j+1}$. The claim follows.
	\end{proof}
	
	\begin{theorem} \label{thm:group_via_cover} 
		Let $G$ be a non-split quasi-split simple algebraic group over a field $k$, let $T\le B\le G$ be a maximal torus and a Borel subgroup, and let $K/k$ be a separable field extension such that $T_K$ is split. Put $\mc{R}^*:=\CH_{K}^*(G/B)$. If $\Delta(G)={}^6\mr{D}_4$, then $\CH^*_K(G)=0$, otherwise, depending on $(\Delta(G),\pi_1(G))$, we have $\CH^*_K(G)\cong \mc{R}^*/I$ or $\CH^*_K(G)\cong \mc{R}^*[x]/I$ with the generators of the ideal $I$ given in the following table.
         \begin{center}
			\def\arraystretch{1.5}
			\begin{longtable*}{c|c|c|c}
    			$\Delta(G)$ & $\pi_1(G)$ & $\mc{R}^*$ or $\mc{R}^*[x]$ & generators of $I$ \\
				\hline
                    $\Att_{2r}$, $r\ge 1$  & $\mut{2}{l}$, $l\mid 2r+1$ & $\mc{R}^*$ & $\{c_2(\mc{V}_{i,2r+1-i})\}_{i=1}^{r}$\\
                    $\Att_{2r-1}$, $r\ge 2$  & $\mut{2}{2m+1}$, $2m+1\mid 2r$ & $\mc{R}^*$ & $\{c_1(\mc{L}_r)\}\cup\{c_2(\mc{V}_{i,2r-i})\}_{i=1}^{r-1}$\\
                     & $\mut{2}{2m}$, $m$ even, $\frac{r}{m}$ odd & $\mc{R}^*$ & $\mc{S}_{2r-1}$\\ 
                    & $\mut{2}{2m}$, $m$ odd, $\frac{r}{m}$ odd & $\mc{R}^*$ & $\{c_2(\mc{V}_{i,2r-i})\}_{i=1}^{r-1}$\\ 
                     & $\mut{2}{2m}$, $m$ even, $\frac{r}{m}$ even & $\mc{R}^*[x]$ & $\{c_1(\mc{L}_r), x^2\}\cup \mc{S}_{2r-1}$\\ 
                    & $\mut{2}{2m}$, $m$ odd, $\frac{r}{m}$ even & $\mc{R}^*[x]$ & $\{c_1(\mc{L}_r), x^2+ c_2(\mc{V}_{1,2r-1})\}\cup \mc{S}_{2r-1}$\\ 
                \hline
                ${}^2\mr{D}_{n}$, $n\ge 3$  & $1$ & $\mc{R}^*$ & $\{c_1(\mc{L}_i)\}_{i=1}^{n-2}\cup \{c_2(\mc{V}_{n-1,n})\}$\\
                 & $\mu_2$ & $\mc{R}^*[x]$ & $\{c_1(\mc{L}_i)\}_{i=1}^{n-2}\cup \{x^2+c_2(\mc{V}_{n-1,n})\}$\\
                ${}^2\mr{D}_{2r}$, $r\ge 2$  & $\mut{2}{2,2}$ & $\mc{R}^*$ & $\{c_1(\mc{L}_i)\}_{i=1}^{2r-2}$\\
                ${}^2\mr{D}_{2r+1}$, $r\ge 1$  & $\mut{2}{4}$ & $\mc{R}^*$ & $\{c_1(\mc{L}_1)^2\}\cup \{c_1(\mc{L}_{2i})\}_{i=1}^{r-1} \cup$ \\
                & & & $\cup \{c_1(\mc{L}_{2i+1})+c_1(\mc{L}_1)\}_{i=1}^{r-1}$\\
                \hline
                ${}^3\mr{D}_{4}$ & $1$, $\mut{3}{2,2}$ & $\mc{R}^*$ & $\{c_1(\mc{L}_2), c_3(\mc{V}_{1,3,4})\}$\\
                \hline
                 ${}^2\mr{E}_{6}$ & $1$, $\mut{2}{3}$ & $\mc{R}^*$ & $\{c_1(\mc{L}_2),c_1(\mc{L}_4),c_2(\mc{V}_{1,6}),c_2(\mc{V}_{3,5})\}$
		\end{longtable*}
		\end{center}       
        Here $\mc{S}_{2r-1}= \{c_2(\mc{V}_{2i,2r-2i})\}_{i=1}^{\lfloor(r-1)/2\rfloor} \cup 
			\{c_2(\mc{V}_{2i+1,2r-2i-1})+c_2(\mc{V}_{1,2r-1})\}_{i=1}^{\lfloor r/2\rfloor-1}$. 
        The isomorphisms are induced by the pullback homomorphisms $
		\mc{R}^*=\CH_{K}^*(G/B) \to \CH_{K}^*(G)
		$
		for the projections $G\to G/B$ and by $x\mapsto c_1(\mc{L}(\bar{\varpi}_m))$ and $x\mapsto c_1(\mc{L}(\bar{\varpi}_{n-1}))$ in the cases $(\Att_{2r-1},\mut{2}{2m})$ and $({}^2\mr{D}_{n},\mu_2)$ respectively.

	\end{theorem}
	\begin{proof}
        Let $\tilde{G}\to G$ be the simply connected cover and $\tilde{T}\le \tilde{B}$ be the maximal torus and the Borel subgroups lying over $T\le B$. Theorem~\ref{thm:char_qs} yields an isomorphism
		\begin{equation} \label{eq:compute}
			\CH_{K}^*(G)\cong (\mc{R}^*\otimes \CH_{K,\pi_1(G)}^*(\Spec k))/ I
		\end{equation}
		where $I$ denotes the ideal generated by $\hat{c}(\CH_{K,\tilde{T}}^{>0}(\Spec k))$ for the extended characteristic map
		\begin{multline*}
			\hat{c}\colon \CH_{K,\tilde{T}}^{*}(\Spec k) \xrightarrow{\Phi_\nu} \CH_{K,\tilde{T}\times \pi_1(G)}^{*}(\Spec k) \xrightarrow{\simeq} \CH_{K,\tilde{T}}^{*}(\Spec k) \otimes \CH_{K,\pi_1(G)}^*(\Spec k) \to \\ \xrightarrow{\tilde{c}\otimes \id} \CH_{K}^*(\tilde{G}/\tilde{B})\otimes \CH_{K,\pi_1(G)}^*(\Spec k) = \mc{R}^* \otimes \CH_{K,\pi_1(G)}^*(\Spec k) ,
		\end{multline*}
		where $\nu\colon \tilde{T}\times \pi_1(G)\to \tilde{T}$ is the homomorphism given by $(t,s)\mapsto t\cdot s^{-1}$.	The torus $\tilde{T}$ is quasi-trivial and decomposes into a product of the Weil restrictions of tori corresponding to the orbits of the Galois action on $\Delta(G)$.
		The rings $\CH_{K,\tilde{T}}^{*}(\Spec k)$ and $\CH_{K,\pi_1(G)}^*(\Spec k)$ were computed in Propositions~\ref{prop:classifying_quasi-trivial} and~\ref{prop:equivariant_coef_nl} respectively. The claim of the theorem then follows from an explicit computation of the homomorphism $\hat{c}$ plugged into isomorphism~\refbr{eq:compute}, we provide below the details of these computations for each $(\Delta(G),\pi_1(G))$ case by case. Put
		\[
		L_i:=L(\varpi_i),\quad V_{i,j}:= V(\{\varpi_i,\varpi_j\}),\quad V_{i,j,r}:= V(\{\varpi_i,\varpi_j,\varpi_r\})
		\]
		for the respective representations of $\tilde{T}$ in the notation of~\refbr{expos:rep_of_weights} and note that $\tilde{c}(c_1(L_i)) = c_1(\mc{L}_i)$, $\tilde{c}(c_2(V_{i,j})) = c_2(\mc{V}_{i,j})$ and $\tilde{c}(c_3(V_{i,j,r})) = c_3(\mc{V}_{i,j,r})$, so it suffices to compute $\Phi_\nu$. In order to do this we use the formulae from~\refbr{expos:simple_groups} for the homomorphism of lattices $\mc{X}^*(\tilde{T})\to \mc{X}^*(\pi_1(G))$ dual to the embedding $\pi_1(G)\le \tilde{T}$. We denote by $L\subseteq K$ the splitting field of $\tilde{T}$ which is a Galois extension $L/k$ with $[L:k]=2$ for $\Delta(G)=\Att_n$, ${}^2\mr{D}_n$ or ${}^2\mr{E}_6$, $[L:k]=3$ for $\Delta(G)={}^3\mr{D}_4$ and $[L:k]=6$ for $\Delta(G)={}^6\mr{D}_4$.

  		$\mathbf{{}^6{D}_4}$. We have $[L:k]=6$. Lemma~\ref{lem:normed_coef} yields $\CH_{K}^*(\Spec k)=0$ and $\CH_{K}^*(X)=0$ for every $X\in\Smk$.
		
		$\mathbf{{}^2A_{2r}}$. For $1\le i\le 2r$ we have $\tau(\varpi_{i})=\varpi_{2r+1-i}$ for the nontrivial element $\tau\in\Gal(L/k)$, and
		\begin{equation}
			\CH_{K,\tilde{T}}^{*}(\Spec k)\cong \CH_{K}^{*}(\Spec k)[b_1,b_2, \hdots, b_{r}],\quad b_i\mapsto c_2(V_{i,2r+1-i}).
		\end{equation}	
		The order of $\pi_1(G)$ is odd, thus $\CH_{K,\pi_1(G)}^*(\Spec k)\cong \CH_{K}^*(\Spec k)$ by Lemmas~\ref{lem:normed_coef} and~\ref{lem:trivial_coprime} and the claim follows.	
		
		$\mathbf{{}^2A_{2r-1}}$. For the nontrivial element $\tau\in\Gal(L/k)$ we have 
        \[
        \tau(\varpi_r)=\varpi_r,\quad \tau(\varpi_{i})=\varpi_{2r-i},\, 1\le i\le 2r-1,\, i\neq r.
        \]
        It follows that
		\begin{equation}\label{eq:torus_an}
			\CH_{K,\tilde{T}}^{*}(\Spec k)\cong \CH_{K}^{*}(\Spec k)[y,b_1,\hdots, b_{r-1}],\quad y\mapsto c_1(L_r),\, b_i\mapsto c_2(V_{i,2r-i}).
		\end{equation}
		If $\pi_1(G)=\mut{2}{2m+1}$, then $\CH_{K,\mut{2}{2m+1}}^*(\Spec k)\cong \CH_{K}^*(\Spec k)$ by Lemmas~\ref{lem:normed_coef} and~\ref{lem:trivial_coprime} and the claim follows.
		
		Now suppose $\pi_1(G)=\mut{2}{2m}$. We need to compute the homomorphism
		\[
		\Phi_{\nu}\colon \CH_{K,\tilde{T}}^*(\Spec k) \to  \CH_{K,\tilde{T}}^*(\Spec k)\otimes \CH_{K,\mut{2}{2m}}^*(\Spec k)
		\]
		with $\CH_{K,\tilde{T}}^*(\Spec k)$ given by isomorphism~\refbr{eq:torus_an} and
		\[
		\CH_{K,\mut{2}{2m}}^*(\Spec k) \cong   \CH_{K}^{*}(\Spec k)[x,b]/(x^2+m^2\cdot b),\quad x\mapsto c_1({}^2\Lambda_{\mu_{2m}}^\pm),\, b\mapsto c_2({}^2V_{\mu_{2m}}),
		\]
		by Proposition~\ref{prop:equivariant_coef_nl}. The homomorphism
		\[
		\Res_\nu\colon \mathrm{Rep}(\tilde{T})\cong \Z[\mc{X}^*(\tilde{T})]^{\Gal(L/k)} \to \Z[\mc{X}^*(\tilde{T})\oplus \mc{X}^*(\mut{2}{2m})]^{\Gal(L/k)} \cong \mathrm{Rep}(\tilde{T}\times \mut{2}{2m})
		\]
		is induced by the morphism of lattices
		\begin{gather*}
			\mc{X}^*(\nu)\colon \mc{X}^*(\tilde{T})= \bigoplus_{i=1}^{2r-1} \Z\cdot \varpi_i \to 
			\bigoplus_{i=1}^{2r-1} \Z\cdot \varpi_i \oplus \Z/2m\Z \cdot \bar{\varpi}_1 = \mc{X}^*(\tilde{T})\oplus \mc{X}^*(\mut{2}{2m}) ,\\
			\varpi_i\mapsto \varpi_i - i\bar{\varpi}_1,\, 1\le i\le 2r-1.
		\end{gather*}
		Then $\Res_\nu([L_r])=x^{\varpi_r-r\bar{\varpi}_1} = x^{\varpi_r}\cdot (x^{m\bar{\varpi}_1})^{\frac{r}{m}} =[L_r]\otimes [{}^2\Lambda_{\mu_{2m}}^\pm]^{\otimes \frac{r}{m}}$ and
		\[
		\Phi_\nu(y)=y\otimes 1+ \frac{r}{m}\cdot 1\otimes x.
		\]
		For $1\le i \le r-1$ we have
		\[
		\Res_\nu([V_{i,2r-i}])=\Res_\nu(x^{\varpi_i}+x^{\varpi_{2r-i}}) = x^{\varpi_i - i\bar{\varpi}_1}+x^{\varpi_{2r-i}-(2r-i)\bar{\varpi}_1}.
		\]
		This element corresponds to a rank $2$ representation $W$ of $R\times \mut{2}{2m}$, where $R:=R_{L/k} \Gmm$. Restricting $W$ to $R$ and $\mut{2}{2m}$ respectively we obtain $W|_R\cong V_{R}$ and $W|_{(\mut{2}{2m})}\cong \psi^i ({}^2V_{\mu_{2m}})$ for the $i$-th Adams operation $\psi^i$. Then Lemma~\ref{lem:c2_mixed_rep} yields
		\[
		c_2(W)= c_2(V_R)\otimes 1 + i^2\cdot 1\otimes c_2({}^2V_{\mu_{2m}}) \in \CH_{K,R}^*(\Spec k) \otimes \CH_{K,\mut{2}{2m}}^*(\Spec k) \cong \CH_{K,R\times \mut{2}{2m}}^*(\Spec k).
		\]
		Hence we have
		\[
		\Phi_\nu(b_i)=c_2(\Res_\nu([V_{i,2r-i}]))=c_2(V_{i,2r-i})\otimes 1 + i^2\cdot 1\otimes c_2({}^2V_{\mu_{2m}}) = b_i\otimes 1 + i^2\cdot 1\otimes b.
		\]
		Combining all the above we obtain $\CH^*_K(G)\cong \mc{R}^*[x,b]/I$ with 
		\[	
		I=I\left(\left\{x^2+m^2\cdot b,\,c_1(\mc{L}_r)+\frac{r}{m}\cdot x\right\}\cup \left\{c_2(\mc{V}_{2i,2r-2i})\right\}_{i=1}^{[(r-1)/2]} \cup \left\{b+c_2(\mc{V}_{2i+1,2r-2i-1})\right\}_{i=0}^{[r/2]-1}\right).
		\]
		Then $I=I(\left\{x^2+m^2\cdot b,\,c_1(\mc{L}_r)+\frac{r}{m}\cdot x,\,b+c_2(\mc{V}_{1,2r-1})\right\}\cup\mc{S}_{2r-1})$. If $\frac{r}{m}$ is even, then the claim follows immediately, and if $\frac{r}{m}$ is odd, then the claim follows from Lemma~{\hyperref[lem:full_flag_codim2]{\ref*{lem:full_flag_codim2}.(2)}}.
		
		$\mathbf{^2D_n}$.  For the nontrivial element $\tau\in\Gal(L/k)$ we have 
        \[
        \tau(\varpi_i)=\varpi_i,\, 1\le i \le n-2,\quad \tau(\varpi_{n-1})=\varpi_{n},\quad \tau(\varpi_{n})=\varpi_{n-1}.
        \]
        It follows that
		\begin{equation}\label{eq:torus_dn}
			\CH_{K,\tilde{T}}^{*}(\Spec k)\cong \CH_{K}^{*}(\Spec k)[x_1, \hdots, x_{n-2}, b],\quad x_i\mapsto c_1(L_i),\, b\mapsto c_2(V_{n-1,n}).
		\end{equation}
		
		If $\pi_1(G)=1$, then the claim follows immediately.
		
		Suppose $\pi_1(G)=\mu_2$. Then we need to compute the homomorphism
		\[
		\Phi_{\nu}\colon \CH_{K,\tilde{T}}^*(\Spec k) \to  \CH_{K,\tilde{T}}^*(\Spec k)\otimes \CH_{K,\mu_2}^*(\Spec k)
		\]
		with $\CH_{K,\tilde{T}}^*(\Spec k)$ given by isomorphism~\refbr{eq:torus_dn} and
		\[
		\CH_{K,\mu_2}^*(\Spec k) \cong   \CH_{K}^{*}(\Spec k)[x],\quad x\mapsto c_1(\Lambda_{\mu_2}^\pm),
		\]
		by Proposition~\ref{prop:equivariant_coef_nl} (we have $\mut{2}{2}=\mu_2$), where $\Lambda_{\mu_2}^\pm$ is the nontrivial linear representation of $\mu_2$. The restriction homomorphism
		\[
		\Res_\nu\colon \mathrm{Rep}(\tilde{T})\cong \Z[\mc{X}^*(\tilde{T})]^{\Gal(L/k)} \to \Z[\mc{X}^*(\tilde{T})\oplus \mc{X}^*(\mu_2)]^{\Gal(L/k)} \cong \mathrm{Rep}(\tilde{T}\times \mu_2)
		\]
		is induced by the morphism of lattices
		\begin{gather*}
			\mc{X}^*(\nu)\colon \mc{X}^*(\tilde{T})= \bigoplus_{i=1}^n \Z\cdot \varpi_i \to 
			\bigoplus_{i=1}^n \Z\cdot \varpi_i \oplus \Z/2\Z \cdot \bar{\varpi}_{n-1} = \mc{X}^*(\tilde{T})\oplus \mc{X}^*(\mu_2) ,\\
			\varpi_i\mapsto \varpi_i,\, 1\le i\le n-2,\quad \varpi_{n-1}\mapsto \varpi_{n-1}+\bar{\varpi}_{n-1},\quad \varpi_{n}\mapsto \varpi_{n}+\bar{\varpi}_{n-1}.
		\end{gather*}
		Then 
		\begin{gather*}
			\Res_\nu ([L_i]) = [L_i],\quad 1\le i \le n-2,\\
            \begin{multlined}
                \Res_\nu ([V_{n-1,n}]) = \Res_\nu (x^{\varpi_{n-1}}+x^{\varpi_{n}}) = x^{\varpi_{n-1}+\bar{\varpi}_{n-1}} + x^{\varpi_{n}+\bar{\varpi}_{n-1}} 
                = (x^{\varpi_{n-1}} + x^{\varpi_{n}})\cdot x^{\bar{\varpi}_{n-1}}=\\ = [V_{n-1,n}]\otimes [\Lambda_{\mu_2}^{\pm}].
            \end{multlined}
		\end{gather*}
		Hence $\Phi_{\nu}(x_i)=x_i\otimes 1$, $1\le i\le n-2$, and
		\begin{multline*}
			\Phi_{\nu}(b)=c_2(\Res_\nu([V_{n-1,n}]))= c_2(V_{n-1,n}\otimes \Lambda_{\mu_2}^{\pm}) = \\ =
			c_2(V_{n-1,n})\otimes 1 + c_1(V_{n-1,n})\otimes c_1(\Lambda_{\mu_2}^{\pm}) + 1\otimes c_1(\Lambda_{\mu_2}^{\pm})^2 = b\otimes 1 + 1\otimes x^2.
		\end{multline*}
		The claim follows.
		
		Now suppose $(\Delta(G),\pi_1(G))=({}^2\mr{D}_{2r},\mut{2}{2,2})$. We need to compute the homomorphism
		\[
		\Phi_{\nu}\colon \CH_{K,\tilde{T}}^*(\Spec k) \to  \CH_{K,\tilde{T}}^*(\Spec k)\otimes \CH_{K,\mut{2}{2,2}}^*(\Spec k)
		\]
		with $\CH_{K,\tilde{T}}^*(\Spec k)$ given by isomorphism~\refbr{eq:torus_dn} and
		\[
		\CH_{K,\mut{2}{2,2}}^*(\Spec k) \cong   \CH_{K}^{*}(\Spec k)[d],\quad d\mapsto c_2({}^2V_{\mu_{2,2}}),
		\]
		by Proposition~\ref{prop:equivariant_coef_nl}. The restriction homomorphism
		\[
		\Res_\nu\colon \mathrm{Rep}(\tilde{T})\cong \Z[\mc{X}^*(\tilde{T})]^{\Gal(L/k)} \to \Z[\mc{X}^*(\tilde{T})\oplus \mc{X}^*(\mut{2}{2,2})]^{\Gal(L/k)} \cong \mathrm{Rep}(\tilde{T}\times \mut{2}{2,2})
		\]
		is induced by the homomorphism of lattices
		\begin{gather*}
			\mc{X}^*(\nu)\colon \mc{X}^*(\tilde{T})= \bigoplus_{i=1}^{2r} \Z\cdot \varpi_i \to 
			\bigoplus_{i=1}^{2r} \Z\cdot \varpi_i \oplus \Z/2\Z \cdot \bar{\varpi}_{2r-1} \oplus \Z/2\Z \cdot \bar{\varpi}_{2r}= \mc{X}^*(\tilde{T})\oplus \mc{X}^*(\mut{2}{2,2}) ,\\
			\varpi_i\mapsto \varpi_i + i\cdot (\bar{\varpi}_{2r-1}+\bar{\varpi}_{2r}),\, 1\le i\le 2r-2,\quad \varpi_{2r-1}\mapsto \varpi_{2r-1}+\bar{\varpi}_{2r-1},\quad \varpi_{2r}\mapsto \varpi_{2r}+\bar{\varpi}_{2r},
		\end{gather*}
		Then 
		\[
		\Res_\nu ([L_{2i}]) = [L_{2i}],\quad \Res_\nu ([L_{2i-1}]) = [L_{2i-1}]\otimes [\det {}^2V_{\mu_{2,2}}],\quad 1\le 2i \le 2r-2.
		\]
		Hence 
		\[
		\Phi_{\nu}(x_{i})=x_{i}\otimes 1,\quad  1\le i\le 2r-2,
		\]
		since $c_1(\det {}^2V_{\mu_{2,2}})=0$ because $\CH_{K,\mut{2}{2,2}}^1(\Spec k)=0$.
		For $V_{2r-1,2r}$ we have
		\[
		\Res_\nu ([V_{2r-1,2r}]) = \Res_\nu (x^{\varpi_{2r-1}}+x^{\varpi_{2r}}) = x^{\varpi_{2r-1}+\bar{\varpi}_{2r-1}} + x^{\varpi_{2r}+\bar{\varpi}_{2r}}.
		\]
		This element corresponds to a rank $2$ representation $W$ of $R\times \mut{2}{2,2}$, where $R:=R_{L/k}\Gmm$. Restricting the representation $W$ to $R$ and $\mut{2}{2,2}$ respectively we obtain $W|_R\cong V_R$ and $W|_{(\mut{2}{2,2})}\cong {}^2V_{\mu_{2,2}}$. Then Lemma~\ref{lem:c2_mixed_rep} yields
		\[
		c_2(W)=c_2(V_R)\otimes 1 + 1\otimes c_2({}^2V_{\mu_{2,2}}) \in \CH_{K,R}^*(\Spec k)\otimes \CH_{K,\mut{2}{2,2}}^*(\Spec k) \cong \CH_{K,R\times \mut{2}{2,2}}^*(\Spec k).
		\]
		Hence we have
		\[
		\Phi_\nu(b)=c_2(\Res_\nu([V_{2r-1,2r}]))=c_2(V_{2r-1,2r})\otimes 1 + 1\otimes c_2({}^2V_{\mu_{2,2}}) = b\otimes 1 + 1\otimes d.
		\]
		The claim follows.
		
		Now suppose $(\Delta(G),\pi_1(G))=({}^2\mr{D}_{2r+1},\mut{2}{4})$. We need to compute the homomorphism
		\[
		\Phi_{\nu}\colon \CH_{K,\tilde{T}}^*(\Spec k) \to  \CH_{K,\tilde{T}}^*(\Spec k)\otimes \CH_{K,\mut{2}{4}}^*(\Spec k)
		\]
		with $\CH_{K,\tilde{T}}^*(\Spec k)$ given by isomorphism~\refbr{eq:torus_dn} and
		\[
		\CH_{K,\mut{2}{4}}^*(\Spec k) \cong   \CH_{K}^{*}(\Spec k)[x,d]/(x^2),\quad x\mapsto c_1({}^2\Lambda_{\mu_4}^\pm),\quad d\mapsto c_2({}^2V_{\mu_4})
		\]
		by Proposition~\ref{prop:equivariant_coef_nl}. The restriction homomorphism
		\[
		\Res_\nu\colon \mathrm{Rep}(\tilde{T})\cong \Z[\mc{X}^*(\tilde{T})]^{\Gal(L/k)} \to \Z[\mc{X}^*(\tilde{T})\oplus \mc{X}^*(\mut{2}{4})]^{\Gal(L/k)} \cong \mathrm{Rep}(\tilde{T}\times \mut{2}{4})
		\]
		is induced by the morphism of lattices
		\begin{gather*}
			\mc{X}^*(\nu)\colon \mc{X}^*(\tilde{T})= \bigoplus_{i=1}^{2r+1} \Z\cdot \varpi_i \to 
			\bigoplus_{i=1}^{2r+1} \Z\cdot \varpi_i \oplus \Z/4\Z \cdot \bar{\varpi}_{2r} = \mc{X}^*(\tilde{T})\oplus \mc{X}^*(\mut{2}{4}) ,\\
			\varpi_i\mapsto \varpi_i + 2i\bar{\varpi}_{2r},\, 1\le i\le 2r-1,\quad \varpi_{2r}\mapsto \varpi_{2r}-\bar{\varpi}_{2r},\quad \varpi_{2r+1}\mapsto \varpi_{2r+1}+\bar{\varpi}_{2r}.
		\end{gather*}
		Then 
		\[
		\Res_\nu ([L_{2i}]) = [L_{2i}],\, 1\le 2i \le 2r-2,
		\quad \Res_\nu ([L_{2i-1}]) = [L_{2i-1}]\otimes [{}^2\Lambda_{\mu_4}^{\pm}],\, 1\le 2i-1 \le 2r-1,
		\]
		and it follows that
		\[
		\Phi_{\nu}(x_{2i})=x_{2i}\otimes 1,\, 1\le 2i \le 2r-2,\quad
		\Phi_{\nu}(x_{2i-1})=x_{2i-1}\otimes 1+1\otimes x,\, 1\le 2i-1 \le 2r-1,.
		\]
		For $V_{2r,2r+1}$ we have
		\[
		\Res_\nu ([V_{2r,2r+1}]) = \Res_\nu (x^{\varpi_{2r}}+x^{\varpi_{2r+1}}) = x^{\varpi_{2r}+3\bar{\varpi}_{2r}} + x^{\varpi_{2r+1}+\bar{\varpi}_{2r}}.
		\]
		This element corresponds to a rank $2$ representation $W$ of $R\times \mut{2}{4}$, where $R:=R_{L/k}\Gmm$. Restricting the representation $W$ to $R$ and $\mut{2}{4}$ respectively we obtain $W|_R\cong V_R$ and $W|_{(\mut{2}{4})}\cong {}^2V_{\mu_4}$. Then Lemma~\ref{lem:c2_mixed_rep} yields
		\[
		c_2(W)=c_2(V_R)\otimes 1 + 1\otimes c_2({}^2V_{\mu_4}) \in \CH_{K,R}^*(\Spec k)\otimes \CH_{K,\mut{2}{4}}^*(\Spec k) \cong \CH_{K,R\times \mut{2}{4}}^*(\Spec k).
		\]
		Hence we have
		\[
		\Phi_\nu(b)=c_2(\Res_\nu([V_{2r,2r+1}]))=c_2(V_{2r,2r+1})\otimes 1 + 1\otimes c_2({}^2V_{\mu_4}) = b\otimes 1 + 1\otimes d
		\]
		and the claim follows.	
		
		$\mathbf{^3D_4}$ and $\mathbf{^2E_6}$. The order of $\pi_1(G)$ is coprime to $[L:k]$ in both cases, and Lemmas~\ref{lem:normed_coef} and~\ref{lem:trivial_coprime} yield $\CH_{K,\pi_1(G)}^*(\Spec k)\cong \CH_{K}^*(\Spec k)$. Thus
		\[
		\CH_{K}^*(G)\cong \CH_{K}^*(\tilde{G})\cong  \mc{R}^*/(\tilde{c}(\CH_{K,\tilde{T}}^{>0}(\Spec k))).
		\]
		We have
		\[
		\CH_{K,\tilde{T}}^{*}(\Spec k)\cong \CH_{K}^{*}(\Spec k)[x,w],\quad  x\mapsto c_1(L_2),\, w\mapsto c_3(V_{1,3,4}),
		\]
		in the case of $^3\mr{D}_4$ while in the case of ${}^2\mr{E}_6$
		\begin{gather*}
			\CH_{K,\tilde{T}}^{*}(\Spec k)\cong \CH_{K}^{*}(\Spec k)[x_1,x_2,b_1,b_2],\\
			x_1\mapsto c_1(L_2),\,x_2\mapsto c_1(L_4),\, b_1\mapsto c_2(V_{1,6}),\, b_2\mapsto c_2(V_{3,5}).
		\end{gather*}
		The claim follows.
	\end{proof}
	
	\begin{remark}
	One can extend Theorem~\ref{thm:group_via_cover} also to the split case in a straightforward way, obtaining analogous formulae. We do not do this because in the split case the characteristic sequence is exact and supplies us with all the necessary information.
	\end{remark}
	
	\begin{remark} \label{rem:pullback_conormed_group_1}
		Let $G$ be a simply connected quasi-split simple group over a field $k$ of type $\Delta(G)=\Att_{2r-1}$ and let $K/k$ be the splitting field of $G$. Let $m_1\mid m_2\mid r$ be integers with $\frac{r}{m_2}$ and $\frac{m_2}{m_1}$ being even. Consider the pullback homomorphism
		\[
		f^*\colon \CH_{K}^*(G/\mut{2}{2m_2}) \to \CH_{K}^*(G/\mut{2}{2m_1})
		\]
        for the quotient morphism $f\colon G/\mut{2}{2m_1} \to G/\mut{2}{2m_2}$. Then although these rings may be isomorphic (i.e. if $m_1$ is even), $f^*$ fails to be an isomorphism, since $f^*(x)=0$ in the notation of Theorem~\ref{thm:group_via_cover} by Lemma~\ref{lem:Pic_pull_zero}. 
	\end{remark}

        \subsection{Conormed Chow ring of a quasi-split group $G$ via $\mc{C}^*_K(G)$}

        In this section we show how to recover $\CH^*_K(G)$ out of the cokernel $\mc{C}^*_K(G)$ of the conormed characteristic map. If $G$ is either simply connected or adjoint, then the conormed characteristic sequence sequence is exact and $\CH^*_K(G)\cong \mc{C}^*_K(G)$, otherwise one has to add one generator to recover the conormed Picard group $\CH^1_K(G)$ and also one relation on the square of this element, with the answer being  $\CH^*_K(G)\cong \mc{C}^*_K(G)[x]/(x^2+a)$ for a certain element $a\in \mc{C}^*_K(G)$. We continue to use the standard notation for quasi-split groups and vector bundles recalled in Sections~\ref{sec:quasi-split_simple} and~\ref{sec:vector_bundles}.
 
	\begin{definition} \label{def:Theta_coker}
		In the notation of Definition~\ref{def:characteristic_map} one has $(\phi^*\circ c)(\CH_{K,T}^{>0}(\Spec k))=0$, e.g. by the construction of the Eilenberg--Moore homomorphism~\refbr{eq:EM2}. We put
		\[
		\mc{C}^*_{K}(G):=\CH_{K}^*(G/B)/\left( c(\CH_{K,T}^{>0}(\Spec k))\cdot \CH_{K}^*(G/B) \right)
		\]
		and denote by
		\[
		\Theta_G\colon \mc{C}^*_{K}(G)\to \CH_{K}^*(G)
		\]
		the induced homomorphism.
	\end{definition}
	
	\begin{theorem} \label{thm:cokernel}
		Let $G$ be a quasi-split group over a field $k$, let $T\le B\le G$ be a maximal torus and a Borel subgroup, and let $K/k$ be a separable extension of fields such that $T_K$ is split. Then the following holds.
		\begin{enumerate}
			\item If $T$ is quasi-trivial, then 
			\[
			\Theta_G\colon \mc{C}^*_{K}(G)\to \CH_{K}^*(G)
			\]		
			is an isomorphism. In particular, this applies to a split $G$ and, by Lemma~\ref{lem:quasi-split_quasi-trivial}, to a simply connected or an adjoint quasi-split semisimple group $G$.
			
			\item If $(\Delta(G),\pi_1(G))=(\Att_n,\mut{2}{l})$ with $l$ or $\frac{n+1}{l}$ being odd, then 
			\[
			\Theta_G\colon \mc{C}^*_{K}(G)\to \CH_{K}^*(G)
			\]		
			is an isomorphism.
			
			\item If  $(\Delta(G),\pi_1(G))=(\Att_{2r-1},\mut{2}{2m})$ with $\frac{r}{m}$ being even, then  $\Theta_G$ induces an isomorphism
			\[
			\mc{C}^*_{K}(G)[x]/(x^2+m^2\cdot \bar{c}_2(\mc{V}_{1,2r-1}))\xrightarrow{\simeq} \CH_{K}^*(G),\quad x\mapsto c_1(\mc{L}(\bar{\varpi}_m)),
			\]		
                where $\bar{c}_2(\mc{V}_{1,2r-1}) \in \mc{C}^*_{K}(G)$ is the image of $c_2(\mc{V}_{1,2r-1})\in \CH_{K}^*(G/B)$ under the quotient homomorphism.
			
			\item If $(\Delta(G),\pi_1(G))=({}^2\mr{D}_n,\mu_2)$, then  $\Theta_G$ induces an isomorphism
			\[
			\mc{C}^*_{K}(G)[x]/(x^2+\bar{c}_2(\mc{V}_{n-1,n}))\xrightarrow{\simeq} \CH_{K}^*(G),\quad x\mapsto c_1(\mc{L}(\bar{\varpi}_{n-1})),
			\]
			where $\bar{c}_2(\mc{V}_{n-1,n})\in \mc{C}^*_{K}(G)$ is the image of $c_2(\mc{V}_{n-1,n})\in \CH_{K}^*(G/B)$ under the quotient homomorphism.
		\end{enumerate}
	\end{theorem}
	\begin{proof}
		(1) follows from Theorem~{\hyperref[thm:char_qs]{\ref*{thm:char_qs}.(1)}}.
		
		In cases (2)-(4) the considered homomorphisms are well-defined and surjective by the formulae obtained in Theorem~\ref{thm:group_via_cover}, and it is sufficient to show that the generators of the ideals given in Theorem~\ref{thm:group_via_cover} that do not contain terms with $x^2$, belong to the ideal $I(c)$ generated by the image of the characteristic map 
		\[
		c\colon \CH_{K,T}^{>0}(\Spec k) \to  \CH_{K}^*(G/B).
		\]
		Let $f\colon \tilde{G}\to G$ be the simply connected cover and $g\colon G\to \bar{G}$ be the adjoint quotient, put
		\[
		\tilde{T}:=f^{-1}(T),\quad \tilde{B}:=f^{-1}(B),\quad  \bar{T}:=g(T),\quad \bar{B}:=g(B).
		\]
		The following diagram commutes.
		\begin{equation} \begin{aligned}
				\xymatrix{ 
					\CH^*_{K,\bar{T}}(\Spec k) \ar[r]^{\Phi_g} \ar[d]_{\bar{c}} & \CH^*_{K,T}(\Spec k) \ar[r]^{\Phi_f} \ar[d]^{c} & \CH^*_{K,\tilde{T}}(\Spec k) \ar[d]^{\tilde{c}}\\
					\CH_{K}^*(\bar{G}/\bar{B}) \ar[r]^= & \CH_{K}^*(G/B) \ar[r]^= & \CH_{K}^*(\tilde{G}/\tilde{B})
				}\label{eq:diagram_adj}
			\end{aligned}
		\end{equation}
		Here $\bar{c}$, $c$ and $\tilde{c}$ are the characteristic maps, $\Phi_g$ and $\Phi_f$ are induced by the homomorphisms $g\colon T\to \bar{T}$ and  $f\colon \tilde{T}\to T$ respectively, and the bottom maps are induced by the isomorphism $\tilde{G}/\tilde{B}\xrightarrow{\simeq} G/B\xrightarrow{\simeq} \bar{G}/\bar{B}$. Below we use the same notation as in Theorem~\ref{thm:group_via_cover}.
		
		(2), $n=2r$. We need to show that $c_2(\mc{V}_{i,2r+1-i})\in I(c)$ for $1\le i\le r$. It follows from the formulae obtained in Theorem~\ref{thm:group_via_cover} in the adjoint case (i.e. $l=2r+1$) and by the exactness of the characteristic sequence in the adjoint case proved in Theorem~\ref{thm:char_qs} that $c_2(\mc{V}_{i,2r+1-i})
		\in I(\bar{c})$, $1\le i\le r$. We have $I(\bar{c})\subseteq I(c)$ by the commutativity of Diagram~\ref{eq:diagram_adj}. The claim follows.
		
		(2), $n=2r-1$, $l=2m$, $\frac{r}{m}$ is odd. The claim follows from the adjoint case as above.
		
		(2), $n=2r-1$, $l=2m+1$. We need to show that $c_1(\mc{L}_r)\in I(c)$ and $c_2(\mc{V}_{i,2r-i})\in I(c)$ for $1\le i\le r-1$. Using the adjoint case as above we obtain $c_2(\mc{V}_{2i,2r-2i})\in I(c)$ for $1\le i\le \lfloor \frac{r-1}{2}\rfloor$ and  $c_2(\mc{V}_{2i+1,2r-2i-1})+c_2(\mc{V}_{1,2r-1})\in I(c)$ for $1\le i\le \lfloor \frac{r}{2}\rfloor -1$. It suffices to show $c_1(\mc{L}_r)\in I(c)$ and either $c_2(\mc{V}_{2m+1,2r-2m-1})\in I(c)$ if $2m+1\neq r$ or $c_2(\mc{V}_{1,2r-1})\in I(c)$ if $2m+1=r$. Since $2m+1\mid 2r$, it follows that $2m+1\mid r$ and $\varpi_r\in \mc{X}^*(T)$, thus $c_1(L_r)\in \CH_{K,\tilde{T}}^*(\Spec k)$ belongs to the image of $\Phi_f$. Now suppose $2m+1\neq r$, then furthermore $\varpi_{2m+1},\varpi_{2r-2m-1}\in \mc{X}^*(T)$, thus $c_2(V_{2m+1,2r-2m-1})\in \CH_{K,\tilde{T}}^*(\Spec k)$ also belongs to the image of $\Phi_f$. The claim follows from the commutativity of diagram~\refbr{eq:diagram_adj}, since $c_1(\mc{L}_r)= \tilde{c}(c_1(L_r))$ and $c_2(\mc{V}_{2m+1,2r-2m-1})= \tilde{c}(c_2(V_{2m+1,2r-2m-1}))$. If $2m+1=r$ then the same reasoning as above yields $c_1(\mc{L}_r)\in I(c)$, while $c_2(\mc{V}_{1,2r-1})\in I(c)$ immediately follows from the adjoint case.
		
		(3) We need to show that $c_1(\mc{L}_r)\in I(c)$ and $\mc{S}_{2r-1}\subseteq I(c)$. It follows from the adjoint case as above that $\mc{S}_{2r-1}\subseteq I(c)$. Since $\frac{r}{m}$ is even, it follows that $2m\mid r$ and $\varpi_r\in \mc{X}^*(T)$, thus $c_1(\mc{L}_r)$ belongs to the image of $\Phi_f$. The claim follows as above.
		
		(4) We need to show that $c_1(\mc{L}_i)\in I(c)$, $1\le i \le n-2$. We have $\varpi_i\in \mc{X}^*(T)$, $1\le i\le n-2$, thus $c_1(L_i)$ belongs to the image of $\Phi_f$. The claim follows as above.
	\end{proof}
	
	\subsection{Conormed Chow ring of a quasi-split group: the answer}
        In this section we combine the previous computations and results from \cite{SZ22} to compute the conormed Chow ring $\CH_K^*(G)$ of a non-split quasi-split simple group assuming $K$ to be the splitting field. If the type of $G$ is $\Delta(G)={}^6\mr{D}_4$, then $\CH^*_K(G)=0$, while in the other types it follows from Proposition~\ref{prop:Kunneth_conormed} and a Borel's result on the structure of Hopf algebras \cite[Theorem 7.11 and Proposition 7.8]{MM65} that 
        \[
        \CH_K^*(G)\cong \FF_p[e_1,e_2,\hdots,e_s]/(e_1^{p^{k_1}},e_2^{p^{k_2}},\hdots,e_s^{p^{k_s}}),\quad \deg e_i=d_i,\quad p=[K:k],
	\]
	for some parameters $s,d_i,k_i$, and we give a table for them depending on $(\Delta(G),\pi_1(G))$.
 
 We continue to use the notation from Sections~\ref{sec:multiplicative_groups},~\ref{sec:quasi-split_simple} and~\ref{sec:vector_bundles}.
	
	\begin{lemma} \label{lem:fullflag_cod2_D}
		Let $G$ be a quasi-split simple group over a field $k$ of type $\Delta(G)={}^2\mathrm{D}_n$, let $B\le G$ be a Borel subgroup and let $K/k$ be the splitting field of $G$. Then in $\CH^*_K(G/B)$ we have the following:
		\begin{itemize}
			\item 
			$\CH^1_K(G/B)=\bigoplus_{i=1}^{n-2} \FF_2 \cdot c_1(\mc{L}_i)$,
			\item
			$\CH^2_K(G/B)=\CH^1_K(G/B) \cdot \CH^1_K(G/B) + \FF_2 \cdot c_2(\mc{V}_{n-1,n})$.
		\end{itemize}
	\end{lemma}
	\begin{proof}
		We have $[K:k]=2$, thus $\CH^*_K(\Spec k)\cong \FF_2$ by Lemma~\ref{lem:normed_coef}. Lemma~\ref{lem:Chow_quasi-split_homogeneous} yields that 
		\[
		\CH^*_K(G/B) \cong (\CH^*((G/B)_K))^{\Gal(K/k)} / (\id+\tau)(\CH^*((G/B)_K))
		\]
	with $\tau\in \Gal(K/k)$ being the non-trivial element.	Since $G_K$ is split, $\CH^*((G/B)_K)$ has a basis given by the Schubert cycles $Z_{w}\in \CH^{l(w)}((G/B)_K)$, $w\in W(G_K)$, \cite[Proposition~8]{Che94}, here $l(w)$ is the length of $w$. Put $s_i$ to be the reflection for the simple root $\alpha_i$. The involution $\tau$ acts on the Weyl group $W(G_K)$ via the action on the generators $\tau(s_i)=s_i$, $1\le i \le n-2$, $\tau(s_{n-1})=s_n$, $\tau(s_{n})=s_{n-1}$, which yields a permutation action on the Schubert cycles compatible with the action of $\tau$ on $\CH^*_K(G/B)$. In particular, we have the following $\Gal(K/k)$-invariant decompositions
		\[
		\CH^1((G/B)_K) = \left(\bigoplus_{i=1}^{n-2} \Z \cdot Z_{s_i}\right) \oplus \left[ \Z\cdot Z_{s_{n-1}} \oplus \Z\cdot Z_{s_{n}} \right],
		\]
		\begin{multline*}	
			\CH^2((G/B)_K) = \left(\bigoplus_{1\le i<j\le n-2} \Z \cdot Z_{s_i s_j}\right) \oplus \left(\bigoplus_{i=1}^{n-3} \Z \cdot Z_{s_{i+1} s_i}\right) \oplus \left(\bigoplus_{i=1}^{n-2} \left[\Z \cdot Z_{s_i s_{n-1}}\oplus \Z \cdot Z_{s_i s_{n}}\right]\right) \oplus \\ \oplus \left[\Z \cdot Z_{s_{n-1}s_{n-2}}\oplus \Z \cdot Z_{s_{n}s_{n-2}}\right] \oplus \Z \cdot Z_{s_{n-1}s_{n}}.
		\end{multline*}
		Here $\tau$ acts on the modules in the square brackets permuting the elements of the bases, and $\tau$ acts trivially on all the other summands. Then 
		\begin{gather*}
			\CH^1_K(G/B)= \bigoplus_{i=1}^{n-2} \FF_2 \cdot \bar{Z}_{s_i},\\
			\CH^2_K(G/B) = \left(\bigoplus_{1\le i<j\le n-2} \FF_2 \cdot \bar{Z}_{s_i s_j}\right) \oplus \left(\bigoplus_{i=1}^{n-3} \FF_2 \cdot \bar{Z}_{s_{i+1} s_i}\right) \oplus \FF_2 \cdot \bar{Z}_{s_{n-1}s_{n}}.
		\end{gather*}	
		In $\CH^*((G/B)_K)$ we have $Z_{s_i}=c_1(\mc{L}_i)$, so the first claim of the lemma follows. For the second claim note that the Chevalley's formula \cite[Proposition~10]{Che94} yields
		\[
		Z_{s_i} \cdot Z_{s_j}=\begin{cases}
			Z_{s_is_j}, & \text{$i\neq j$ and $\alpha_i$ is not connected to $\alpha_j$ in the Dynkin diagram $\mr{D}_n$}\\
			Z_{s_is_j}+Z_{s_js_i}, & \text{$i\neq j$ and $\alpha_i$ is connected to $\alpha_j$ in the Dynkin diagram $\mr{D}_n$}\\
			\sum\limits_{(\alpha_l,\alpha_i)< 0} Z_{s_ls_i}, & \text{$i=j$, the sum is taken over all $l$ such that $\alpha_l$ is connected to $\alpha_i$.} 
		\end{cases}
		\]
		In particular, we have 
		\[
		Z_{s_{n-1}s_{n}}=Z_{s_{n-1}}\cdot Z_{s_{n}}= c_1(\mc{L}_{n-1})\cdot c_1(\mc{L}_{n}) = c_2(\mc{V}_{n-1,n})
		\]
		in $\CH^*((G/B)_K)$. Hence $c_2(\mc{V}_{n-1,n})$ coincides with $\bar{Z}_{s_{n-1}s_{n}}$ in $\CH^*_K(G/B)$. The claim follows.
	\end{proof}
	
	\begin{theorem} \label{thm:conormed_group_answer}
		Let $G$ be a non-split quasi-split simple group over a field $k$ and let $K/k$ be the splitting field of $G$. Suppose $\Delta(G)\neq {}^6\mathrm{D}_4$ and put $p:=[K:k]$. Then
		\[
		\CH_K^*(G)\cong \FF_p[e_1,e_2,\hdots,e_s]/(e_1^{p^{k_1}},e_2^{p^{k_2}},\hdots,e_s^{p^{k_s}}),\quad \deg e_i=d_i,
		\]
		with the parameters $s,d_i,k_i$ given by the following table.
		\begin{center}
			\def\arraystretch{1.5}
			\begin{longtable}{l|l|l|l|l} 
				\caption{$\CH^*_K(G)$ for a non-split quasi-split simple group $G$}
				\label{tab:conormed}\\
				$\Delta(G)$ & $\pi_1(G)$ & $s$ & $d_i$, $i=1,2,\ldots, s$ & $k_i$, $i=1,2,\ldots, s$ \\
				\hline
				$\Att_{n}$, $n\ge 2$ & $\mut{2}{l}$, $l\mid (n+1)$, $l$ is odd &  $[\frac{n}{2}]$ & $2i+1$ & $1$ \\
				
				$\Att_{2r-1}$, $r\ge 2$ & $\mut{2}{2m}$, $m\mid r$, $m$ is odd &  $r$ & $2i-1$ & $v_2(r)+1$, $i=1$ \\
				& & &  & $1$, $i\ge 2$ \\
				
				$\Att_{2r-1}$, $r\ge 2$ & $\mut{2}{2m}$, $m\mid r$, $m$ is even &  $r+1$ & $2$, $i=1$ & $v_2(r)$, $i=1$ \\
				& &  & $2i-3$, $i\ge 2$ & $1$, $i\ge 2$ \\		
				
				\hline
				${}^2\mr{D}_{n}$, $n\ge 3$ & $1$ &  $[\frac{n+1}{2}]-1$ & $2i+1$ & $[\log_2 \frac{2n}{2i+1}]$ \\						
				
				& $\mu_2$ &  $[\frac{n+1}{2}]$ & $1$, $i=1$ & $[\log_2 n]+1$ \\						
				&  &    & $2i-1$, $i\ge 2$ & $[\log_2 \frac{2n}{2i-1}]$ \\								
				
				${}^2\mr{D}_{2r+1}$, $r\ge 1$	& $\mut{2}{4}$ &  $r+2$ & $1$, $i=1$ & $1$, $i=1$ \\			
				& &  & $2$, $i=2$ & $[\log_2 (2r+1)]$, $i=2$ \\					
				&  &  & $2i-3$, $i\ge 3$ & $[\log_2 \frac{4r+2}{2i-3}]$, $i\ge 3$ 		\\
				
				${}^2\mr{D}_{2r}$, $r\ge 2$	& $\mut{2}{2,2}$ &  $r$ & $2$, $i=1$ & $[\log_2 2r]$, $i=1$ \\			
				&  &  & $2i-1$, $i\ge 2$ & $[\log_2 \frac{4r}{2i-1}]$, $i\ge 2$ 		\\
				
				\hline
				${}^3\mr{D}_4$ & $1$, $\mut{3}{2,2}$&  $1$ & $4$ & $1$\\
				\hline
				${}^2\mr{E}_6$ & $1$, $\mut{2}{3}$&  $3$ & $3,5,9$ & $1,1,1$ 
			\end{longtable}
		\end{center}	
		Here $v_2$ is the $2$-adic valuation.
	\end{theorem}
	\begin{proof}
        Let $\tilde{G}\to G$ and $G\to \bar{G}$ be the simply connected cover and the adjoint quotient respectively. We have $\CH_{K}^*(\bar{G})\cong \mc{C}^*_K(\bar{G})$ by Theorem~\ref{thm:cokernel} and the last ring was determined in \cite[Section~8]{SZ22} (see especially Table~5 of loc. cit.). Furthermore, it follows from Theorem~\ref{thm:group_via_cover} that $\CH^*_K(G)\cong \CH^*_K(\bar{G})$, if $(\Delta(G),\pi_1(G))=(\Att_n,\mut{2}{l})$ with $\frac{n+1}{l}$ being odd,  $\Delta(G)={}^3\mr{D}_4$, and $\Delta(G)={}^2\mr{E}_6$. Then the remaining cases are $(\Att_{2r-1},\mut{2}{l})$ with $\frac{2r}{l}$ being even, $({}^2\mr{D}_n,1)$ and $({}^2\mr{D}_n,\mu_2)$. In each case we use Theorem~\ref{thm:group_via_cover} to derive the answers from the ones for the respective adjoint groups. Below we adopt the notation of Theorem~\ref{thm:group_via_cover} and repeatedly use its claims without additional references.
		
		\textbf{(1)} $(\Delta(G),\pi_1(G))=(\Att_{2r-1},\mut{2}{l})$ with $\frac{2r}{l}$ being even. Suppose that $r$ is odd. We have
		\begin{equation} \label{eq:conormed_ch_adj_a1}
			\mc{R}^*/I(\{c_2(\mc{V}_{i,2r-i})\}_{i=1}^{r-1})\cong \CH^*_K(\bar{G})\cong \FF_2 [e_2,\hdots,e_{r+1}]/(e_2^2,\hdots,e_{r+1}^2)
		\end{equation}
		with $\deg e_i=2i-3$, $i\ge 2$. The element $c_1(\mc{L}_r)\in \mc{R}^*$ is the generator of $\mc{R}^1$ by Lemma~\ref{lem:full_flag_codim2}, thus its image in the quotient ring on the left-hand side of the formula corresponds under isomorphisms~\refbr{eq:conormed_ch_adj_a1} to $e_2$ on the right-hand side, which is a unique non-trivial element of degree $1$. Since $\frac{2r}{l}$ is even, it follows that $l$ is odd and
		\begin{multline*}
			\CH^*_K(G)\cong (\mc{R}^*/I(\{c_2(\mc{V}_{i,2r-i})\}_{i=1}^{r-1}))/(\bar{c}_1(\mc{L}_r))\cong (\FF_2 [e_2,\hdots,e_{r+1}]/(e_2^2,\hdots,e_{r+1}^2))/(e_2)\cong \\ \cong \FF_2[e_3,\hdots,e_{r+1}]/(e_3^2,\hdots,e_{r+1}^2).
		\end{multline*}
		
		Suppose $r$ is even, then
		\begin{equation} \label{eq:conormed_ch_adj_a2}
			\mc{R}^*/I(\mc{S}_{2r-1})\cong \CH^*_K(\bar{G})\cong \FF_2 [e_1,e_2,\hdots,e_{r+1}]/(e_1^{2^{v_2(r)}},e_2^2,\hdots,e_{r+1}^2)
		\end{equation}
		with $\deg e_1=2$, $\deg e_i=2i-3$, $i\ge 2$. As above, the image of the element $c_1(\mc{L}_r)\in \mc{R}^*$ in the quotient ring on the left-hand side of the formula corresponds under isomorphisms~\refbr{eq:conormed_ch_adj_a2} to $e_2$ on the right-hand side. Furthermore, Lemma~\ref{lem:full_flag_codim2} yields that the image of $c_2(\mc{V}_{1,2r-1})$ in the quotient ring on the left-hand side of the formula is the generator of the degree $2$ part, thus it is mapped to $e_1$ under isomorphisms~\refbr{eq:conormed_ch_adj_a2}. If $l$ is odd, then
		\begin{multline*}
			\CH^*_K(G)\cong (\mc{R}^*/I(\mc{S}_{2r-1}))/(\bar{c}_1(\mc{L}_r),\bar{c}_2(\mc{V}_{1,2r-1}))\cong \\\cong (\FF_2 [e_1,e_2,\hdots,e_{r+1}]/(e_1^{2^{v_2(r)}},e_2^2,\hdots,e_{r+1}^2))/(e_1,e_2)\cong \FF_2 [e_3,\hdots,e_{r+1}]/(e_3^2,\hdots,e_{r+1}^2).
		\end{multline*}
		If $l=2m$ and $m$ is even, then
		\begin{multline*}
			\CH^*_K(G)\cong (\mc{R}^*/I(\mc{S}_{2r-1}))[x]/(\bar{c}_1(\mc{L}_r),x^2)\cong \\ \cong (\FF_2[e_1,\hdots,e_{r+1}]/(e_1^{2^{v_2(r)}},e_2^2,\hdots,e_{r+1}^2))[x]/(e_2,x^2)\cong  \FF_2 [e_1,\hdots,e_{r+1}]/(e_1^{2^{v_2(r)}},e_2^2,\hdots,e_{r+1}^2).
		\end{multline*}
		If $l=2m$ and $m$ is odd, then
		\begin{multline*}
			\CH^*_K(G)\cong (\mc{R}^*/I(\mc{S}_{2r-1}))[x]/(\bar{c}_1(\mc{L}_r),x^2+\bar{c}_2(\mc{V}_{1,2r-1}))\cong \\ \cong (\FF_2 [e_1,\hdots,e_{r+1}]/(e_1^{2^{v_2(r)}},e_2^2,\hdots,e_{r+1}^2))[x]/(e_2,x^2+e_1)\cong  \\
			\cong \FF_2 [x,e_3,\hdots,e_{r+1}]/(x^{2^{v_2(r)+1}},e_3^2\hdots,e_{r+1}^2),
		\end{multline*}
		where $\deg x=1$, $\deg e_i=2i-3$.
		
		\textbf{(2)} $\Delta(G)={}^2\mr{D}_{n}$. Suppose that $n=2r+1$ and put 
            \[
            I:=I(\{c_1(\mc{L}_1)^2\}\cup\{c_1(\mc{L}_{2i})\}_{i=1}^{r-1} \cup \{c_1(\mc{L}_{2i+1})+ c_1(\mc{L}_1)\}_{i=1}^{r-1}) \subseteq \mc{R}^*.
            \]
            We have
		\begin{equation}\label{eq:conormed_ch_adj_d1}
			\mc{R}^*/I
			\cong \CH^*_K(\bar{G})
			\cong \FF_2 [e_1,\hdots,e_{r+2}]/(e_1^{2},e_2^{2^{k_2}},\hdots,e_{r+2}^{2^{k_{r+2}}})
		\end{equation}
		with $\deg e_1=1$, $\deg e_2=2$, $\deg e_i=2i-3$, $i\ge 3$, and $k_i$ as in the statement of the theorem for the adjoint case. The image of the element $c_1(\mc{L}_1)\in \mc{R}^*$ in the quotient ring on the left-hand side of the formula is the generator in degree $1$ by Lemma~\ref{lem:fullflag_cod2_D}, thus it corresponds under isomorphisms~\refbr{eq:conormed_ch_adj_d1} to $e_1$ on the right-hand side, which is a unique non-trivial element of degree $1$. Then, since $e_1^2=0$ on the right-hand side, it follows that $\bar{c}_1(\mc{L}_1)^2=0$ on the left-hand side. Thus it follows from Lemma~\ref{lem:fullflag_cod2_D} and the above that the image of the element $c_2(\mc{V}_{2r,2r+1})\in \mc{R}^*$ in the quotient ring on the left-hand side of the formula is the generator in degree $2$, hence it corresponds under isomorphisms~\refbr{eq:conormed_ch_adj_d1} to $e_2$ on the right-hand side, which is a unique non-trivial element of degree $2$. Then we have
		\begin{multline*}
			\CH_K^*(\tilde{G}/\mu_2)\cong \left(\mc{R}^*/I\right)[x] /(c_1(\mc{L}_1), x^2+c_2(\mc{V}_{2r,2r+1})) \cong \\
			\cong (\FF_2 [e_1,\hdots,e_{r+2}]/(e_1^{2},e_2^{2^{k_2}},\hdots,e_{r+2}^{2^{k_{r+2}}}))[x]/(e_1, x^2+e_2) \cong \\
			\cong 
			\FF_2 [x,e_3,\hdots,e_{r+2}]/(x^{2^{k_2+1}},e_{3}^{2^{k_{3}}},\hdots,e_{r+2}^{2^{k_{r+2}}})
		\end{multline*}
		with $\deg x=1$. Similarly,
		\begin{multline*}
			\CH_K^*(\tilde{G})\cong \left(\mc{R}^*/I\right)/(c_1(\mc{L}_1), c_2(\mc{V}_{2r,2r+1})) \cong \\
			\cong (\FF_2 [e_1,\hdots,e_{r+2}]/(e_1^{2},e_2^{2^{k_2}},\hdots,e_{r+2}^{2^{k_{r+2}}}))/(e_1, e_2)
			\cong 
			\FF_2 [e_3,\hdots,e_{r+2}]/(e_{3}^{2^{k_{3}}},\hdots,e_{r+2}^{2^{k_{r+2}}}).
		\end{multline*}	
		
		Suppose $n=2r$. We have
		\begin{equation}\label{eq:conormed_ch_adj_d2}
			\mc{R}^*/I(\{c_1(\mc{L}_i)\}_{i=1}^{2r-2})
			\cong \CH^*_K(\bar{G})
			\cong \FF_2 [e_1,\hdots,e_r]/(e_1^{2^{k_1}},\hdots,e_{r}^{2^{k_{r}}})
		\end{equation}
		with $\deg e_1=2$, $\deg e_i=2i-1$, $i\ge 2$, and $k_i$ as in the statement of the theorem for the adjoint case. The degree $1$ part on the right-hand side of the formula is trivial, thus the same holds on the left-hand side. It follows from Lemma~\ref{lem:fullflag_cod2_D} that the image of the element $c_2(\mc{V}_{2r-1,2r})\in \mc{R}^*$ in the quotient ring on the left-hand side of the formula is the generator in degree $2$, so it corresponds under isomorphisms~\refbr{eq:conormed_ch_adj_d2} to $e_1$ on the right-hand side, which is a unique non-trivial element of degree $2$. Then we have
		\begin{multline*}
			\CH_K^*(\tilde{G}/\mu_2)\cong \left(\mc{R}^*/I(\{c_1(\mc{L}_i)\}_{i=1}^{2r-2})\right)[x] /(x^2+c_2(\mc{V}_{2r-1,2r})) \cong \\
			\cong (\FF_2 [e_1,\hdots,e_{r}]/(e_1^{2^{k_1}},\hdots,e_{r}^{2^{k_{r}}}))[x]/(x^2+e_1)
			\cong 
			\FF_2 [x,e_2,\hdots,e_{r}]/(x^{2^{k_1+1}},e_{2}^{2^{k_{2}}},\hdots,e_{r}^{2^{k_{r}}})
		\end{multline*}
		with $\deg x=1$. Similarly,
		\begin{multline*}
			\CH_K^*(\tilde{G})\cong \left(\mc{R}^*/I(\{c_1(\mc{L}_i)\}_{i=1}^{2r-2})\right)/(c_2(\mc{V}_{2r-1,2r})) \cong \\
			\cong (\FF_2 [e_1,\hdots,e_{r}]/(e_1^{2^{k_1}},\hdots,e_{r}^{2^{k_{r}}}))/(e_1)
			\cong 
			\FF_2 [e_2,e_3,\hdots,e_{r}]/(e_{2}^{2^{k_{2}}},\hdots,e_{r}^{2^{k_{r}}}).
		\end{multline*}	
	\end{proof}
	
	\begin{remark}
        Let $G$ be a simply connected quasi-split simple group over a field $k$ of type $\Delta(G)=\Att_{2r-1}$ and let $K/k$ be the splitting field of $G$. Let $m_1\mid m_2\mid r$ be integers. Consider the pullback homomorphism
		\[
		f^*\colon \CH_{K}^*(G/\mut{2}{2m_2}) \to \CH_{K}^*(G/\mut{2}{2m_1})
		\]
        for the quotient morphism $f\colon G/\mut{2}{2m_1} \to G/\mut{2}{2m_2}$. It follows from the proof of Theorem~\ref{thm:conormed_group_answer} that $f^*$ is an isomorphism, if $\frac{m_2}{m_1}$ is odd, while for $\frac{m_2}{m_1}$ even one has $f^*(e_2)=0$ for the degree $1$ generator $e_2$ by Lemma~\ref{lem:Pic_pull_zero}.
	\end{remark}

    \subsection{Cokernel in the conormed characteristic sequence}

    In this section we complement the previous results with a computation of the cokernel $\mc{C}^*_K(G)$ in the conormed characteristic sequence. This cokernel coincides with $\CH^*_K(G)$ for the simply connected and adjoint groups, while for intermediate groups it is only a subring of $\CH^*_K(G)$ in general.  We continue to use the notation from Sections~\ref{sec:multiplicative_groups},~\ref{sec:quasi-split_simple} and~\ref{sec:vector_bundles}.
    
    \begin{theorem} \label{thm:cokernel_answer}
		Let $G$ be a non-split quasi-split simple group over a field $k$ and let $K/k$ be the splitting field of $G$. Suppose $\Delta(G)\neq {}^6\mathrm{D}_4$ and put $p:=[K:k]$. Then
		\[
		\mc{C}_K^*(G)\cong \FF_p[e_1,e_2,\hdots,e_s]/(e_1^{p^{k_1}},e_2^{p^{k_2}},\hdots,e_s^{p^{k_s}}),\quad \deg e_i=d_i,
		\]
		with the parameters $s,d_i,k_i$ given by the following table.
		\begin{center}
			\def\arraystretch{1.5}
			\begin{longtable}{l|l|l|l|l} 
				\caption{$\mc{C}^*_K(G)$ for a non-split quasi-split simple group $G$}
				\label{tab:coker}\\
				$\Delta(G)$ & $\pi_1(G)$ & $s$ & $d_i$, $i=1,2,\ldots, s$ & $k_i$, $i=1,2,\ldots, s$ \\
				\hline
				$\Att_{n}$, $n\ge 2$ & $\mut{2}{l}$, $l\,|\, (n+1)$, $l$ is odd & $[\frac{n}{2}]$ & $2i+1$ & $1$ \\

                $\Att_{2r-1}$, $r\ge 2$ odd & $\mut{2}{2m}$, $m\mid r$ &  $r$ & $2i-1$ & $1$ \\
								
				$\Att_{2r-1}$, $r\ge 2$ even & $\mut{2}{2m}$, $m\mid r$, $\frac{r}{m}$ is even &  $r$ & $2$, $i=1$ & $v_2(r)$, $i=1$ \\
				& & & $2i-1$, $i\ge 2$ & $1$, $i\ge 2$ \\		
				
				$\Att_{2r-1}$, $r\ge 2$ even & $\mut{2}{2m}$, $m\mid r$, $\frac{r}{m}$ is odd &  $r+1$ & $2$, $i=1$ & $v_2(r)$, $i=1$ \\
				& & & $2i-3$, $i\ge 2$ & $1$, $i\ge 2$ \\		
				
				\hline
				${}^2\mr{D}_{n}$, $n\ge 3$ & $1$ &  $[\frac{n+1}{2}]-1$ & $2i+1$ & $[\log_2 \frac{2n}{2i+1}]$ \\						
				
				& $\mu_2$&  $[\frac{n+1}{2}]$ & $2$, $i=1$ & $[\log_2 n]$ \\						
				&  &   & $2i-1$, $i\ge 2$ & $[\log_2 \frac{2n}{2i-1}]$ \\								
				
				${}^2\mr{D}_{2r+1}$, $r\ge 1$	& $\mut{2}{4}$ &  $r+2$ & $1$, $i=1$ & $1$, $i=1$ \\			
				& & &  $2$, $i=2$ & $[\log_2 (2r+1)]$, $i=2$ \\					
				&  & & $2i-3$, $i\ge 3$ & $[\log_2 \frac{4r+2}{2i-3}]$, $i\ge 3$ 		\\
				
				${}^2\mr{D}_{2r}$, $r\ge 2$	& $\mut{2}{2,2}$ &  $r$ & $2$, $i=1$ & $[\log_2 2r]$, $i=1$ \\			
				&  & & $2i-1$, $i\ge 2$ & $[\log_2 \frac{4r}{2i-1}]$, $i\ge 2$ 		\\
				
				\hline
				${}^3\mr{D}_4$ & $1$, $\mut{3}{2,2}$&  $1$ & $4$ & $1$\\
				\hline
				${}^2\mr{E}_6$ & $1$, $\mut{2}{3}$& $3$ & $3,5,9$ & $1,1,1$ 
			\end{longtable}
		\end{center}	
    \end{theorem}
 \begin{proof}
     If $G$ is adjoint, then this was computed in \cite[Section~8]{SZ22} (see especially Table~5 of loc. cit.). If $G$ is simply connected, i.e. $\pi_1(G)=1$, or if $(\Delta(G),\pi_1(G))=(\Att_{n},\mut{2}{l})$ with $l\mid n+1$ and either $l$ or $\frac{n+1}{l}$ being odd, then the claim follows from Theorems~\ref{thm:cokernel} and~\ref{thm:conormed_group_answer}. Thus the remaining cases are $(\Delta(G),\pi_1(G))=(\Att_{2r-1},\mut{2}{2m})$ with $m\mid r$ and $\frac{r}{m}$ being even, and $(\Delta(G),\pi_1(G))=({}^2\mr{D}_{n},\mu_2)$. Let $g\colon G\to \bar{G}$ be the adjoint quotient and consider the following commutative diagram.
     \begin{equation} \label{eq:cokernel_group_diagram}
     \begin{gathered}
     \xymatrix{
     \mc{C}^*_K(\bar{G}) \ar[d]_{g_{\mc{C}}^*} \ar[r]_(0.45){\Theta_{\bar{G}}}^(0.45){\simeq} & \CH^*_K(\bar{G}) \ar[d]^{g^*} \\
     \mc{C}^*_K(G) \ar[r]^(0.45){\Theta_{G}} & \CH^*_K(G)
     }
     \end{gathered}
     \end{equation}
     Here the horizontal morphisms are the canonical ones from Definition~\ref{def:Theta_coker}, $g^*$ is the pullback, and $g_{\mc{C}}^*$ is induced by the isomorphism $\CH^*_K(\bar{G}/\bar{B})\xrightarrow{\simeq} \CH^*_K(G/B)$ given by the isomorphism $G/B\xrightarrow{\simeq}\bar{G}/\bar{B}$ with $\bar{B}:=g(B)$. The homomorphism $g_{\mc{C}}^*$ is clearly surjective while the homomorphism $\Theta_{\bar{G}}$ is an isomorphism by Theorem~\ref{thm:cokernel}. Furthermore, Theorem~\ref{thm:cokernel} yields that $\Theta_{G}$ is injective. It follows that 
     \[
     \mc{C}^*_K\cong \Theta_{G}(\mc{C}^*_K)= g^*(\CH_K^*(\bar{G})).
     \]

     $(\Att_{2r-1},\mut{2}{2m})$, $m\mid r$, $\frac{r}{m}$ is odd.  Theorem~\ref{thm:group_via_cover} yields that $g^*$ is an isomorphism, thus $g_{\mc{C}}^*$ is an isomorphism as well.
     
     $(\Att_{2r-1},\mut{2}{2m})$, $m\mid r$, $\frac{r}{m}$ is even. There is a unique nontrivial element $e_2=c_1(\mc{L}(\bar{\varpi}_r))\in \CH^1_K(\bar{G})\cong \FF_2$, and we have $g^*(e_2)=0$ by Lemma~\ref{lem:Pic_pull_zero}. In this case $r$ is even and Theorem~\ref{thm:group_via_cover} yields that $g^*$ induces a monomorphism $\CH_K^*(\bar{G})/(e_2)\to \CH_K^*(G)$. Thus $\mc{C}^*_K(G)\cong \CH_K^*(\bar{G})/(e_2)$ and the claim follows.

     $({}^2\mr{D}_{2r},\mu_2)$. The homomorphism $g^*$ is injective by Theorem~\ref{thm:group_via_cover}, thus $g_{\mc{C}}^*$ is an isomorphism.
     
     $({}^2\mr{D}_{2r+1},\mu_2)$. There is a unique nontrivial element $e_1=c_1(\mc{L}(\bar{\varpi}_1))\in \CH^1_K(\bar{G})\cong \FF_2$, we have $g^*(e_1)=0$ by Lemma~\ref{lem:Pic_pull_zero}, and Theorem~\ref{thm:group_via_cover} yields that $g^*$ induces a monomorphism $\CH_K^*(\bar{G})/(e_1)\to \CH_K^*(G)$. Thus $\mc{C}^*_K(G)\cong \CH_K^*(\bar{G})/(e_1)$ and the claim follows.
 \end{proof}

 \section{Chow ring of a quasi-split simple group}
 In this part of the article we compute the Chow ring $\Ch^*(G)$ for a non-split quasi-split simple group $G$. Here and below we put $\Ch^*(-):=\CH^*(-)\otimes \FF_p$ for a given prime number $p$. The computations proceed differently depending on the relation of $p$ and $G$.
 
 If $p=[K:k]$ for the splitting field $K/k$ of $G$, then we analyse the exact sequence
 \[
 \Ch^*(G_K)\xrightarrow{\pi_*} \Ch^*(G) \to \CH_K^*(G) \to 0,
 \]
 where $\pi\colon G_K\to G$ is the projection. We show that $\pi_*=0$ unless $(\Delta(G),\pi_1(G))=({}^2\mr{D}_{2r},\mut{2}{2,2})$, thus in all but one cases the computation follows from the computation of $\CH_K^*(G)$ carried out before. In the case of $({}^2\mr{D}_{2r},\mut{2}{2,2})$ we perform a detailed analysis of the Chow ring $\Ch^*(\OGr(2r-1;q))$ of a quasi-split submaximal isotropic Grassmannian which yields the computation of $\Ch^*(G)$.

  If $p$ is coprime to $[K:k]$, then we show that $\Ch^*(G)\cong \Ch^*(G_K)^{\Gal(K/k)}$ and explicitly compute the action of the Galois group and the  subalgebra of invariant elements.

In the exceptional cases $\Delta(G)={}^3\mr{D}_4,{}^6\mr{D}_4$ and $\Delta(G)={}^2\mr{E}_6$ we combine both of the above approaches and also use some explicit computer-assisted computations. We use the Maple package \cite{P} which explicitly computes the Chow ring structure of flag varieties in the basis of Schubert classes.


	\subsection{Quasi-split submaximal orthogonal Grassmannian}

    In this section we compute the ring $\Ch^*(\OGr(n-1;q)):=\CH^*(\OGr(n-1;q))\otimes \FF_2$ for the Grassmannian of isotropic subspaces of dimension $n-1$ in a dimension $2n$ vector space equipped with a quadratic form corresponding to a quasi-split simple group $G$ of type $\Delta(G)={}^2\mr{D}_{n}$. For this we use the isomorphism $\Ch^*(\OGr(n-1;q))\cong \Ch^*(\OGr(n-1;q)_K)^{\Gal(K/k)}$, where $K$ is the splitting field of $G$, the calculations done in the split case in \cite[\S~2]{Vi07} and \cite[Chapter~XVI]{EKM08} and compute explicitly the action of the Galois group. We further use the identification of $\OGr(n-1;q)$ as a $G$-homogeneous variety to show that 
    \[
    c_1(\mc{L}(\bar{\varpi}_{n-1}+\bar{\varpi}_{n}))^{2^{v_2(n)}}=0
    \]
     in $\Ch^*(G)$ as in the split case.
     
     We continue to use the notation of Sections~\ref{sec:quasi-split_simple} and~\ref{sec:vector_bundles}, and in this section we put 
     \[
     \Ch^*(-):=\CH^*(-)\otimes \FF_2.
     \]
	
    \begin{proposition} \label{prop:OGr_Chow_ring}
		Let $G$ be a simply connected quasi-split simple group of type $\Delta(G)=\mathrm{D}_{n}$ or $\Delta(G)={}^2\mathrm{D}_{n}$ over a field $k$ with the splitting field $K/k$, and let $T\le B\le G$ be a maximal torus and a Borel subgroup. Consider the parabolic subgroup $B\le P=P_{n-1,n}\le G$ in the notation of~\refbr{expos:parabolic} and the line bundle $\mc{L}:=\mc{L}_P(\varpi_{n-1}+\varpi_{n})$ over $G/P$. Then the following holds.
        \begin{enumerate}
			\item For the adjoint quotient $G\to \bar{G}$ and the projection $\varphi\colon \bar{G}\to G/P$ we have
            \[
            c_1(\mc{L}(\bar{\varpi}_{n-1}+\bar{\varpi}_{n}))^{2^{v_2(n)}} = c_1(\varphi^*\mc{L})^{2^{v_2(n)}}=0
            \]
            in $\Ch^*(\bar{G})$, where $v_2$ is the $2$-adic valuation. 
            \item Suppose $\Delta(G)={}^2\mathrm{D}_{n}$ and put $s:=[\frac{n}{2}]+1$, $k_i:=[\log_2 \frac{2n-1}{2i-3}]$, $2\le i \le s$. Then there is an isomorphism
            \[
                \Theta\colon \FF_2 [\varepsilon_1,\varepsilon_2,\hdots, \varepsilon_{s}]/(\varepsilon_1^n,\varepsilon_2^{2^{k_2}},\hdots, \varepsilon_{s}^{2^{k_{s}}})\xrightarrow{\simeq} \Ch^*((G/P)_K)
            \]
            such that $\Theta(\varepsilon_1)=c_1(\mc{L}_K)$ and for the nontrivial element $\tau\in \Gal(K/k)$ we have
            \[
            \tau(\Theta(\varepsilon_1))=\Theta(\varepsilon_1), \quad\tau(\Theta(\varepsilon_i))=\Theta(\varepsilon_i)+\Theta(\varepsilon_1)^{2i-3},\, 2\le i\le s.          
            \]
        \end{enumerate}
    \end{proposition}
    \begin{proof}
        First we recall the necessary geometric setting, see e.g. \cite[Chapter~XVI]{EKM08} and \cite[\S~2]{Vi07}. Note that \cite[Chapter~XVI]{EKM08} does not have a restriction on the characteristic of the base field, while \cite[\S~2]{Vi07} assumes $\chark k\neq 2$, and this is why we mostly follow the former and appeal to the latter only once, where the proof of the referred statement works verbatim also for $\chark k= 2$.
        
        By \cite[Example~27.10]{KMRT98} we have $G\cong \Spin(V,q)$, where $(V,q)$ is a vector space of dimension $2n$ equipped with a non-degenerate quadratic form. For integers $1\le i_1<i_2<\hdots <i_l\le n$ we denote by $\mr{F}(i_1,i_2,\hdots,i_l;q)$ the variety of totally isotropic flags in $(V,q)$ of the respective dimensions, and put $\OGr(i;q):=\mr{F}(i;q)$. The standard action of $G$ on $V$ yields an isomorphism 
        \[
        \OGr(n-1;q)\cong G/P.
        \]
        The projection 
        \[
        \alpha\colon \mr{F}(n-1,n;q) \to \OGr(n-1;q)
        \]
        is a finite morphism of degree $2$, and there is a $G_K$-equivariant isomorphism 
        \begin{equation} \label{eq:FK_splits}
        \mr{F}(n-1,n;q)_K\cong \OGr(n-1;q)_K \sqcup \OGr(n-1;q)_K \cong (G/P)_K \sqcup (G/P)_K
        \end{equation}
        such that $\alpha_K$ is the obvious projection. The projection 
        \[
        \beta\colon \mr{F}(n-1,n;q) \to \OGr(n;q)
        \]
        is canonically identified with the projective bundle 
        \[
        \beta\colon \PP(\mc{E}^{\vee})\to \OGr(n;q)
        \]
        for the tautological rank $n$ vector bundle $\mc{E}$ over $\OGr(n;q)$. For the tautological line bundle $\struct(-1)$ over $\PP(\mc{E}^{\vee})\cong \mr{F}(n-1,n;q)$ one has 
        \begin{equation} \label{eq:taut_over_Fl}
        \struct(-1)_K \cong \mc{L}_{P_K}(\varpi_n-\varpi_{n-1}) \sqcup \mc{L}_{P_K}(\varpi_{n-1}-\varpi_{n})
        \end{equation}
        under isomorphism~\refbr{eq:FK_splits}. This decomposition immediately follows from a straightforward computation of the action of the torus $T_K$ on the fibers of $\struct(-1)_K$ over the two $P_K$-stable points of $\mr{F}(n-1,n;q)_K$.
        
        \textbf{(1)}
		We claim that in $\Ch^*(G/P)$ for every $i\in \NN_0$ one has
		\begin{equation}\label{eq:tangent_valuation}
			c_i(\mc{L}^{\oplus n}) = c_i(\mc{T}_{G/P})
		\end{equation}
		for the tangent bundle $\mc{T}_{G/P}$. Suppose that equality~\refbr{eq:tangent_valuation} holds and put $N:=2^{v_2(n)}$. Then we have 
		\[
		c_N(\mc{T}_{G/P}) = c_N(\mc{L}^{\oplus n}) = {n \choose N} c_1(\mc{L})^N = c_1(\mc{L})^N
		\]
		with the last equality arising from the fact that ${n \choose N}$ is odd by the choice of $N$. Let $\bar{P}\le \bar{G}$ be the image of $P$ in the adjoint quotient, then $G/P\cong \bar{G}/\bar{P}$. The tangent bundle $T_{\bar{G}/\bar{P}}$ clearly admits the structure of a $\bar{G}$-equivariant vector bundle over $\bar{G}/\bar{P}$, thus for the projection $\varphi\colon \bar{G}\to \bar{G}/\bar{P}\cong G/P$ the vector bundle $\varphi^*(T_{G/P})$ is trivial. Then
		\[
		c_1(\varphi^*\mc{L})^{N} = \varphi^*(c_1(\mc{L})^{N}) = \varphi^*(c_N(\mc{T}_{G/P})) = c_N(\varphi^*\mc{T}_{G/P})=0.
		\]
		
		Now we turn to the proof of the equality~\refbr{eq:tangent_valuation}. The injectivity part of Lemma~\ref{lem:Chow_quasi-split_homogeneous} yields that it suffices to prove the equality after a base change to $K$, so till the end of the proof of claim (1) we assume that $G$ is split, i.e. $K=k$ and the quadratic form $q$ is hyperbolic. By \cite[Example~104.20]{EKM08} we have a short exact sequence of vector bundles
		\[
		0\to \struct_{\PP(\mc{E}^\vee)} \to \struct(1)\otimes \beta^*\mc{E}^\vee \to \mc{T}_\beta \to 0,
		\]
		where $\beta\colon \PP(\mc{E}^\vee)\to \OGr(n;q)$ is the projection and $\mc{T}_\beta$ is the relative tangent bundle that is defined by the short exact sequence
		\[
		0 \to \mc{T}_\beta \to \mc{T}_{\PP(\mc{E}^\vee)} \to \beta^*\mc{T}_{\OGr(n;q)} \to 0.
		\]
		For a vector bundle $\mc{V}$ over $\PP(\mc{E}^\vee)$ put 
		\[
		c_t(\mc{V}):=1+c_1(\mc{V})t+c_2(\mc{V})t^2+\hdots \in \Ch^*(\PP(\mc{E}^\vee)) [[t]]
		\]
		to be the total Chern class mod $2$. Multiplicativity of the total Chern classes with respect to the short exact sequences applied to the sequences above yields
		\[
		c_t(\mc{T}_{\PP(\mc{E}^\vee)}) = c_t(\beta^*\mc{T}_{\OGr(n;q)}) \cdot c_t(\struct(1)\otimes \beta^*\mc{E}^\vee).
		\]
		Furthermore, it follows from \cite[Proposition~86.13 and Exercise~87.8]{EKM08} that $c_t(\beta^*\mc{E}^\vee)=1$ and $ c_t(\beta^*\mc{T}_{\OGr(n;q)})=1$. Hence 
		\[
		c_t(\mc{T}_{\PP(\mc{E}^\vee)}) = c_t(\struct(1)\otimes \beta^*\mc{E}^\vee)= c_t(\struct(1)^{\oplus n})=c_t(\struct(-1)^{\oplus n}).
		\]
		Restricting to a connected component of $\PP(\mc{E}^\vee)$ via decomposition~\refbr{eq:FK_splits} and using~\refbr{eq:taut_over_Fl} we obtain
            \[
		c_i(\mc{T}_{G/P}) = c_i(\mc{L}_P(\varpi_n-\varpi_{n-1})^{\oplus n}).
		\]
	In the mod $2$ Chow groups we have $c_1(\mc{L}_P(\varpi_{n}-\varpi_{n-1}))=c_1(\mc{L}_P(\varpi_{n}+\varpi_{n-1}))=c_1(\mc{L})$, thus
		\[
		c_i(\mc{T}_{G/P}) = c_i(\mc{L}_P(\varpi_{n}-\varpi_{n-1})^{\oplus n}) = c_i(\mc{L}^{\oplus n}).
		\]

        \textbf{(2)} Consider the following diagram with all the morphisms being projections.
		\begin{equation}\label{eq:diagram_z}
            \begin{gathered}
		\xymatrix{
		& \mr{F}(1,n-1;q)_K \ar[dl]_{f_1} \ar[r]^{f_2} & \OGr(n-1;q)_K  &\\
            \OGr(1;q)_K & & & \mr{F}(n-1,n;q)_K \ar[lu]_{\alpha_K} \ar[ld]_{\beta_K} \\
            & \mr{F}(1,n;q)_K \ar[ul]_{g_1} \ar[r]^{g_2} & \OGr(n;q)_K& 
		}
            \end{gathered}
		\end{equation}
		The variety $\OGr(1;q)_K$ is a smooth projective hyperbolic quadric of dimension $2n-2$, thus \cite[Proposition~68.1]{EKM08} yields
		\[
		\Ch^i(\OGr(1;q)_K) = \begin{cases}
			\FF_2 \cdot h^i, & 0 \le i\le n-2,\\
			\FF_2 \cdot h^{n-1}\oplus \FF_2 \cdot l_{n-1}, & i=n-1,\\		 
			\FF_2 \cdot l_{2n-2-i}, & n \le i\le 2n-2,\\		 
		\end{cases}
		\]
		where $h\in \Ch^1(\OGr(1;q)_K)$ is the class of a hyperplane section and $l_{2n-2-i}\in \Ch^i(\OGr(1;q)_K)$ is the class of a totally isotropic projective subspace in $\OGr(1;q)_K$ of dimension $2n-2-i$. Note that there are two different classes $l_{n-1},l_{n-1}'\in \Ch^{n-1}(\OGr(1;q)_K)$ of totally isotropic projective subspaces in $\OGr(1;q)_K$ of dimension $n-1$ and one has $l_{n-1}+l_{n-1}' = h^{n-1}$. Put 
		\begin{gather*}
			z_0:=(g_2)_* g_1^* (l_{n-1})\in \Ch^{0}(\OGr(n;q)_K),\quad z'_0:=(g_2)_* g_1^* (l'_{n-1})\in \Ch^{0}(\OGr(n;q)_K),\\
            z_i:= (g_2)_* g_1^* (l_{n-1-i})\in \Ch^{i}(\OGr(n;q)_K),\quad 1\le i\le n-1,
		\end{gather*}
		then \cite[Proposition~86.16]{EKM08} yields that
		\begin{gather*}
            \FF_2[\varepsilon_0,\varepsilon'_0,\varepsilon_2,\varepsilon_3,\hdots,\varepsilon_{s}]/(\varepsilon_0+\varepsilon'_0-1,\varepsilon_0^2-\varepsilon_0,(\varepsilon'_0)^2-\varepsilon'_0, \varepsilon_0\varepsilon'_0,\varepsilon_2^{2^{k_2}},\hdots,\varepsilon_{s}^{2^{k_{s}}})\xrightarrow{\simeq} \Ch^*(\OGr(n;q)_K),\\
            \varepsilon_0\mapsto z_0,\quad \varepsilon'_0\mapsto z'_0,\quad \varepsilon_i \mapsto z_{2i-3},\, 2\le i\le s,
		\end{gather*}
            is an isomorphism. Applying the projective bundle formula to 
            \[
            \beta\colon \mr{F}(n-1,n;q)_K\cong \PP(\mc{E}^\vee)_K\to \OGr(n;q)_K
            \]
            and using \cite[Proposition~86.13]{EKM08} we obtain an isomorphism
            \begin{gather*}
            \Psi\colon \FF_2[\varepsilon_0,\varepsilon'_0,\varepsilon_1,\varepsilon_2,\hdots,\varepsilon_{s}]/(\varepsilon_0+\varepsilon'_0-1,\varepsilon_0^2-\varepsilon_0,(\varepsilon'_0)^2-\varepsilon'_0, \varepsilon_0\varepsilon'_0,\varepsilon_1^n,\varepsilon_2^{2^{k_2}},\hdots,\varepsilon_{s}^{2^{k_{s}}})\xrightarrow{\simeq} \Ch^*(\mr{F}(n-1,n;q)_K),\\
            \varepsilon_0\mapsto \beta_K^*(z_0),\quad \varepsilon'_0\mapsto \beta_K^*(z'_0),\quad \varepsilon_1\mapsto c_1(\struct(-1)_K),\quad \varepsilon_i \mapsto \beta_K^*(z_{2i-3}),\, 2\le i\le s,
            \end{gather*}		
            where $c_1(\struct(-1)_K)$ is the first Chern class of the tautological line bundle over $\mr{F}(n-1,n;q)_K$.
            
            In view of decomposition~\refbr{eq:FK_splits} the composition
            \begin{multline*}
            \Ch^*((G/P)_K)\cong \Ch^*(\OGr(n-1;q)_K)\xrightarrow{\alpha_K^*} \Ch^*(\mr{F}(n-1,n;q)_K) \xrightarrow{\Psi^{-1}}\\
            \xrightarrow{\Psi^{-1}}
            \FF_2[\varepsilon_0,\varepsilon'_0,\varepsilon_1,\hdots,\varepsilon_{s}]/(\varepsilon_0+\varepsilon'_0-1,\varepsilon_0^2-\varepsilon_0,(\varepsilon'_0)^2-\varepsilon'_0, \varepsilon_0\varepsilon'_0,\varepsilon_1^n,\varepsilon_2^{2^{k_2}},\hdots,\varepsilon_{s}^{2^{k_{s}}}) \xrightarrow{\Pi} \\
            \xrightarrow{\Pi}
            \FF_2[\varepsilon_1,\varepsilon_2,\hdots,\varepsilon_{s}]/(\varepsilon_1^n,\varepsilon_2^{2^{k_2}},\hdots,\varepsilon_{s}^{2^{k_{s}}})
            \end{multline*}
            with the projection $\Pi$ given by $\Pi(\varepsilon_0)=1$, $\Pi(\varepsilon'_0)=0$, is an isomorphism. Denote by
            \[
            \Theta\colon \FF_2[\varepsilon_1,\varepsilon_2,\hdots,\varepsilon_{s}]/(\varepsilon_1^n,\varepsilon_2^{2^{k_2}},\hdots,\varepsilon_{s}^{2^{k_{s}}}) \to \Ch^*((G/P)_K)
            \]
            the inverse isomorphism. Since
            \[
            c_1(\mc{L}_K) = c_1(\mc{L}_{P_K}(\varpi_n-\varpi_{n-1})) =c_1(\mc{L}_{P_K}(\varpi_{n-1}-\varpi_{n}))
            \]
            in $\Ch^*((G/P)_K)$, it follows from  decomposition~\refbr{eq:taut_over_Fl} that $\alpha_K^*(c_1(\mc{L}_K))=c_1(\struct(-1)_K)$ and
            \[
            \Theta(\varepsilon_1)=c_1(\mc{L}_K).
            \]
            Furthermore, for the nontrivial element $\tau\in \Gal(K/k)$ we have
            \[
            \tau(\Theta(\varepsilon_1))=\tau(c_1(\mc{L}_K))=c_1(\mc{L}_K) = \Theta(\varepsilon_1),
            \]
            since $c_1(\mc{L}_K)$ is obtained by a base change from $k$ to $K$.

            For the computation of the action of $\Gal(K/k)$ on $\Theta(\varepsilon_i)$, $2\le i\le s$, note that all the morphisms in diagram~\refbr{eq:diagram_z} and the isomorphism $(G/P)_K\cong \OGr(n-1;q)_K$ are given by a base change of the corresponding morphisms over $k$, thus the action of $\Gal(K/k)$ commutes with the respective pullbacks and pushforwards. For the nontrivial element $\tau\in \Gal(K/k)$ we have
            \[
            \tau(l_{n-1}) = l'_{n-1},\quad \tau(l'_{n-1}) = l_{n-1},\quad \tau(l_{n-1-i}) = l_{n-1-i},\, 1\le i\le n-1,
            \]
            with the latter equality arising from the fact that $\Ch^{n-1+i}(\OGr(1;q)_K)\cong \FF_2$, $1\le i\le n-1$, and $\FF_2$ admits only a trivial action of $\Gal(K/k)$. Thus
            \[
            \tau(\beta_K^*(z_0)) = \beta_K^*(z'_0),\quad \tau(\beta_K^*(z_i)) = \beta_K^*(z_i),\, 1\le i\le n-1.
            \]
            Furthermore, put $x:=c_1(\struct(-1)_K)\in \Ch^1(\mr{F}(n-1,n;q)_K)$, then we have $\tau(x)=x$, since $x$ is obtained by a base change from $k$. Consider the elements
		\begin{gather*}
			y_1:=(f_2)_* f_1^* (l_{n-1})\in \Ch^{1}(\OGr(n-1;q)_K),\quad y'_1:=(f_2)_* f_1^* (l'_{n-1})\in \Ch^{1}(\OGr(n-1;q)_K),\\
            y_i:= (f_2)_* f_1^* (l_{n-i})\in \Ch^{i}(\OGr(n-1;q)_K),\quad 2\le i\le n-1.
		\end{gather*}
            It follows from \cite[Lemma~2.6]{Vi07} (loc. cit. has a blanket assumption that $\chark k\neq 2$. However, the same proof also works for $\chark k=2$, since the same description as an incidence divisor in $\mr{F}(1,n;q)\times_{\OGr(n;q)} \mr{F}(n-1,n;q)$ can be applied to $\mr{F}(1,n-1,n;q)$) that
            \[
            \alpha_K^*(y'_1) = \beta_K^* (z_1) + x \cdot \beta_K^*(z'_{0}),\quad \alpha_K^*(y_i) = \beta_K^* (z_i) + x\cdot \beta_K^*(z_{i-1}), \, 1\le i\le n-1.
            \]
            In particular, we have $x=\alpha_K^*(y_1+y'_1)$. Then
            \[
            \beta_K^*(z_i)+x^i\cdot \beta_K^*(z'_0) =  \sum_{j=0}^{i-2} x^j \cdot \alpha_K^*(y_{i-j}) + x^{i-1} \cdot \alpha_K^*(y'_1) = \alpha_K^*(w_i),\quad w_i:=\sum_{j=0}^{i-2} (y_1+y'_1)^j y_{i-j} + (y_1+y'_1)^{i-1} y'_1,
            \]
            where $w_i\in \Ch^*(\OGr(n-1;q)_K)$. Thus for $2\le i\le s$ we have
            \begin{gather*}
            \Theta^{-1}(w_{2i-3}) = \Pi \circ \Psi^{-1} \circ \alpha_K^*(w_{2i-3}) = \Pi \circ \Psi^{-1} (\beta_K^*(z_{2i-3})+x^{2i-3}\cdot \beta_K^*(z'_0)) = \varepsilon_{i},\\
            \begin{multlined}
            \Theta^{-1}(\tau(w_{2i-3})) = \Pi \circ \Psi^{-1} \circ \alpha_K^*(\tau(w_{2i-3})) = \Pi \circ \Psi^{-1} \circ \tau (\beta_K^*(z_{2i-3})+x^{2i-3}\cdot \beta_K^*(z'_0)) = \\ = \Pi \circ \Psi^{-1} (\beta_K^*(z_{2i-3})+x^{2i-3}\cdot \beta_K^*(z_0)) = \varepsilon_i + \varepsilon_1^{2i-3}.
            \end{multlined}
            \end{gather*}
            It follows that $\tau(\Theta(\varepsilon_i)) = \Theta(\varepsilon_i) + \Theta(\varepsilon_1)^{2i-3}$ for $2\le i\le s$.
    \end{proof}

    	\begin{lemma} \label{lem:invariants_OGr}
		For $n\ge 2$ put 
  \[
  s:=\left[\frac{n}{2}\right]+1,\quad k_1:=v_2(n),\quad k_i:=\left[\log_2 \frac{2n-1}{2i-3}\right],\,2\le i \le s,
  \]
  where $v_2$ is the $2$-adic valuation. Consider the algebra
  \[
		\mc{P}:=\mc{P}_{n}:= \FF_2 [\varepsilon_1,\varepsilon_2,\hdots, \varepsilon_s]/(\varepsilon_1^n,\varepsilon_2^{2^{k_2}},\hdots, \varepsilon_s^{2^{k_s}}),
		\]
        denote by $\bar{\varepsilon}_i$ the image of $\varepsilon_{i}$ in the quotient algebra $\mc{P}$, let $\tau \colon \mc{P}\to \mc{P}$ be the involution given by $\tau(\bar{\varepsilon}_1)=\bar{\varepsilon}_1$, $\tau(\bar{\varepsilon}_i)=\bar{\varepsilon}_i+\bar{\varepsilon}_1^{2i-3}$, $2\le i\le s$, and put 
		\[
		\mc{P}^\tau :=\{x\in \mc{P}\mid \tau(x)=x\}\subseteq \mc{P}
		\]
	to be the subalgebra of $\tau$-invariant elements.     For $1\le m\le n$ consider the sequence
		\[
		\mc{P}/(\bar{\varepsilon}_1^m) \xrightarrow{\id + \tau} \mc{P}^\tau /(\bar{\varepsilon}_1^m)\xrightarrow{J} \mc{P}/(\bar{\varepsilon}_1^m),
		\]
		where $J$ is induced by the inclusion $\mc{P}^\tau \subseteq \mc{P}$. Then the following holds.
		\begin{enumerate}
			\item $(\id + \tau)(\mc{P}/(\bar{\varepsilon}_1^m))=\bar{\varepsilon}_1\cdot (\mc{P}^\tau /(\bar{\varepsilon}_1^m))$.		
			\item $J$ restricted to $(\id + \tau)(\mc{P}/(\bar{\varepsilon}_1^m))$ is injective.
            \item There is an isomorphism
            \begin{gather*}
            \mc{P}^\tau /(\bar{\varepsilon}_1^{m}) \cong \FF_2[e_1,e_2,\hdots,e_s,x]/(e_1^m, e_1\cdot x, x^2, h_2,h_3,\hdots,h_s),\\
            h_2:= e_1^{2^{k_2}-n}\cdot x+ \sum_{j=0}^{k_2-1} e_1^{2^{k_2}-2^{j+1}} \cdot e_2^{2^{j}},
            \\
            h_i:=e_i^{2^{k_i}} +  e_1^{(2i-3)\cdot 2^{k_i}-n}\cdot x+e_1^{(2i-4)\cdot 2^{k_i}}\cdot \left(\sum_{j=0}^{k_i-1} e_1^{2^{k_i}-2^{j+1}} \cdot e_2^{2^{j}}\right),\quad 3\le i\le s,
            \end{gather*}
            given by $e_1\mapsto \bar{\varepsilon}_1$, $x\mapsto \bar{\varepsilon}_1^{n-1}\bar{\varepsilon}_2$, $e_2\mapsto \bar{\varepsilon}_2^2+\bar{\varepsilon}_1\bar{\varepsilon}_2$, $e_i\mapsto \bar{\varepsilon}_i+\bar{\varepsilon}_1^{2i-4}\bar{\varepsilon}_2$, $3\le i\le s$.
            In particular, we have
            \[
            \mc{P}^\tau /(\bar{\varepsilon}_1^{2^{k_1}}) \cong
            \begin{cases}
            \FF_2[e_2,\hdots,e_s,x]/(e_2^{2^{k'_2}},\hdots, e_s^{2^{k'_s}}, x^2), & \text{$n$ odd},\\
            \FF_2[e_1,\hdots,e_s]/(e_1^{2^{k_1'}}, e_2^{2^{k'_2}},\hdots, e_s^{2^{k'_s}}, e_1\cdot e_l^{2^{k_l'-1}}), & \text{$n$ even, $n\neq 2^{k_1}$},\\
            \FF_2[e_1,\hdots,e_s]/\left(e_1^{2^{k_1'}}, e_2^{2^{k'_2}},\hdots, e_s^{2^{k'_s}}, e_1\cdot \left(\sum_{j=0}^{k'_2-1} e_1^{2^{k_2'}-2^{j+1}} \cdot e_2^{2^{j}}\right)\right), & n=2^{k_1},\\            
            \end{cases}
            \]
            where $l\in\NN$ is such that $n=2^{k_1}(2l-3)$ and
            \[
            k'_i=
            \begin{cases}
                k_i, & \text{$i\neq 2$, $l$},\\
                k_i-1, & i=2\neq l,\\
                k_i+1, & i=l\neq 2,\\
                k_i, & \text{$i=l=2$}.
            \end{cases}
            \]
            Note that $k_1'=v_2(n)$, $k'_2=[\log_2 n]$ and $k'_i=[\log_2 \frac{2n}{2i-3}]$, $3\le i\le s$.
	\end{enumerate}
	\end{lemma}	
	\begin{proof}
        Put 
        \[
        \bar{e}_1:=\bar{\varepsilon}_1,\quad \bar{e}_2:=\bar{\varepsilon}_2^2+\bar{\varepsilon}_1 \bar{\varepsilon}_2\in \mc{P}^\tau,\quad \bar{e}_i:=\bar{\varepsilon}_i+\bar{\varepsilon}_1^{2i-4}\bar{\varepsilon}_2\in \mc{P}^\tau, \, 3\le i\le s.
        \]
        The monomials 
        \[
        \bar{\varepsilon}_1^{i_1}\bar{\varepsilon}_2^{i_2}\hdots \bar{\varepsilon}_s^{i_s},\quad 0\le i_1<n,\, 0\le i_j< 2^{k_j},\, 2\le j\le s,
        \]
        form a basis of $\mc{P}$. Inducting on the multidegree $(i_si_{s-1}\hdots i_1)$ of the monomial $\bar{\varepsilon}_1^{i_1}\bar{\varepsilon}_2^{i_2}\hdots \bar{\varepsilon}_s^{i_s}$ in the lexigraphical order it is straightforward to see that $\bar{\varepsilon}_1^{i_1}\bar{\varepsilon}_2^{i_2}\hdots \bar{\varepsilon}_s^{i_s} \in \mc{P}^\tau + \mc{P}^\tau \bar{\varepsilon}_2$. It follows that $1$ and $\bar{\varepsilon}_2$ generate $\mc{P}$ as a module over $\mc{P}^\tau$. The morphism $\id + \tau$ is clearly $\mc{P}^\tau$-linear and we have
		\[
		(\id+\tau)(1)=0,\quad (\id + \tau)(\bar{\varepsilon}_2)=\bar{\varepsilon}_1,
		\]
		so the first claim follows.
		
	      Put $\bar{x}:=\bar{\varepsilon}_1^{n-1}\bar{\varepsilon}_2$. Again inducting on the multidegree one obtains that the set
        \begin{equation}\label{eq:basis_Ptau}
            \bar{e}_1^{i_1}\bar{e}_2^{i_2}\hdots \bar{e}_s^{i_s},\quad \bar{e}_2^{i_2}\hdots \bar{e}_s^{i_s} \bar{x},\quad 0\le i_1<n,\, 0\le i_2<2^{k_2-1},\, 0\le i_j< 2^{k_j},\, 3\le j\le s,
        \end{equation} 
        generates $\mc{P}^\tau$ as a vector space over $\FF_2$. Furthermore, these elements are linearly independent in $\mc{P}$, thus~\refbr{eq:basis_Ptau} is a basis of $\mc{P}^\tau$. We have $\bar{\varepsilon}_1\bar{x}=\bar{\varepsilon}_1^n\bar{\varepsilon}_2=0$, $\bar{\varepsilon}_1\bar{e}_1^{i_1}=\bar{e}_1^{i_1+1}$, $\bar{e}_1^n=0$, thus multiplication by $\bar{\varepsilon}_1$ takes the elements of this basis either to $0$ or to distinct elements of the basis. It follows that
        \begin{equation} \label{eq:basis_Ptau_modm}
            \bar{e}_1^{i_1}\bar{e}_2^{i_2}\hdots \bar{e}_s^{i_s},\quad \bar{e}_2^{i_2}\hdots \bar{e}_s^{i_s}\bar{x},\quad 0\le i_1<m,\, 0\le i_2<2^{k_2-1},\, 0\le i_j< 2^{k_j},\, 3\le j\le s,
        \end{equation}
        form a basis of $\mc{P}^\tau/(\bar{\varepsilon}_1^m)$. Then $(\id+\tau)(\mc{P}/(\bar{\varepsilon}_1^m)) = \bar{\varepsilon}_1\cdot (\mc{P}^\tau/(\bar{\varepsilon}_1^m))$ is generated by the elements
        \[
            \bar{e}_1^{i_1}\bar{e}_2^{i_2}\hdots \bar{e}_s^{i_s}, \quad 1\le i_1<m,\, 0\le i_2<2^{k_2-1},\, 0\le i_j< 2^{k_j},\, 3\le j\le s,
        \]
        whose images under $J$ are clearly linearly independent. This proves the second claim.

        For the third claim we first check the multiplicative relations between $\bar{e}_i$ and $\bar{x}$. The relation $\bar{e}_1^m=0$ is clear, the relations $\bar{e}_1\cdot \bar{x}=0$ and $\bar{x}^2=0$ hold in $\mc{P}$, thus also in $\mc{P}^\tau$ and $\mc{P}^\tau/(\bar{\varepsilon}_1^m)$. The relations $\bar{h}_i=0$, where $\bar{h}_i$ is the image of $h_i$ in $\mc{P}$, also hold already in $\mc{P}$, thus in $\mc{P}^\tau$ and $\mc{P}^\tau/(\bar{\varepsilon}_1^m)$, which could be seen expanding 
        \[
        \bar{e}_2^{2^{k_2-1}}=(\bar{\varepsilon}_2^2+\bar{\varepsilon}_1\bar{\varepsilon}_2)^{2^{k_2-1}}=\bar{\varepsilon}_1^{2^{k_2-1}}\bar{\varepsilon}_2^{2^{k_2-1}},\quad 
        \bar{e}_i^{2^{k_i}}=(\bar{\varepsilon}_i+\bar{\varepsilon}_1^{2i-4}\bar{\varepsilon}_2)^{2^{k_i}}=\bar{\varepsilon}_1^{(2i-4)\cdot{2^{k_i}}}\bar{\varepsilon}_2^{2^{k_i}},\quad 3\le i\le s,
        \]
        and substituting $\bar{e}_1=\bar{\varepsilon}_1$, $\bar{e}_2=\bar{\varepsilon}_2^2+\bar{\varepsilon}_1\bar{\varepsilon}_2$ and $\bar{x}=\bar{\varepsilon}_1^{n-1}\bar{\varepsilon}_2$ into the formula defining $\bar{h}_i$. Thus we have a homomorphism
        \[
        \FF_2[e_1,e_2,\hdots,e_s,x]/(e_1^m, e_1\cdot x, x^2, h_2,h_3,\hdots,h_s) \to \mc{P}^\tau /(\bar{\varepsilon}_1^{m})
        \]
        induced by $e_i\mapsto \bar{e}_i$ and $x\mapsto \bar{x}$. This homomorphism is clearly surjective, and is an isomorphism by the dimension count using the basis~\refbr{eq:basis_Ptau_modm} on the right-hand side.
        
        Finally, let $m=2^{k_1}$ and let $l\in\NN$ be such that $n=2^{k_1}(2l-3)$. Note that the integers $k_i$, $2\le i\le s$, are defined in such a way that the following inequalities hold:
        \[
        (2i-3)2^{k_i-1}< n \le (2i-3)2^{k_i}.
        \]
        
        First suppose $l\neq 2$. Then we have
        \begin{gather*}
        2^{k_2}-n\ge 2^{k_1},\quad 2^{k_2}-2^{j+1}\ge 2^{k_2-1} \ge 2^{k_1},\, 0\le j\le k_2-2,\\
        (2i-4)\cdot 2^{k_i}\ge 2^{k_1}, \, 3\le i\le s,\quad (2i-3)\cdot 2^{k_i}-n\ge 1,\, i\neq l,
        \end{gather*}
        and it follows that for some $h_i'\in \FF_2[e_1,e_2,\hdots,e_s,x]$ we have
        \[
        h_2=e_2^{2^{k_2-1}}+e_1^{2^{k_1}}\cdot h_2',\quad h_i=e_i^{2^{k_i}} + (e_1\cdot x) \cdot e_1^{(2i-3)\cdot 2^{k_i}-n-1}+ e_1^{2^{k_1}}\cdot h_i',\quad 3\le i \le s,\, i\neq l.
        \]
        If $n$ is odd, then $l>s$, so it follows from the above that
        \[
        \FF_2[e_1,\hdots,e_s,x]/(e_1^{2^{k_1}}, e_1\cdot x, x^2, h_2,h_3,\hdots,h_s) = 
        \FF_2[e_1,\hdots,e_s,x]/(e_1^{2^{k_1}}, e_1\cdot x, x^2, e_2^{2^{k_2-1}},e_3^{2^{k_3}},\hdots,e_s^{2^{k_s}}),
        \]
        since the ideals are the same. Furthermore, we have $k_1=0$ and
        \[
        \FF_2[e_1,\hdots,e_s,x]/(e_1^{2^{k_1}}, e_1\cdot x, x^2, e_2^{2^{k_2-1}},e_3^{2^{k_3}},\hdots,e_s^{2^{k_s}}) \cong 
        \FF_2[e_2,\hdots,e_s,x]/(e_2^{2^{k_2-1}},e_3^{2^{k_3}},\hdots,e_s^{2^{k_s}},x^2).
        \]
        If $n$ is even, then $l\le s$ and we have $h_l=e_l^{2^{k_l}} + x + e_1^{2^{k_1}}\cdot h_l'$, thus
        \begin{multline*}
        \FF_2[e_1,\hdots,e_s,x]/(e_1^{2^{k_1}}, e_1\cdot x, x^2, h_2,h_3,\hdots,h_s) = \\=\FF_2[e_1,\hdots,e_s,x]/(e_1^{2^{k_1}}, e_1\cdot x, x^2, e_2^{2^{k_2-1}},e_3^{2^{k_3}},\hdots, e_{l-1}^{2^{k_{l-1}}}, e_l^{2^{k_l}} + x, e_{l+1}^{2^{k_{l+1}}},\hdots,e_s^{2^{k_s}})\cong
        \\
        \cong 
        \FF_2[e_1,\hdots,e_s]/(e_1^{2^{k_1}}, e_2^{2^{k_2-1}},e_3^{2^{k_3'}},\hdots,e_s^{2^{k_s'}}, e_1\cdot e_l^{2^{k'_l-1}}),
        \end{multline*}
        where $k_i'=k_i$, $3\le i\neq l$, and $k_l'=k_l+1$. The claim follows.
        
        Now suppose $l=2$, i.e. $n=2^{k_1}$, then
        \[
        (2i-3)\cdot 2^{k_i}-n \ge 1,\quad (2i-4)\cdot 2^{k_i} \ge 2^{k_1}, \quad 3\le i\le s,
        \]
        and it follows that for some $h_i'\in \FF_2[e_1,e_2,\hdots,e_s,x]$ we have
        \[
        h_i=e_i^{2^{k_i}} + (e_1\cdot x) \cdot e_1^{(2i-3)\cdot 2^{k_i}-n-1}+ e_1^{2^{k_1}}\cdot h_i',\quad 3\le i \le s.
        \]
        Then we have
        \[
        \FF_2[e_1,\hdots,e_s,x]/(e_1^{2^{k_1}}, e_1\cdot x, x^2, h_2,h_3,\hdots,h_s) = 
        \FF_2[e_1,\hdots,e_s,x]/(e_1^{2^{k_1}}, e_1\cdot x, x^2, h_2, e_3^{2^{k_3}},\hdots,e_s^{2^{k_s}}),
        \]
        since the ideals are the same. For $h_2$ we have
        \[
        h_2=e_2^{2^{k_2-1}}+x+ \sum_{j=0}^{k_2-2} e_1^{2^{k_2}-2^{j+1}} \cdot e_2^{2^{j}}, \quad h^2_2=e_2^{2^{k_2}}+x^2+ \sum_{j=0}^{k_2-2} e_1^{2^{k_2+1}-2^{j+2}} \cdot e_2^{2^{j+1}} = e_2^{2^{k_2}}+x^2+ e_1^{2^{k_1}} \cdot h_2'
        \]
        for some $h_2'\in  \FF_2[e_1,e_2,\hdots,e_s,x]$, since
        \[
        {2^{k_2+1}-2^{j+2}} \ge 2^{k_2} = 2^{k_1},\quad 0\le j \le k_2-2.
        \]
        Then using the equality expressing $x$ via $h_2$, $e_1$ and $e_2$ we obtain
        \begin{multline*}
        \FF_2[e_1,\hdots,e_s,x]/\left(e_1^{2^{k_1}}, e_1\cdot x, x^2, h_2, e_3^{2^{k_3}},\hdots,e_s^{2^{k_s}}\right) \cong \\
        \cong 
        \FF_2[e_1,\hdots,e_s]/\left(e_1^{2^{k_1}},e_2^{2^{k_2}},\hdots,e_s^{2^{k_s}}, e_1\cdot \left(\sum\nolimits_{j=0}^{k_2-1} e_1^{2^{k_2}-2^{j+1}} \cdot e_2^{2^{j}}\right)\right).
        \end{multline*}
	\end{proof}

        \subsection{Chow ring of an adjoint quasi-split simple group of type ${}^2\mr{D}_{2r}$}

    In this section we compute the ring $\Ch^*(G):=\CH^*(G)\otimes \FF_2$ for an adjoint quasi-split simple group of type $\Delta(G)={}^2\mr{D}_{2r}$. For this we use the previous computations with the submaximal isotropic Grassmannian and the computation of the conormed Chow ring $\CH^*_K(G)$ for the splitting field $K$ of $G$. 
    
    We continue to use the notation of Sections~\ref{sec:quasi-split_simple} and~\ref{sec:vector_bundles}. In this section we put
     \[
     \Ch^*(-):=\CH^*(-)\otimes \FF_2.
     \]

	\begin{lemma} \label{lem:surj_special}
		Let $G$ be a split simple group of type $\Delta(G)=\mathrm{D}_n$ over a field $k$ and let $T\le B\le G$ be a maximal torus and a Borel subgroup. Consider the parabolic subgroup $B\le P=P_{n-1,n}\le G$ in the notation of~\refbr{expos:parabolic}. Then the pullback homomorphism
		\[
		\varphi^*\colon \CH^*(G/P)\to \CH^*(G)
		\]
		for the projection $\varphi\colon G\to G/P$ is surjective.
	\end{lemma}
	\begin{proof}
		The derived subgroup $[L(P),L(P)]\le L(P)$ of the Levi subgroup $L(P)\le P$ is isomorphic to $\SL_{n-1}$, thus $P$ is special and the claim follows from \cite[Lemma~7.1]{PS17}.
	\end{proof}

        \begin{theorem} \label{thm:D2r_ad_answer}
            Let $G$ be an adjoint quasi-split simple group over a field $k$ of type $\Delta(G)={}^2\mr{D}_{2r}$, $r\ge 2$.
            \begin{enumerate}
                \item Let $T\le B\le G$ be a maximal torus and a Borel subgroup, consider the parabolic subgroup $B\le P=P_{2r-1,2r}\le G$ in the notation of~\refbr{expos:parabolic} and put $\mc{L}:=\mc{L}_P(\varpi_{2r-1}+\varpi_{2r})$ for the respective line bundle over $G/P$. Then the pullback for the projection $\varphi\colon G\to G/P$ induces an isomorphism
                \[
                \tilde{\varphi}^*\colon \Ch^*(G/P)/(c_1(\mc{L})^{2^{v_2(2r)}}) \xrightarrow{\simeq} \Ch^*(G).
                \]
                In particular, if $2r$ is a power of $2$, then the pullback homomorphism
                \[
                \varphi^*\colon \Ch^*(G/P) \to \Ch^*(G)
                \]
                is an isomorphism.
                \item 
                There is an isomorphism
                \[
            \Ch^*(G) \cong
            \begin{cases}
            \FF_2[e_1,\hdots,e_s]/\left(e_1^{2^{k_1}}, \hdots, e_s^{2^{k_s}}, e_1\cdot \left(\sum_{j=0}^{k_1-1} e_1^{2^{k_1}-2^{j+1}} \cdot e_2^{2^{j}}\right)\right), & l=2,\\
            \FF_2[e_1,\hdots,e_s]/(e_1^{2^{k_1}}, \hdots, e_s^{2^{k_s}}, e_1\cdot e_l^{2^{k_l-1}}), & l\ge 3,
            \end{cases}
                \]
            where $l\in \NN$ is such that $2r=2^m(2l-3)$ and
            \begin{gather*}
            s=r+1,\quad \deg e_1=1,\quad \deg e_2=2,\quad k_1=v_2(2r), \quad k_2=[\log_2 2r],\\
            \deg e_i=2i-3,\quad k_i=\left[\log_2 \frac{4r}{2i-3}\right],\quad 3\le i\le s.
            \end{gather*}
            \end{enumerate}
            Here $v_2$ is the $2$-adic valuation.
        \end{theorem}
        \begin{proof}
            Let $K/k$ be the splitting field of $G$, let $\mc{P}$ be the algebra with involution $\tau$ defined in Lemma~\ref{lem:invariants_OGr} and let the non-trivial element of $\Gal(K/k)$ act on $\mc{P}$ via the involution $\tau$. Then
            Proposition~{\hyperref[prop:OGr_Chow_ring]{\ref*{prop:OGr_Chow_ring}.(2)}} combined with Lemma~\ref{lem:Chow_quasi-split_homogeneous} yields isomorphisms 
            \[
            \Theta\colon \mc{P} \xrightarrow{\simeq} \Ch^*((G/P)_K),\quad \Theta^\tau\colon \mc{P}^{\tau} \xrightarrow{\simeq} \Ch^*(G/P).
            \]
            Furthermore, Lemma~{\hyperref[lem:invariants_OGr]{\ref*{lem:invariants_OGr}.(1)}} yields an isomorphism
            \[
            \Ch^*(G/P)/(c_1(\mc{L}))\xrightarrow{\simeq} \CH_K^*(G/P).
            \]
            We have $\varphi^*(c_1(\mc{L})^{2^{v_2(2r)}})=0$ by Proposition~{\hyperref[prop:OGr_Chow_ring]{\ref*{prop:OGr_Chow_ring}.(1)}}, thus we have the following commutative diagram with exact rows.
            \[
            \xymatrix{
            \Ch^*(G_K) \ar[r]^{\pi_*} & \Ch^*(G) \ar[r] & \CH_K^*(G) \ar[r] & 0\\
            \Ch^*((G/P)_K)/(c_1(\mc{L}_K)^{2^{v_2(2r)}}) \ar[u]^{\tilde{\varphi}_K^*} \ar[r]^(0.55){\tilde{\rho}_*} & \Ch^*(G/P)/(c_1(\mc{L})^{2^{v_2(2r)}})  \ar[r] \ar[u]^{\tilde{\varphi}^*} & \CH_K^*(G/P) \ar[r] \ar[u]^{\bar{\varphi}^*} & 0\\
            \mc{P}/(\bar{\varepsilon}_1^{2^{v_2(2r)}}) \ar[u]_\cong^{\tilde{\Theta}} \ar[r]^{\id + \tau} & \mc{P}^\tau/(\bar{\varepsilon}_1^{2^{v_2(2r)}}) \ar[u]_\cong^{\tilde{\Theta}^\tau} \ar[r] & \mc{P}^\tau/(\bar{\varepsilon}_1) \ar[r] \ar[u]^\cong & 0\\
            }
            \]
            Here $\tilde{\rho}_*$ is induced by the projection $\rho\colon (G/P)_K\to G/P$, the isomorphisms $\tilde{\Theta}$ and $\tilde{\Theta}^\tau$ are induced by $\Theta$, the homomorphisms $\tilde{\varphi}^*,\tilde{\varphi}^*_K$, $\bar{\varphi}^*$ are induced by the pullback for the projection $\varphi\colon G\to G/P$, and the unlabelled morphisms are the quotient morphisms. The homomorphism $\bar{\varphi}^*$ is an isomorphism by \cite[case of orthogonal groups in Section~8]{SZ22} and Theorem~{\hyperref[thm:cokernel]{\ref*{thm:cokernel}.(1)}}. The homomorphism $\tilde{\varphi}_K^*$ is surjective by Lemma~\ref{lem:surj_special} and is an isomorphism by the dimension count using the explicit presentation for $\mc{P}$ and the computation of $\Ch^*(G_K)$ recalled in~\refbr{expos:Chow_ring_split}. Then $\tilde{\varphi}^*$ is surjective. Furthermore, a simple diagram chase yields that for injectivity of $\tilde{\varphi}^*$ it suffices to check that if $\pi_*\tilde{\varphi}^*_K\tilde{\Theta}(\alpha)=0$ for $\alpha \in \mc{P}/(\bar{\varepsilon}_1^{2^{v_2(2r)}})$, then $(\id +\tau)(\alpha)=0$. Consider the following commutative diagram:
            \[
            \xymatrix{
            \Ch^*(G_K) \ar[r]^{\pi_*} & \Ch^*(G) \ar[r]^{\pi^*} & \Ch^*(G_K) \\
            \mc{P}/(\bar{\varepsilon}_1^{2^{v_2(2r)}}) \ar[u]^{\tilde{\varphi}^*_K\circ \tilde{\Theta}} \ar[r]^{\id + \tau} & \mc{P}^\tau/(\bar{\varepsilon}_1^{2^{v_2(2r)}}) \ar[u]^{\tilde{\varphi}^*\circ \tilde{\Theta}^\tau} \ar[r]^J & \mc{P}/(\bar{\varepsilon}_1^{2^{v_2(2r)}}) \ar[u]^{\tilde{\varphi}^*_K\circ \tilde{\Theta}}  \\
            }
            \]
            Here $J$ is induced by the inclusion $\mc{P}^{\tau}\subseteq \mc{P}$. We have
            \[
            \tilde{\varphi}^*_K\circ \tilde{\Theta}\circ J \circ (\id+\tau)(\alpha)=\pi^*\circ \pi_*\circ \tilde{\varphi}^*_K\circ\tilde{\Theta}(\alpha)=0,
            \]
            thus $J \circ (\id+\tau)(\alpha)=0$, since $\tilde{\varphi}^*_K\circ \tilde{\Theta}$ is an isomorphism. Then Lemma~{\hyperref[lem:invariants_OGr]{\ref*{lem:invariants_OGr}.(2)}} yields $(\id+\tau)(\alpha)=0$. Thus $\tilde{\varphi}^*$ is an isomorphism. The remaining claims of the Theorem follow from this, the isomorphism $\Theta^\tau\colon \mc{P}^\tau\xrightarrow{\simeq}\Ch^*(G/P)$ and Lemma~{\hyperref[lem:invariants_OGr]{\ref*{lem:invariants_OGr}.(3)}}.
        \end{proof}
        
        \begin{corollary}
            Let $G$ be an adjoint quasi-split simple group over a field $k$ of type $\Delta(G)={}^2\mr{D}_{2r}$, $r\ge 2$. Then $\Ch^*(G)$ does not admit a structure of a Hopf algebra, in particular, the K\"unneth homomorphism
            \[
            \Ch^*(G) \otimes \Ch^*(G) \to \Ch^*(G\times G),\quad x\otimes y \mapsto p_1^*(x)\cdot p_2^*(y),
            \]
            fails to be an isomorphism.
        \end{corollary}
        \begin{proof}
            It follows from Theorem~{\hyperref[thm:D2r_ad_answer]{\ref*{thm:D2r_ad_answer}.(2)}} that $\Ch^*(G)$ admits a basis given by
            \begin{gather*}
            e_1^{i_1}e_2^{i_2}\hdots e_{s}^{i_{s}}, \quad 0\le i_j < 2^{k_j},\, 1\le j\le s,\,j\neq l,\, 0\le i_l < 2^{k_l-1},\\
            e_2^{i_2}e_3^{i_3}\hdots e_{s}^{i_{s}}, \quad 0\le i_j < 2^{k_j},\, 2\le j\le s,\,j\neq l,\, 2^{k_l-1}\le i_l < 2^{k_l},
            \end{gather*}
            in the notation of loc. cit. Then
            \[
            \dim_{\FF_2} \Ch^*(G) = 2^{N-1} + 2^{N-1-k_1},\quad N:=\sum_{i=1}^s k_i.
            \]
            Since $k_1=v_2(2r)\ge 1$, the integer $ 2^{N-1} + 2^{N-1-k_1}$ is not a power of $2$, and it could not be the dimension of a Hopf algebra by a result of Borel on the structure of finite dimensional connected Hopf algebras  \cite[Theorem 7.11 and Proposition 7.8]{MM65}.
        \end{proof}

	\subsection{Chow ring of a quasi-split group at the splitting prime}

    In this section we show that if $G$ is a non-split quasi-split simple group over a field $k$ such that $\Delta(G)\neq {}^6\mr{D}_4$ and $(\Delta(G),\pi_1(G))\neq({}^2\mr{D}_{2r},\mut{2}{2,2})$, then $\CH^*(G)\otimes \FF_p\cong \CH^*_K(G)$, where $K/k$ is the splitting field of $G$ and $p=[K:k]$. This is done analysing the pushforward homomorphism $\pi_*\colon \CH^*(G_K)\otimes \FF_p \to \CH^*(G)\otimes \FF_p$, which we show to be trivial in this case.

    We continue to use the notation of Sections~\ref{sec:quasi-split_simple} and~\ref{sec:vector_bundles}. In this section we denote
        \[
        \Ch^*(-):=\CH^*(-)\otimes \FF_p
        \]
    for some given prime number $p$.

 \begin{lemma} \label{lem:Pic_zero}
	Let $G$ be a non-split quasi-split simple algebraic group over a field $k$ with the splitting field $K/k$ and let $\pi\colon G_K\to G$ be the projection. Then the following holds.
	\begin{enumerate}
		\item If $(\Delta(G),\pi_1(G))\neq({}^2\mr{D}_{2r},\mut{2}{2,2})$, then $\pi_*\colon \Pic(G_K) \to \Pic(G)$ is the zero homomorphism.
		\item If $(\Delta(G),\pi_1(G))=({}^2\mr{D}_{2r},\mut{2}{2,2})$, then
		\[
		\pi_*([\mc{L}(\bar{\varpi}_{2r-1})])=\pi_*([\mc{L}(\bar{\varpi}_{2r})])=[\mc{L}(\bar{\varpi}_{1})],\quad \pi^*([\mc{L}(\bar{\varpi}_{1})]) = [\mc{L}(\bar{\varpi}_{2r-1})] + [\mc{L}(\bar{\varpi}_{2r})]
		\]
            in $\Pic(G)$ and $\Pic(G_K)$ respectively.
	\end{enumerate}
\end{lemma}
\begin{proof}
	Straightforward from the recollection in Section~\ref{sec:fundamental_and_Picard}.
\end{proof}

	\begin{proposition} \label{prop:norm_zero}
	Let $G$ be a non-split quasi-split simple algebraic group over a field $k$ with the splitting field $K/k$ and let $\pi\colon G_K\to G$ be the projection. Suppose $\Delta(G)\neq {}^6\mr{D}_{4}$ and $(\Delta(G),\pi_1(G))\neq({}^2\mr{D}_{2r},\mut{2}{2,2})$ and put $p:=[K:k]$. Then $\pi_*\colon \Ch^*(G_K)\to \Ch^*(G)$ is the zero homomorphism.
	\end{proposition}
	\begin{proof}
		Let $T\le B\le G$ be a maximal torus and a Borel subgroup. We argue case by case on $\Delta(G)$.
		
		$\mathbf{{}^2A_n}$. We have $p=2$ and it follows from recollection~\refbr{expos:Chow_ring_split} that the ring $\Ch^*(G_K)$ is either trivial or $\Ch^*(G_K)\cong \FF_2 [e_1]/(e_1^{2^{v_2(n+1)}})$ with $\deg e_1=1$. Consider the following commutative diagram.
		\[
		\xymatrix{
            \Ch^*((G/B)_K) \ar[d]_{\phi_K^*}  \ar[r]^{\rho_*} & \Ch^*(G/B) \ar[d]^{\phi^*}\\	
			 \Ch^*(G_K) \ar[r]^{\pi_*} &\Ch^*(G)
		}
		\]
		Here $\phi\colon G\to G/B$, $\phi_K\colon G_K\to (G/B)_K$ and $\rho \colon (G/B)_K\to G/B$ are the projections. For the line bundle $\mc{L}(\varpi_1)$ over $(G/B)_K$ we have $\phi_K^*(c_1(\mc{L}(\varpi_1)))=e_1$ by recollection~\refbr{expos:Picard}. Then for all $m\in \NN$ we have
        \[
        \pi_*(e_1^m)=\pi_*\phi^*_K (c_1(\mc{L}(\varpi_1))^m ) =  \phi^* \rho_* (c_1(\mc{L}(\varpi_1))^m ).
        \]
        Furthermore, for the non-trivial element $\tau\in \Gal(K/k)$ we have
        \begin{multline*}
        \rho^*\rho_* (c_1(\mc{L}(\varpi_1))^m ) = c_1(\mc{L}(\varpi_1))^m + \tau(c_1(\mc{L}(\varpi_1))^m) = c_1(\mc{L}(\varpi_1))^m + c_1(\mc{L}(\varpi_n))^m  =\\ = (c_1(\mc{L}(\varpi_1)) + c_1(\mc{L}(\varpi_n)))\cdot (\sum_{i=0}^{m-1} c_1(\mc{L}(\varpi_1))^ic_1(\mc{L}(\varpi_n))^{m-1-i})
        \end{multline*}
        with the last sum being $\Gal(K/k)$-invariant. Then Lemma~\ref{lem:Chow_quasi-split_homogeneous} yields that
        \[
        \rho_*(c_1(\mc{L}(\varpi_1))^m )=(c_1(\mc{L}(\varpi_1)) + c_1(\mc{L}(\varpi_n)))\cdot\alpha = \rho_*(c_1(\mc{L}(\varpi_1))) \cdot \alpha
        \]
        for some $\alpha \in \Ch^*(G/B)$. Thus
       \[
       \pi_*(e_1^m) = \phi^* \rho_* (c_1(\mc{L}(\varpi_1))^m ) = \phi^* \rho_*(c_1(\mc{L}(\varpi_1))) \cdot \phi^*\alpha = \pi_*(e_1) \cdot \phi^*\alpha =0
       \]
       with the last equality following from $\pi_*(e_1)=0$ by Lemma~\ref{lem:Pic_zero}.
		
		$\mathbf{{}^2D_n}$. We have $p=2$. Consider the parabolic subgroup $B\le P:=P_{n-1,n}\le G$ in the notation of~\refbr{expos:parabolic} and the following commutative diagram.
		\[
		\xymatrix{
        \Ch^*((G/P)_K) \ar[d]_{\varphi_K^*}  \ar[r]^{\rho_*} & \Ch^*(G/P) \ar[d]^{\varphi^*}\\	  
			 \Ch^*(G_K) \ar[r]^{\pi_*} &\Ch^*(G)
		}
		\]
		Here $\varphi\colon G\to G/P$, $\varphi_K\colon G_K\to (G/P)_K$ and $\rho \colon (G/P)_K\to G/P$ are the projections. The homomorphism $\varphi^*_K$ is surjective by Lemma~\ref{lem:surj_special}, thus
		\[
		\pi_*(\Ch^*(G_K)) = \varphi^* \rho_* (\Ch^*(G/P)_K).
		\]
	Lemma~\ref{lem:Chow_quasi-split_homogeneous} combined with Proposition~{\hyperref[prop:OGr_Chow_ring]{\ref*{prop:OGr_Chow_ring}.(2)}} and Lemma~{\hyperref[lem:invariants_OGr]{\ref*{lem:invariants_OGr}.(1)}} yields
    \[
    \rho_* (\Ch^*(G/P)_K) = c_1(\mc{L}_P({\varpi}_{n-1}+{\varpi}_{n}))\cdot (\Ch^*(G/P)).
    \]
 Since $(\Delta(G),\pi_1(G))\neq({}^2\mr{D}_{2r},\mut{2}{2,2})$, it follows that $\bar{\varpi}_{n-1}+\bar{\varpi}_{n}=0$ in $\mc{X}^*(\pi_1(G))$ and
		\[
		\varphi^*c_1(\mc{L}_P({\varpi}_{n-1}+{\varpi}_{n})) = c_1(\varphi^*\mc{L}_P({\varpi}_{n-1}+{\varpi}_{n}))=c_1(\mc{L}(\bar{\varpi}_{n-1}+\bar{\varpi}_{n}))=c_1(\struct_G)=0.
		\]
		
		$\mathbf{{}^2E_6}$. We have $p=2$. Consider the following commutative diagram.
		\[
		\xymatrix{
\Ch^*((G/B)_K) \ar[d]_{\phi_K^*} \ar[r]^{\rho_*} & \Ch^*(G/B) \ar[d]^{\phi^*}  \\	  
			\Ch^*(G_K) \ar[r]^{\pi_*} & \Ch^*(G) 
		}
		\]
		Here $\phi\colon G\to G/B$, $\phi_K\colon G_K\to (G/B)_K$ and $\rho\colon (G/B)_K\to G/B$ are the canonical morphisms. By recollection~\refbr{expos:Chow_ring_split} we have $\Ch^*(G_K)\cong \FF_2[e_1]/(e_1^2)$. A straightforward computer-assisted computation shows that for the Schubert cycle $Z_{s_3s_4s_2}\in \Ch^3((G/B)_K)$ one has $\phi_K^*(Z_{s_3s_4s_2})=e_1$, where $s_i$ is the reflection for the simple root $\alpha_i$. Then
		\[
		\pi_*(e_1) = \pi_* \phi_K^* (Z_{s_3s_4s_2}) = \phi^* \rho_* (Z_{s_3s_4s_2}).
		\]
		The Galois group $\Gal(K/k)$ acts on $\Ch^*((G/B)_K)$ permuting Schubert cycles accordingly to the permutation of simple roots, thus
		\[
		\rho^*\rho_*(Z_{s_3s_4s_2})=Z_{s_3s_4s_2}+Z_{s_5s_4s_2}.
		\]
            Another straightforward computer-assisted computation shows that
		\[
		c_1(\mc{L}(\varpi_2)_K)^3 = Z_{s_3s_4s_2}+Z_{s_5s_4s_2}.
		\]
		Then the injectivity part of Lemma~\ref{lem:Chow_quasi-split_homogeneous} yields
		\[
		c_1(\mc{L}(\varpi_2))^3=\rho_*(Z_{s_3s_4s_2}).
		\]
            Since the fundamental weight $\varpi_2$ belongs to the root lattice, it follows that $\bar{\varpi}_2=0$ in $\mc{X}^*(\pi_1(G))$ and
		\[
		\pi_*(e_1)=\phi^* \rho_* (Z_{s_3s_4s_2}) = \phi^* c_1(\mc{L}(\varpi_2))^3=c_1(\mc{L}(\bar{\varpi}_2))^3=0.
		\]
		
		$\mathbf{{}^3D_4}$. We have $p=3$ and the claim is trivial, since $\Ch^*(G_K)=\FF_3$ by recollection~\refbr{expos:Chow_ring_split}.
	\end{proof}

        \begin{theorem} \label{thm:answer_at_p}
            Let $G$ be a non-split quasi-split simple group over a field $k$ with the splitting field $K/k$ and suppose that $\Delta(G)\neq {}^6\mr{D}_{4}$ and $(\Delta(G),\pi_1(G))\neq({}^2\mr{D}_{2r},\mut{2}{2,2})$. Then for $p=[K:k]$ there is an isomorphism
            \[
            \Ch^*(G) \cong \FF_p[e_1,e_2,\hdots,e_s]/(e^{p^{k_1}}_1,e^{p^{k_2}}_2,\hdots, e^{p^{k_s}}_s),\quad \deg e_i=d_i,
		\]
		with the parameters $s,d_i,k_i$ given in the Table~\ref{tab:main}.
        \end{theorem}
        \begin{proof}
            Proposition~\ref{prop:norm_zero} yields that the quotient morphism $\Ch^*(G)\to \CH_K^*(G)$ is an isomorphism. Thus the claim follows from Theorem~\ref{thm:conormed_group_answer}.
        \end{proof}

        \subsection{Chow ring of a quasi-split group away from the splitting primes}

    In this section we compute the Chow ring $\Ch^*(G):=\CH^*(G)\otimes \FF_p$ for a non-split quasi-split simple group $G$ over a field $k$ assuming that $p$ is coprime to $[K:k]$ where $K$ is the splitting field of $G$. For this we show that for such $p$ one has $\Ch^*(G)\cong \Ch^*(G_K)^{\Gal(K/k)}$ and then we compute the Galois action and the subalgebra of invariant elements.

    We continue to use the notation of Sections~\ref{sec:quasi-split_simple} and~\ref{sec:vector_bundles}. In this section we put
        \[
        \Ch^*(-):=\CH^*(-)\otimes \FF_p
        \]
    for some given prime number $p$.

    \begin{lemma}\label{lem:pull_inj}
        Let $X\in \Smk$, $K/k$ be a field extension of finite degree and $p$ be a prime such that $p\nmid [K:k]$. Then for the projection $\pi\colon X_K\to X$ the pullback homomorphism
        \[
        \pi^*\colon \Ch^*(X) \to \Ch^*(X_K)
        \]
        is injective. If $K/k$ is Galois then furthermore
        \[
        \pi^*(\Ch^*(X))=(\Ch^*(X_K))^{\Gal(K/k)}.
        \]
    \end{lemma}
    \begin{proof}
        This is well-known. Injectivity follows from the fact that $\pi_*\pi^*(\alpha) = [K:k]\cdot\alpha$ for $\alpha\in \Ch^*(X)$, see e.g. \cite[Example~1.7.4]{Fu98}. For the second claim pick $\alpha\in (\Ch^*(X_K))^{\Gal(K/k)}$, then for $\tilde{\alpha}=\frac{1}{[K:k]}\pi_*(\alpha)$ one has 
        \[
        \pi^*(\tilde{\alpha})=\frac{1}{[K:k]}\pi^*\pi_*(\alpha) =\frac{1}{[K:k]}\sum_{\sigma\in\Gal(K/k)}\sigma(\alpha)=\alpha.
        \]
    \end{proof}

	\begin{theorem} \label{thm:answer_away_p}
        Let $G$ be a non-split quasi-split simple group over a field $k$ with the splitting field $K/k$ and let $p\in \NN$ be a prime such that $p\nmid [K:k]$. If $(\Delta(G_K),\pi_1(G_K),p)$ is not in Table~\ref{tab:Kac}, then $\Ch^*(G)\cong \FF_p$, otherwise $\Ch^*(G)$ is given by the following table.

            \begin{center}
			\def\arraystretch{1.5}
			\begin{longtable}{l|l|l|l|l}
   \caption{$\Ch^*(G)$ away from the splitting primes}\\
   \label{tab:away}
    			$\Delta(G)$ & $\pi_1(G)$ & $p$ & $\Ch^*(G)$ & $\deg e_i$ \\
				\hline
                    $\Att_n$  & $\mut{2}{l}$, $l\mid n+1$ & $2\neq p\mid l$ & $\FF_p[e_1]/(e_1^m)$, $m=\frac{p^{v_p(n+1)}+1}{2}$ & $2$\\
				${}^3\mr{D}_4$ & $1$ & $2$ & $\FF_2[e_1]/(e_1^2)$ & $3$\\
    & $\mut{3}{2,2}$ & $2$ & $\FF_2[e_1,e_2,e_3,e_4]/(e_1^4,e_2^2,e_3^2,e_4^2,e_1e_2,e_1e_3,e_1^3+e_2e_3)$ & $2,3,3,3$\\
                ${}^2\mr{E}_6$ & $1$ & $3$ & $\FF_3[e_1]/(e_1^3)$ & $4$\\
                 & $\mut{2}{3}$ & $3$ & $\FF_3[e_1,e_2]/(e_1^5,e_2^3)$ & $2,4$ 
		\end{longtable}
		\end{center}  
        Here $v_p$ is the $p$-adic valuation.

	\end{theorem}
	\begin{proof}
		The claim follows from Lemma~\ref{lem:pull_inj}, the computation of $\Ch^*(G_K)$ recalled in~\refbr{expos:Chow_ring_split} and the description of the action of $\Gal(K/k)$ on $\Ch^*(G_K)$.

        If $(\Delta(G_K),\pi_1(G_K),p)$ is not in Table~\ref{tab:Kac}, then $\Ch^*(G_K)\cong \FF_p$ and the claim follows. Otherwise, since $p \nmid [K:k]$, one can easily check that the triple $(\Delta(G),\pi_1(G),p)$ is in Table~\ref{tab:away}, and we study them case by case.

        $\Delta(G)=\Att_n$. We have
        \[
        \Ch^*(G_K)\cong \FF_p[e]/(e^{p^{v_p(n+1)}}), \quad \deg e=1.
        \]
        In view of recollection~\refbr{expos:Picard} we can assume $e=c_1(\mc{L}(\bar{\varpi}_1))$. Thus for the nontrivial element $\tau\in \Gal(K/k)$ we have
        \[
        \tau(e)=\tau(c_1(\mc{L}(\bar{\varpi}_1))) = c_1(\mc{L}(\bar{\varpi}_{n})) = -e.
        \]
        Hence $\Ch^*(G_K)^{\Gal(K/k)}\subseteq \Ch^*(G_K)$ is the subalgebra generated by $e^2$ and the claim follows.

        $(\Delta(G),\pi_1(G))=({}^3\mr{D}_4,1)$. We have
        \[
        \Ch^*(G_K)\cong \FF_2[e]/(e^2), \quad \deg e=3.
        \]
        The vector space $\FF_2\cong \Ch^3(G_K)$ admits only the trivial action of $\Gal(K/k)$, thus $e$ is invariant and
        \[
        \Ch^*(G)\cong \Ch^*(G_K)^{\Gal(K/k)}\cong \FF_2[e]/(e^2).
        \]

        $(\Delta(G),\pi_1(G))=({}^3\mr{D}_4,\mut{3}{2,2})$.  We have
        \[
        \Ch^*(G_K)\cong \FF_2[e_1,e_2,e_3]/(e_1^4,e_2^4,e_3^2), \quad \deg e_1=\deg e_2 =1,\, \deg e_3=3.
        \]
        In view of recollection~\refbr{expos:Picard} we can assume $e_1=c_1(\mc{L}(\bar{\varpi}_3))$ and $e_2=c_1(\mc{L}(\bar{\varpi}_4))$. For a generator ${\sigma\in \Gal(K/k)\cong C_3}$ we have
        \[
        \sigma(e_1)=e_2,\quad \sigma(e_2)=e_1+e_2.
        \]
        Let $\mc{A}^*\subseteq \Ch^*(G_K)$ be the subalgebra generated by $e_1,e_2$. Then it is straightforward to check that $(\mc{A}^*)^{\Gal(K/k)}$ has a basis given by
        \[
        1,\quad e_1^2+e_1e_2+e_2^2,\quad e_1^2e_2+e_1e_2^2,\quad e_1^3+e_1^2e_2+e_2^3,\quad e_1^2e_2^2,\quad e_1^3e_2^3.
        \]
        We have $\mc{A}^3\subseteq \Ch^3(G_K)$ a codimension one $\Gal(K/k)$-invariant subspace. Maschke's theorem yields that 
        \[
        \Ch^3(G_K) = \mc{A}^3 \oplus \FF_2 \tilde{e}_3
        \]
        for some $\tilde{e}_3\in \Ch^3(G_K)$ with $\sigma(\tilde{e}_3)=\tilde{e}_3$. We have $\tilde{e}_3=e_3+p(e_1,e_2)$ for some degree $3$ homogeneous polynomial. Then we obtain a $\Gal(K/k)$-invariant decomposition
        \[
        \Ch^*(G_K) =  \mc{A}^* \oplus  \mc{A}^*\tilde{e}_3,
        \]
        and it follows that $\Ch^*(G_K)^{\Gal(K/k)}$ has a basis given by
        \begin{gather*}
        1,\quad e_1^2+e_1e_2+e_2^2,\quad e_1^2e_2+e_1e_2^2,\quad e_1^3+e_1^2e_2+e_2^3,\quad e_1^2e_2^2,\quad e_1^3e_2^3, \\
        \tilde{e}_3,\quad (e_1^2+e_1e_2+e_2^2)\tilde{e}_3,\quad (e_1^2e_2+e_1e_2^2)\tilde{e}_3,\quad (e_1^3+e_1^2e_2+e_2^3)\tilde{e}_3,\quad e_1^2e_2^2\tilde{e}_3,\quad e_1^3e_2^3\tilde{e}_3.
        \end{gather*}
        Thus $\Ch^*(G_K)^{\Gal(K/k)}$ is generated as an algebra by $e_1^2+e_1e_2+e_2^2$, $e_1^2e_2+e_1e_2^2$, $e_1^3+e_1^2e_2+e_2^3$ and $\tilde{e}_3$, and the multiplicative relations are straightforward.

        $(\Delta(G),\pi_1(G))=({}^2\mr{E}_6,1)$. We have
        \[
        \Ch^*(G_K)\cong \FF_3[e]/(e^3), \quad \deg e=4.
        \]
        Let $T\le B\le G$ be a maximal torus and a Borel subgroup and consider the characteristic sequence
        \[
        \Ch^*_{T_K}(\Spec K) \xrightarrow{c} \Ch^*(G_K/B_K) \xrightarrow{\phi^*_K} \Ch^*(G_K)\to 0.
        \]
        A straightforward computer-assisted computation shows that for the Schubert cycle 
        \[
        Z_{s_2s_4s_3s_1}\in \Ch^4(G_K/B_K)
        \]
        its image $ \phi^*_K(Z_{s_2s_4s_3s_1})$ is a generator of  $\Ch^4(G_K)$. We have $\tau(Z_{s_2s_4s_3s_1}) = Z_{s_2s_4s_5s_6}$ for the nontrivial element $\tau\in \Gal(K/k)$, and a further computer-assisted computation shows that $Z_{s_2s_4s_3s_1}-Z_{s_2s_4s_5s_6}$ lies in $c(\Ch^*_{T_K}(\Spec K))$. Then
        \[
        \tau(\phi^*_K(Z_{s_2s_4s_3s_1})) = \phi^*_K(\tau(Z_{s_2s_4s_3s_1})) = \phi^*_K(Z_{s_2s_4s_5s_6}) = \phi^*_K(Z_{s_2s_4s_3s_1}),
        \]
        so $\Gal(K/k)$ acts trivially on $\Ch^4(G_K)$ and
        \[
        \Ch^*(G) \cong \Ch^*(G_K)^{\Gal(K/k)}\cong \FF_3[e]/(e^3), \quad \deg e=4.
        \]

        $(\Delta(G),\pi_1(G))=({}^2\mr{E}_6,\mut{2}{3})$. We have
        \[
        \Ch^*(G_K)\cong \FF_3[e_1,e_2]/(e_1^9,e_2^3), \quad \deg e_1=1,\, \deg e_2=4.
        \]
        In view of recollection~\refbr{expos:Picard} we can assume $e_1=c_1(\mc{L}(\bar{\varpi}_1))$, thus for the nontrivial element $\tau\in \Gal(K/k)$ we have
        \[
        \tau(e_1)=\tau(c_1(\mc{L}(\bar{\varpi}_1))) = c_1(\mc{L}(\bar{\varpi}_6)) = c_1(\mc{L}(-\bar{\varpi}_1)) = -e_1.
        \]
        We have $\Ch^4(G_K)=\FF_3e_1^4\oplus \FF_3 e_2$, and it follows from the above that $\tau(e_1^4)=e_1^4$. Furthermore, the simply connected case yields $\tau(e_2)=e_2+ae_1^4$ for some $a\in \FF_3$. Since
        \[
        e_2=\tau^2(e_2) = \tau(e_2+ae_1^4) = e_2+2ae_1^4,
        \]
        we have $a=0$ and $\tau(e_2)=e_2$. Thus $\Ch^*(G_K)^{\Gal(K/k)}\subseteq \Ch^*(G_K)$ is the subalgebra generated by $e_1^2$ and $e_2$. The claim follows.
	\end{proof}

        \subsection{Chow ring of the group of type ${}^6\mr{D}_4$ at $p=2,3$}
        In this section we compute the ring $\Ch^*(G):=\CH^*(G)\otimes \FF_p$ in the last remaining cases, namely $\Delta(G)={}^6\mr{D}_4$ and $p=2,3$. This is done partially splitting $G$ to $G_L$ with $\Delta(G_L)={}^3\mr{D}_4$ or $\Delta(G_L)={}^2\mr{D}_4$, and analyzing the Galois action on $\Ch^*(G_K)$, where $K$ is the splitting field of $G$.
    
        We continue to use the notation of Sections~\ref{sec:quasi-split_simple} and~\ref{sec:vector_bundles}. In this section we put
        \[
        \Ch^*(-):=\CH^*(-)\otimes \FF_p
        \]
        for $p=2$ or $p=3$.
    
         \begin{theorem} \label{thm:6D4}
            Let $G$ be a quasi-split simple group over a field $k$ of type $\Delta(G)={}^6\mr{D}_4$. Then for $p=2,3$ the algebra $\Ch^*(G)$ is given by the following table.

                        \begin{center}
			\def\arraystretch{1.5}
			\begin{longtable*}{l|l|l|l|l}
    			$\Delta(G)$ & $\pi_1(G)$ & $p$ & $\Ch^*(G)$ & $\deg e_i$ \\
				\hline
				${}^6\mr{D}_4$ & $1$ & $2$ & $\FF_2[e_1]/(e_1^2)$ & $3$\\
    & $\mut{6}{2,2}$ & $2$ & $\FF_2[e_1,e_2,e_3]/(e_1^4,e_2^2,e_1e_2,e_3^2)$ & $2,3,3$\\
    & $1$, $\mut{6}{2,2}$ & $3$ & $\FF_3[e_1]/(e_1^3)$ & $4$\\
		\end{longtable*}
		\end{center}  
  
        \end{theorem}
        \begin{proof}
            Let $K/k$ be the splitting field of $G$ and $T\le B\le G$ be a maximal torus and a Borel subgroup. We argue case by case on $(\pi_1(G),p)$.

            $(\pi_1(G),p)=(\mut{6}{2,2},2)$. Let $K/L/k$ be an intermediate extension with $[L:k]=3$ such that $\Gal(K/L)$ stabilizes the root $\alpha_1$ on the Dynkin diagram of $G_K$, then $\Delta({G}_L)={}^2\mr{D}_4$ with the nontrivial element of $\Gal(K/L)$ permuting the roots $\alpha_3$ and $\alpha_4$. Consider the parabolic subgroup $B_L\le P:={P}_{3,4}\le G_L$ in the notation of~\refbr{expos:parabolic} and consider the following commutative diagram.
            \[
            \xymatrix{
            & \Ch^*({G}_L/P) \ar[r]^{g^*} \ar[d]^{\psi_L^*} & \Ch^*({G}_K/P_K) \ar[d]^{\psi_K^*}\\
            \Ch^*({G}/{B}) \ar[r]^{\rho_{L/k}^*} \ar[d]_{\phi^*} & \Ch^*({G}_L/{B}_L) \ar[r]^{\rho_{K/L}^*} \ar[d]^{\phi_L^*} & \Ch^*({G}_K/{B}_K) \ar[d]^{\phi_K^*}\\
            \Ch^*({G}) \ar[r]^{\pi_{L/k}^*} & \Ch^*({G}_L) \ar[r]^{\pi_{K/L}^*} & \Ch^*({G}_K)
            }
            \]
            Here all the morphisms are pullbacks for the respective projections. We have the following:
            \begin{itemize}
                \item $g^*$, $\rho_{L/k}^*$, $\rho_{K/L}^*$ are injective by Lemma~\ref{lem:Chow_quasi-split_homogeneous},
                \item $\pi^*_{L/k}$ is injective by Lemma~\ref{lem:pull_inj},
                \item $\phi^*_L\circ \psi^*_L$ is an isomorphism by Theorem~{\hyperref[thm:D2r_ad_answer]{\ref*{thm:D2r_ad_answer}.(1)}},
                \item $\phi^*_K\circ \psi^*_K$ is an isomorphism, this is well-known, see e.g. the beginning of the proof of Theorem~\ref{thm:D2r_ad_answer}.
            \end{itemize}
            It follows that $\pi^*_{K/L} = \phi^*_K\circ \psi^*_K \circ g^*\circ (\phi^*_L\circ \psi^*_L)^{-1} $ is injective, so the composition
            \[
            \pi^*_{K/k}:=\pi^*_{K/L}\circ \pi^*_{L/k} \colon \Ch^*({G}) \to \Ch^*(G_K)
            \]
            is injective. It is clear that $\pi^*_{K/k}(\Ch^*(G))\subseteq \Ch^*(G_K)^{\Gal(K/k)}$, and we claim that this is in fact an equality. Indeed, pick $x\in \Ch^*(G_K)^{\Gal(K/k)}$ and consider 
            \[
            y' := \psi^*_K\circ (\phi^*_K\circ \psi^*_K)^{-1} (x),\quad y:=y'+\sigma(y')+\sigma^2(y'),
            \]
            where $\sigma\in \Gal(K/k)\cong S_3$ is an element of order $3$. It is clear that $\sigma(y)=y$. For the nontrivial element $\tau\in \Gal(K/L)\subseteq \Gal(K/k)$ we have $\tau(y')=y'$, since $\tau(x)=x$ by the assumption and $\psi_K$ and $\phi_K$ are obtained by a base change from $\Spec L$. Then
            \[
            \tau(y)=\tau(y')+\tau\sigma(y')+\tau\sigma^2(y')=\tau(y')+\sigma^2\tau(y')+\sigma\tau(y') = y'+\sigma^2(y')+\sigma(y') = y,
            \]
            and it follows that $y\in \Ch^*({G}_K/{B}_K)^{\Gal(K/k)}$. Lemma~\ref{lem:Chow_quasi-split_homogeneous} yields that there exists $z\in \Ch^*({G}/{B})$ such that 
            \[
            (\rho_{L/k}\circ \rho_{K/L})^*(z) = y.
            \]
            Therefore, we have
            \[
            \pi^*_{K/k}\phi^*(z) = \phi^*_K(y) = \phi^*_K(y'+\sigma(y')+\sigma^2(y')) = \phi^*_K(y')+\sigma \phi^*_K(y')+\sigma^2 \phi^*_K(y') = 3x= x,
            \]
            with the third equality following from the fact that $\phi_K$ is obtained by a base change from $\Spec k$. Then
            \[
            \pi^*_{K/k}(\Ch^*({G}))= \Ch^*({G}_K)^{\Gal(K/k)}
            \]
            as claimed, and it remains to compute the algebra of invariants.

            We have an isomorphism 
            \[
            \Theta\colon \Ch^*({G}_K) \xrightarrow{\simeq} \FF_2[e_1,e_2,e_3]/(e_1^4,e_2^4,e_3^2),\quad  \deg e_1 =\deg e_2 =1,\quad \deg e_3=3,
            \]
            by recollection~\refbr{expos:Chow_ring_split}.  A straightforward computer-assisted computation with the characteristic sequence
            \[
            \Ch^*_{{T}_K}(\Spec K) \to \Ch^*({G}_K/{B}_K) \to \Ch^*(G_K)\to 0
            \]
            yields that there is a choice of the isomorphism $\Theta$ such that
            \begin{gather*}
            \tau(e_1)=e_2, \quad\tau(e_2)=e_1,\quad \sigma(e_1)=e_2,\quad \sigma(e_2)=e_1+e_2,\\
            \tau(e_3) =  e_3+e_1^2e_2+e_1e_2^2,\quad \sigma(e_3)=e_3+e_1^2e_2+e_1e_2^2+e_2^3
            \end{gather*}
            for order $2$ and order $3$ elements $\tau,\sigma\in \Gal(K/k)\cong S_3$. As in the proof of Theorem~{\hyperref[thm:answer_away_p]{\ref*{thm:answer_away_p}}}, a basis of the $\sigma$-invariant subalgebra of $\Ch^*({G}_K)$ is given by
            \begin{equation}\label{eq:basis_6D4}
            \begin{gathered}
        1,\quad e_1^2+e_1e_2+e_2^2,\quad e_1^2e_2+e_1e_2^2,\quad e_1^3+e_1^2e_2+e_2^3,\quad e_1^2e_2^2,\quad e_1^3e_2^3, \\
        \tilde{e}_3,\quad (e_1^2+e_1e_2+e_2^2)\tilde{e}_3,\quad (e_1^2e_2+e_1e_2^2)\tilde{e}_3,\quad (e_1^3+e_1^2e_2+e_2^3)\tilde{e}_3,\quad e_1^2e_2^2\tilde{e}_3,\quad e_1^3e_2^3\tilde{e}_3.
            \end{gathered}
            \end{equation}
            for some $\tilde{e}_3=e_3+p(e_1,e_2)$ with $p$ being a homogeneous polynomial of degree $3$. Using the above description for the action we see that one can choose
            \[
            \tilde{e}_3 := e_3+e_1e_2^2.
            \]
            It turns out that all the elements of the basis~\refbr{eq:basis_6D4} except $e_1^3+e_1^2e_2+e_2^3$ and $ (e_1^3+e_1^2e_2+e_2^3)\tilde{e}_3$ are also $\tau$-invariant, while
            \begin{gather*}
            \tau(e_1^3+e_1^2e_2+e_2^3) = (e_1^3+e_1^2e_2+e_2^3)+(e_1^2e_2+e_1e_2^2),\\
            \tau((e_1^3+e_1^2e_2+e_2^3)\tilde{e}_3) = (e_1^3+e_1^2e_2+e_2^3)\tilde{e}_3+(e_1^2e_2+e_1e_2^2)\tilde{e}_3.
            \end{gather*}
            It follows that the remaining $10$ elements form a basis of $\Ch^*({G}_K)^{\Gal(K/k)}$. The claim of the theorem follows with $e_1^2+e_1e_2+e_2^2$, $e_1^2e_2+e_1e_2^2$ and $e_3+e_1e_2^2$ being the algebra generators of $\Ch^*({G}_K)^{\Gal(K/k)}\cong \Ch^*({G})$.

             $(\pi_1(G),p)=(1,2)$. Let $K/L/k$ be an intermediate extension with $[L:k]=3$ and $f\colon G\to \bar{G}$ be the adjoint quotient. Consider the following commutative diagram.
            \[
            \xymatrix{
            \Ch^*(\bar{G}) \ar[r]^{\bar{\pi}_{L/k}^*} \ar[d]^{f^*} &   \Ch^*(\bar{G}_L) \ar[d]^{f_L^*}\ar[r]^{\bar{\pi}_{K/L}^*} &   \Ch^*(\bar{G}_K) \ar[d]^{f_K^*} \\
            \Ch^*(G) \ar[r]^{\pi_{L/k}^*} &   \Ch^*(G_L) \ar[r]^{\pi_{K/L}^*}  &   \Ch^*(G_K) \\            
            }
            \]
            Here all the morphisms are pullbacks for the respective projections. By recollection~\refbr{expos:Chow_ring_split} we have
            \[
            \Ch^*(G_K)\cong \FF_2[e]/(e^2),\quad \Ch^*(\bar{G}_K)\cong \FF_2[e_1,e_2,e_3]/(e_1^4,e_2^4,e_3^2),
            \]
            where for the second isomorphism we choose the same one as above in the proof. We have $f^*_K(e_1)=f^*_K(e_2)=0$ by Lemma~\ref{lem:Pic_pull_zero}. Furthermore, $f^*_K$ is surjective by Theorem~\ref{thm:char_qs}, since $G_K/B_K\cong \bar{G}_K/\bar{B}_K$ for the respective Borel subgroups, and thus $f^*_K(e_3)=e$. Then by the adjoint case discussed above there exists an element $\tilde{e}\in \Ch^*(\bar{G})$ such that 
            \[
            f^*_K\bar{\pi}^*_{K/L}\bar{\pi}^*_{L/k}(\tilde{e}) = f^*_K(e_3+e_1e_2^2) =e.
            \]
            Hence the composition $f^*_K\bar{\pi}^*_{K/L}\bar{\pi}^*_{L/k}=\pi_{K/L}^*\pi_{L/k}^*f^*$ is surjective, and it follows that $\pi_{K/L}^*\pi_{L/k}^*$ is surjective as well. Moreover, the homomorphism ${\pi}^*_{L/k}$ is injective by Lemma~\ref{lem:pull_inj}, and $\Ch^*(G_L)\cong \Ch^*(G_K)$ by Theorem~\ref{thm:answer_at_p} and recollection~\refbr{expos:Chow_ring_split}, thus $\pi_{K/L}^*$ and $\pi_{L/k}^*$ are isomorphisms. The claim follows.

            $(\pi_1(G),p)=(\mut{6}{2,2},3)$. Let $K/L/k$ be the intermediate Galois extension with $[L:k]=2$. Then $\Delta(G_L)={}^3\mr{D}_4$. Consider the following commutative diagram.
            \[
            \xymatrix{
            & \Ch^*({G}_L/{B}_L) \ar[r]^{\phi_L^*}\ar[d]_{q_2} & \Ch^*({G}_L)\ar[d]^{q_1}\\
            \CH^*_{K,{T}_L}(\Spec L) \ar[r] & \CH^*_K({G}_L/{B}_L) \ar[r]^{\tilde{\phi}_L^*} & \CH^*_K({G}_L)\\
            }
            \]
            Here the bottom row is the conormed characteristic sequence from Definition~\ref{def:characteristic_map}, ${G}_L$ and ${G}_L/{B}_L$ are considered as varieties over $L$, the vertical homomorphisms are the quotient morphisms, and $\phi_L^*, \tilde{\phi}_L^*$ are the pullbacks for the projection $\phi_L\colon  {G}_L\to  {G}_L/ {B}_L$. The bottom sequence is exact in the sense of Definition~\ref{def:exact_sequence}. The algebra $\CH^*_{K,{T}_L}(\Spec L)$ was computed in Proposition~\ref{prop:classifying_quasi-trivial} and $\CH^*_K( {G}_L/ {B}_L)$ is an explicit quotient of $\Ch^*( {G}_L/ {B}_L)$ which in turn can be identified by Lemma~\ref{lem:Chow_quasi-split_homogeneous} with $\Ch^*( {G}_K/ {B}_K)^{\Gal(K/L)}$ and admits a description via Schubert cycles. Let $\rho_{K/L}\colon  {G}_K/ {B}_K\to  {G}_L/ {B}_L$ be the projection. We have $\CH^*_K( {G}_L)\cong \FF_3[e]/(e^3)$ by Theorem~\ref{thm:conormed_group_answer} and a straightforward computer-assisted computation shows that for the cycle $Z\in \Ch^*( {G}_L/ {B}_L)$ such that $\rho^*_{K/L}(Z) = Z_{s_2s_1s_3s_4}$ the element $\tilde{\phi}_L^*(Z)$ is a generator of $\CH^4_K( {G}_L)\cong \FF_3$. The homomorphism $q_1$ is an isomorphism 
            \[
            q_1\colon \Ch^*( {G}_L)\xrightarrow{\simeq} \CH^*_K( {G}_L) \cong  \FF_3[e]/(e^3)
            \]
            by Proposition~\ref{prop:norm_zero}. It follows that $\phi^*_L(Z)$ is a generator of $\Ch^4( {G}_L)\cong \FF_3$.

            Consider the following diagram.
            \[
            \xymatrix{
                \Ch^*( {G}/ {B}) \ar[d]^{\rho_{L/k}^*} \ar[r]^{\phi^*}& \Ch^*( {G}) \ar[d]^{\pi^*}\\
                \Ch^*( {G}_L/ {B}_L) \ar[r]^{\phi_L^*} \ar[d]^{\rho_{K/L}^*} & \Ch^*( {G}_L) \\
                \Ch^*( {G}_K/ {B}_K) & 
            }
            \]
            Here all the morphisms are pullbacks for the respective projections. Since $\rho_{K/L}^*(Z)=Z_{s_2s_1s_3s_4}$ is clearly $\Gal(K/k)$-invariant (the group permutes $\alpha_1,\alpha_3,\alpha_4$ and $s_1,s_3,s_4$ commute with each other), Lemma~\ref{lem:Chow_quasi-split_homogeneous} yields that there exists $Z'\in \Ch^4( {G}/ {B})$ such that $\rho^*_{L/k}(Z')=Z$. Then $\pi^*\phi^*(Z')$ is a generator of $\Ch^4( {G}_L)$ and it follows that $\pi^*$ is surjective. At the same time $\pi^*$ is injective by Lemma~\ref{lem:pull_inj}, thus it is an isomorphism. The claim follows.
            
            $(\pi_1(G),p)=(1,3)$. Let $K/L/k$ be the intermediate Galois extension with $[L:k]=2$, then $\Delta(G_L)={}^3\mr{D}_4$. Let $f\colon G\to \bar{G}$ be the adjoint quotient and consider the following diagram.
            \[
            \xymatrix{
            \Ch^*(\bar{G}) \ar[r]^{\bar{\pi}^*} \ar[d]_{f^*} & \Ch^*(\bar{G}_L)\ar[d]^{f_L^*}\\
            \Ch^*(G) \ar[r]^{\pi^*} & \Ch^*(G_L)
            }
            \]
            Here all the morphisms are pullbacks for the respective projections. The morphism $f^*_L$ is surjective by a combination of Proposition~\ref{prop:norm_zero} and Theorem~\ref{thm:char_qs}, thus an isomorphism by the dimension count and Theorem~\ref{thm:answer_at_p}. The morphism $\bar{\pi}^*$ is an isomorphism by the proof in the adjoint case above, and $\pi^*$ is injective by Lemma~\ref{lem:pull_inj}. Thus $f^*$ is an isomorphism yielding the claim.
        \end{proof}
	
    \appendix
    \section{Recollection on algebraic groups} \label{sec:appendix}

    \subsection{Some groups of multiplicative type and their representations} \label{sec:multiplicative_groups}
        \begin{expos}
        An affine algebraic group $S$ over a field $k$ is of \textit{multiplicative type}, if $S_{k^{sep}}$ is diagonalizable \cite[Chapter~12]{Mi17}. Algebraic tori, their subgroups and quotient groups are of multiplicative type. Recall \cite[Theorem~12.23]{Mi17} that the category of groups of multiplicative type over a field $k$ is dual to the category of finitely generated $\Z$-modules equipped with a continuous action of $\Gal(k^{sep}/k)$ with the contravariant equivalence given by
		\[
		S\mapsto \mc{X}^*(S):=\Hom(S_{k^{sep}},\Gmm).
		\]
        \end{expos}

        \begin{expos}        \label{expos:tori}
		A \textit{quasi-trivial} (or \textit{induced}) torus $T$ over a field $k$ is the torus corresponding to a permutation $\Gal(k^{sep}/k)$-module. One can show \cite[Lemma~12.61 and below]{Mi17} that a quasi-trivial torus is a torus $T$ admitting an isomorphism
		\[
		T\cong R_{L_1/k} \Gmm \times \hdots \times R_{L_r/k} \Gmm
		\]
		for some finite separable field extensions $L_1,\hdots,L_r/k$, where $R_{L_i/k} \Gmm$, $1\le i\le r$, are the respective Weil restrictions of a one-dimensional split torus.
	    
	Let $L/k$ be a finite separable extension of fields and let $R:=R_{L/k} \Gmm$ be the Weil restriction of a one-dimensional split torus. Put $V_R:= R_{L/k} \AAA^1_k$. Since Weil restrictions commute with products, the standard action $\Gmm\times \AAA^1_k\to \AAA^1_k$ gives rise to the action
		\[
		R\times V_R \cong R_{L/k} (\Gmm\times \AAA^1_k)\to R_{L/k} \AAA^1_k = V_R.
		\]
		Recall that $V_R= R_{L/k} \AAA^1_k\cong \AAA^n_k$ with $n=[L:k]$, and it is straightforward to see that the above action is linear.	We refer to it as the \textit{standard vector representation of $R$}. It is well-known that $R$ can be identified with the open subset of $V_R$ given by the condition $N_{L/k}\neq 0$, where $N_{L/k}$ is the norm form of the field extension $L/k$. We briefly recall some details. Let $\bar{L}/k$ be a Galois closure of $L/k$, which also coincides with the splitting field of $R$. A choice of a basis $\{e_1,e_2,\hdots, e_n\}$ for $L$ over $k$ gives rise to an isomorphism $V_R\cong \AAA^n_k=\Spec k[x_1,x_2,\hdots, x_n]$ and under this isomorphism the norm form can be described as
		\[
		N_{L/k}(x_1,x_2,\hdots, x_n) = \prod_{\sigma \in \mathcal{G}/\mathcal{G}'} \sigma (x_1e_1+ \hdots + x_ne_n) \in k[x_1, \hdots, x_n]\subseteq \bar{L}[x_1, \hdots, x_n],
		\]
		where $\mathcal{G}=\Gal(\bar{L}/k)$, $\mathcal{G}'=\Gal(\bar{L}/L)$, and the product is taken over some chosen set of representatives of the cosets $\mathcal{G}/\mathcal{G}'$. Note that the formula does not depend on the choice of the set of representatives of $\mathcal{G}/\mathcal{G}'$, since $e_i\in L$, so all $e_i$ are stable under the action of $\mathcal{G}'$.
        \end{expos}  

        \begin{expos}
        In the article we repeatedly use the following groups of multiplicative type over a field $k$ associated with a degree $2$ Galois field extension $L/k$. These groups are naturally closed subgroups of $R:=R_{L/k}\Gmm$. Recall that there is a canonical isomorphism $\mc{X}^*(R)\cong \Z\oplus\Z$ with the nontrivial element $\tau\in \Gal(L/k)$ permuting the elements of the basis, $\tau(x_1,x_2)=(x_2,x_1)$.
		\begin{itemize}\itemsep0pt
			\item 
            ${}^2\Gmm$ with $\mc{X}^*({}^2\Gmm)=\Z$, $\tau(x)=-x$. The closed embedding ${}^2\Gmm \le R$ corresponds to the projection $\mc{X}^*(R)=\Z\oplus\Z \to \Z = \mc{X}^*({}^2\Gmm)$, $(x_1,x_2)\mapsto x_1-x_2$. The group ${}^2\Gmm$ is a unique non-trivial twisted form of $\Gmm$ split by $L$.
			\item
            $\mut{2}{l}$ with $\mc{X}^*(\mut{2}{l})=\Z/l\Z$, $\tau(\bar{x})=-\bar{x}$. The closed embedding $\mut{2}{l} \le R$ corresponds to the projection $\mc{X}^*(R)=\Z\oplus\Z \to \Z/l\Z = \mc{X}^*(\mut{2}{l})$, $(x_1,x_2)\mapsto \bar{x}_1-\bar{x}_2$. If $l\ge 3$, then the group $\mut{2}{l}$ is a non-trivial twisted form of the group of $l$-th roots of unity $\mu_l$ split by $L$. If $l=2$, then $\mut{2}{2}=\mu_2$.
			\item
            $\mut{2}{2,2}$ with $\mc{X}^*(\mut{2}{2,2})=\Z/2\Z\oplus \Z/2\Z$, $\tau(\bar{x}_1,\bar{x}_2)=(\bar{x}_2,\bar{x}_1)$. The closed embedding $\mut{2}{2,2} \le R$ corresponds to the projection $\mc{X}^*(R)=\Z\oplus\Z \to \Z/2\Z\oplus \Z/2\Z = \mc{X}^*(\mut{2}{2,2})$, $(x_1,x_2)\mapsto (\bar{x}_1, \bar{x}_2)$. The group $\mut{2}{2,2}$ is a unique non-trivial twisted form of the group $\mu_2\times \mu_2$ split by $L$. Moreover, one has $\mut{2}{2,2}\cong R_{L/k}\mu_2$ for the Weil restriction $R_{L/k}\mu_2$ of $\mu_2$.
            \end{itemize}
        Restricting the standard vector representation $V_R$ to the respective subgroups we obtain the following \textit{standard vector representations} of dimension $2$:
        \[
        {}^2V_{\Gmm} := V_R|_{({}^2\Gmm)},\quad {}^2V_{\mu_l} := V_R|_{(\mut{2}{l})}, \quad{}^2V_{\mu_{2,2}} := V_R|_{(\mut{2}{2,2})}.
        \]
        Furthermore, if $l=2m$ is even, then $\mut{2}{l}$ admits a linear \textit{sign representation} ${}^2\Lambda_{\mu_{2m}}^{\pm}$ given by the Galois-invariant character $m\in \Z/2m\Z= \mc{X}^*(\mut{2}{2m})$.
            
        We also have the following groups associated with a degree $3$ Galois field extension $L/k$ and to a degree $6$ Galois field extension with $\Gal(L/k)\cong S_3$.
            \begin{itemize}\itemsep0pt
                \item $\mut{3}{2,2}$ with $\mc{X}^*(\mut{3}{2,2})=\Z/2\Z\oplus \Z/2\Z$ and $\Gal(L/k)\cong C_3$ permuting the non-zero elements. The group $\mut{3}{2,2}$ is a unique non-trivial twisted form of the group $\mu_2\times \mu_2$ split by $L$. 

                \item $\mut{6}{2,2}$ with $\mc{X}^*(\mut{6}{2,2})=\Z/2\Z\oplus \Z/2\Z$ and $\Gal(L/k)\cong S_3\cong \operatorname{Aut}(\Z/2Z\oplus \Z/2\Z)$ acting as the whole group of automorphisms. The group $\mut{6}{2,2}$ is a unique non-trivial twisted form of the group $\mu_2\times \mu_2$ with the splitting field $L$.
		\end{itemize}
        \end{expos}  

        \begin{expos} \label{expos:rep_of_weights}
	Let $S$ be a group of multiplicative type over a field $k$. Then the category $\mc{R}ep(S)$ of representations of $S$ is a semisimple abelian category and the isomorphism classes of simple objects are classified by the orbits of $\Gal(k^{sep}/k)$ acting on $\mc{X}^*(S)$ \cite[Theorem~12.30]{Mi17}. In particular, for the representation ring $\mathrm{Rep}(S)$ there is a canonical isomorphism
		\begin{equation} \label{eq:rep_via_char}
		    \mathrm{Rep}(S) \xrightarrow{\simeq} \mathbb{Z}[\mc{X}^*(S)]^{\Gal(k^{sep}/k)}.
		\end{equation}
		For $\lambda\in \mc{X}^*(S)$ we denote by $x^{\lambda}$ the corresponding element in $\mathbb{Z}[\mc{X}^*(S)]$, and for a finite $\Gal(k^{sep}/k)$-invariant subset $\mc{Q}\subseteq \mc{X}^*(S)$ the element $\sum_{\lambda\in \mc{Q}} x^\lambda$ corresponds to a representation of $S$ which we denote by $V(\mc{Q})$. The class of a representation $V\in \mc{R}ep(S)$ in $\mathrm{Rep}(S)$ is denoted by $[V]$ and we freely use the above isomorphism, in particular, we have $[V(\mc{Q})]=\sum_{\lambda\in \mc{Q}} x^\lambda$. For a homomorphism $\rho\colon S_1\to S_2$ we denote by $\Res_\rho\colon \mathrm{Rep}(S_2)\to \mathrm{Rep}(S_1)$ the induced homomorphism of the representation rings.  Isomorphism~\refbr{eq:rep_via_char} is functorial, in particular, the following diagram commutes.
		\[
		\xymatrix{ \mathrm{Rep}(S_2) \ar[r]^(0.35){\simeq} \ar[d]^{\Res_\rho}& \mathbb{Z}[\mc{X}^*(S_2)]^{\Gal(k^{sep}/k)} \ar[d]^{\Z[\mc{X}^*(\rho)]^{\Gal(k^{sep}/k)}} \\
			\mathrm{Rep}(S_1) \ar[r]^(0.35){\simeq} &  \mathbb{Z}[\mc{X}^*(S_1)]^{\Gal(k^{sep}/k)}
		}
		\]
        Let $L/k$ be a degree $2$ Galois field extension. Then we have the following formulae for the classes of the standard vector representations introduced above:
        \[
        [V_R]=x^{(1,0)}+x^{(0,1)},\quad [{}^2V_{\mu_l}]=x^{\bar{1}}+x^{-\bar{1}},\quad [{}^2V_{\mu_{2,2}}]=x^{(\bar{1},\bar{0})}+x^{(\bar{0},\bar{1})}, \quad [{}^2\Lambda^\pm_{\mu_{2m}}]=x^{\bar{m}}.
        \]
        \end{expos}        
                
        \subsection{Quasi-split simple groups} \label{sec:quasi-split_simple}

        \begin{expos}
        Let $G$ be a reductive group over a field $k$.
        \begin{itemize} \itemsep0pt
            \item {\cite[Definition~17.67]{Mi17}} $G$ is \textit{quasi-split}, if there exists a Borel subgroup $B\le G$.
            \item {\cite[Definition~18.7]{Mi17}} Suppose that $G$ is semisimple. The \textit{fundamental group} of $G$ is the kernel $\pi_1(G):=\ker(\tilde{G}\to G)$ of the simply connected cover. The group $G$ is \textit{adjoint}, if its center is trivial, or, equivalently, if its fundamental group coincides with the center of the simply connected cover.
        \end{itemize}
        \end{expos}        

        \begin{expos}
        Let $G$ be a quasi-split semisimple group over a field $k$ and let $T\le B\le G$ be a maximal torus and a Borel subgroup. Then $T_{k^{sep}}\le B_{k^{sep}}\le G_{k^{sep}}$ gives rise to a pinned root datum and since $B$ is defined over $k$, the action of $\Gal(k^{sep}/k)$ on $\mc{X}^*(T)$ stabilizes the set of simple roots $\Pi\subseteq \mc{X}^*(T)$ yielding an action of $\Gal(k^{sep}/k)$ on the Dynkin diagram of $G_{k^{sep}}$. A simply connected quasi-split semisimple group over $k$ is determined up to an isomorphism by the action of $\Gal(k^{sep}/k)$ on the Dynkin diagram of $G_{k^{sep}}$ \cite[Proposition~27.8]{KMRT98}, see also \cite[Theorem~25.33]{Mi17} and \cite[Theorem~2]{Ti66}. It follows that a quasi-split semisimple group is determined by its fundamental group and the action of $\Gal(k^{sep}/k)$ on the Dynkin diagram of $G_{k^{sep}}$. We refer to the Dynkin diagram of $G_{k^{sep}}$ together with the action of $\Gal(k^{sep}/k)$ as the \textit{type of $G$} and denote it by $\Delta(G)$, so quasi-split semisimple groups are classified by the pairs $(\Delta(G),\pi_1(G))$ consisting of the type of $G$ and the fundamental group of $G$. Note that $\Delta(G_{k^{sep}})$ is just the Dynkin diagram of $G_{k^{sep}}$.
        \end{expos}

	\begin{lemma}[{\cite[Exercise~25-4]{Mi17}}] \label{lem:quasi-split_quasi-trivial}
		Let $G$ be a quasi-split semisimple group over a field $k$ and let $T\le B\le G$ be a maximal torus and a Borel subgroup. Suppose that $G$ is simply connected or adjoint. Then $T$ is quasi-trivial.
	\end{lemma}
	\begin{proof}
		The set of simple roots associated with $T_{k^{sep}}\le B_{k^{sep}}\le G_{k^{sep}}$ is stable under the action of the Galois group $\Gal(k^{sep}/k)$ on the lattice of characters $\mc{X}^*(T)$. In the adjoint case the simple roots form a basis of $\mc{X}^*(T)$ giving a permutation basis under the action of the Galois group. In the simply connected case a permutation basis of $\mc{X}^*(T)$ is given by the set of fundamental weights, which are permuted by the Galois group, since they are dual to the simple coroots.
	\end{proof}

        \begin{expos}
            By a \textit{simple group} in the present article we mean a geometrically almost-simple group in the sense of \cite[Definition~19.7]{Mi17}, i.e. a semisimple group with $\Delta(G_{k^{sep}})$ being a connected Dynkin diagram.
        \end{expos}

         \begin{expos} \label{expos:simple_groups}
        Let $G$	be a non-split quasi-split simple group over a field $k$ and let $T\le B\le G$ be a maximal torus and a Borel subgroup. Then $\Delta(G_{k^{sep}})=\mr{A}_n$, $n\ge 2$, $\mr{D}_n$, $n\ge 4$, or $\mr{E}_6$,
        since only these connected Dynkin diagrams admit non-trivial automorphisms. We recall some structural data in these cases.
		
		$\mathbf{A_n,\, n\ge 2.}$ The action of $\Gal(k^{sep}/k)$ on the Dynkin diagram  factors through an action of $\Gal(L/k)$ for a degree $2$ Galois field extension $L/k$ which is the splitting field of the maximal torus $T$. The type of $G$ in this case is denoted by $\Delta(G)=\Att_n$. The corresponding quasi-split simply connected group is denoted by $G^{sc}(\Att_n)$ and it is isomorphic to the special unitary group $\SU(V,h)$, where $(V,h)$ is a non-degenerate hermitian form over $L/k$ of dimension $n+1$ and of maximal Witt index \cite[Example~27.9]{KMRT98}. We use the following notation for roots and weights (see \cite[Example~22.34]{Mi17}):
		\begin{itemize} \itemsep0pt
			\item $\mc{X}_{\mathbb{R}}:= \left(\bigoplus_{i=1}^{n+1} \mathbb{R} e_i\right)/\mathbb{R}(e_1+e_2+\hdots +e_{n+1})$ and $\epsilon_i\in \mc{X}_{\mathbb{R}}$ denotes the image of $e_i$,
			\item simple roots: $\alpha_i := \epsilon_i-\epsilon_{i+1} \in \mc{X}_{\mathbb{R}} $, $1\le i \le n$,
			\item
			fundamental weights: $
			\varpi_i:=\epsilon_1+\epsilon_2+\hdots+\epsilon_i$, $1\le i \le n$.
		\end{itemize} 
		In the simply connected case we have $\mc{X}^*(T)=\bigoplus_{i=1}^n \mathbb{Z} \varpi_i = \bigoplus_{i=1}^n \mathbb{Z} \epsilon_i\le \mc{X}_{\mathbb{R}}$ with the action of the nontrivial element $\tau\in \Gal(L/k)$ given by $\tau(\epsilon_i)= -\epsilon_{n+2-i}$. For the center $Z\le G^{sc}(\Att_n)$ we have
		\[
		\mc{X}^*(Z)= \mc{X}^*(T) / \bigoplus_{i=1}^n \mathbb{Z}\alpha_i \cong \mathbb{Z}/(n+1)\mathbb{Z}
		\]
		and the embedding $Z\le T$ corresponds to the projection $ \mc{X}^*(T)\to	\mc{X}^*(Z)$ given by $\epsilon_i\mapsto \bar{1}\in \mathbb{Z}/(n+1)\mathbb{Z}$, $1\le i\le n$, so $\varpi_i \mapsto \bar{\imath}\in \mathbb{Z}/(n+1)\mathbb{Z}$. The action of $\Gal(L/k)$ on $\mc{X}^*(Z)$ is compatible with the projection, thus $\tau(x)= -x$ and $Z\cong \mut{2}{n+1}$ in the notation of Section~\ref{sec:multiplicative_groups}. In particular, for a quasi-split simple group $G$ of type $\Att_n$ one has $G\cong G^{sc}(\Att_n)/\mut{2}{l}$ for some $l \mid n+1$. We denote by $G^{ad}(\Att_n):=G^{sc}(\Att_n)/\mut{2}{n+1}$ the adjoint group.
		
		$\mathbf{D_n,\, n\ge 5.}$ The action of $\Gal(k^{sep}/k)$ on the Dynkin diagram factors through an action of $\Gal(L/k)$ for a degree $2$ Galois field extension $L/k$ which is the splitting field of the maximal torus $T$. The type of $G$ in this case is denoted by $\Delta(G)={}^2\mr{D}_n$. The corresponding quasi-split simply connected group is denoted by $G^{sc}({}^2\mr{D}_n)$ and it is isomorphic to $\Spin(V,q)$, where $(V,q)$ is a non-degenerate quadratic form of dimension $2n$, of Witt index $n-1$ and with the discriminant quadratic field extension being $L/k$ \cite[Example~27.10]{KMRT98}. We use the following notation for roots and weights (see \cite[Planche~IV]{Bou81}):
		\begin{itemize}\itemsep0pt
			\item $\mc{X}_{\mathbb{R}}:=\bigoplus_{i=1}^{n} \mathbb{R} e_i$,
			\item simple roots: $\alpha_i := e_i-e_{i+1}$, $1\le i \le n-1$, $\alpha_n:=e_{n-1}+e_n$,
			\item
			fundamental weights: $
			\varpi_i:=e_1+e_2+\hdots+e_i$, $1\le i \le n-2$, $\varpi_{n-1}:=\frac{1}{2}(e_1+e_2+\hdots+e_{n-1}-e_n)$, $\varpi_{n}:=\frac{1}{2}(e_1+e_2+\hdots+e_{n-1}+e_n)$.
		\end{itemize} 
		In the simply connected case we have $\mc{X}^*(T)=\bigoplus_{i=1}^n \mathbb{Z} \varpi_i\le \mc{X}_{\mathbb{R}}$ with the action of the nontrivial element $\tau\in \Gal(L/k)$ given by $\tau(\varpi_{n-1})= \varpi_{n}$, $\tau(\varpi_{n})= \varpi_{n-1}$ and $\tau(\varpi_{i})=\varpi_i$ for $1\le i\le n-2$. This action extends to $\mc{X}_{\mathbb{R}}$ via $\tau(e_n)=-e_n$ and $\tau(e_i)=e_i$ for $1\le i\le n-1$. For the center $Z\le G^{sc}(^2\mr{D}_n)$ we have
		\[
		\mc{X}^*(Z)= \mc{X}^*(T) / \bigoplus_{i=1}^n \mathbb{Z}\alpha_i \cong \begin{cases} \mathbb{Z}/4\mathbb{Z}, &\text{$n$ is odd}\\ \mathbb{Z}/2\mathbb{Z}\times \mathbb{Z}/2\mathbb{Z}, &\text{$n$ is even.}\end{cases}
		\]
		If $n$ is odd, then the embedding $Z\le T$ corresponds to the projection  $ \mc{X}^*(T)\to	\mc{X}^*(Z)$ given by 
		\[
		\varpi_i\mapsto \overline{2i}\in \mathbb{Z}/4\mathbb{Z},\, 1\le i\le n-2,  \quad \varpi_{n-1}\mapsto \bar{1}\in \mathbb{Z}/4\mathbb{Z}, \quad\varpi_{n}\mapsto \bar{3}\in \mathbb{Z}/4\mathbb{Z}.
		\]
		If $n$ is even, then the embedding $Z\le T$ corresponds to the projection $ \mc{X}^*(T)\to	\mc{X}^*(Z)$ given by 
		\[
		\varpi_i\mapsto (\bar{i},\bar{i})\in \mathbb{Z}/2\mathbb{Z}\times \mathbb{Z}/2\mathbb{Z},\, 1\le i\le n-2,\quad \varpi_{n-1}\mapsto (\bar{1},0), \quad \varpi_{n}\mapsto (0,\bar{1}).
		\]
		The action of $\Gal(L/k)$ on $\mc{X}^*(Z)$ is compatible with the projection, thus it is given by $x\mapsto -x$ on $\mathbb{Z}/4\mathbb{Z}$, if $n$ is odd and by $(x_1,x_2)\mapsto (x_2,x_1)$ on $\mathbb{Z}/2\mathbb{Z}\times \mathbb{Z}/2\mathbb{Z}$, if $n$ is even, and in the notation of Section~\ref{sec:multiplicative_groups} we have $Z\cong \mut{2}{4}$ for $n$ odd and $Z\cong \mut{2}{2,2}$ for $n$ even. In both cases $\mc{X}^*(Z)$ admits a unique non-trivial proper quotient Galois module and this module is isomorphic to $\mathbb{Z}/2\mathbb{Z}$ with the trivial action. In particular, for a quasi-split simple group $G$ with $\Delta(G)={}^2\mr{D}_n$, $n\ge 5$, one has $G\cong G^{sc}(^2\mr{D}_n)/\pi_1(G)$ with $\pi_1(G)= 1$, $\mu_2$ or $\mut{2}{4}$ for odd $n$ and $\pi_1(G)=1$, $\mu_2$ or $\mut{2}{2,2}$ for even $n$. We denote by $G^{ad}(^2\mr{D}_{2r+1}):=G^{sc}(^2\mr{D}_{2r+1})/\mut{2}{4}$ and by $G^{ad}(^2\mr{D}_{2r}):=G^{sc}(^2\mr{D}_{2r})/\mut{2}{2,2}$ the respective adjoint groups. Note that in contrast to the split situation in the non-split case there is a unique intermediate group between $G^{sc}(^2\mr{D}_{2r})$ and $G^{ad}(^2\mr{D}_{2r})$.
		
		$\mathbf{D_4.}$ The action of $\Gal(k^{sep}/k)$ on the Dynkin diagram factors through a subgroup of the automorphism group of the diagram which is isomorphic to $S_3$. If the action factors through a subgroup of order $2$, then the same reasoning (possibly renumbering the simple roots) as in the case $n\ge 5$ applies, yielding the groups  $G^{sc}(^2\mr{D}_4)$, $G^{sc}({}^2\mr{D}_4)/\mu_2$ and $G^{ad}({}^2\mr{D}_4)=G^{sc}({}^2\mr{D}_4)/\mut{2}{2,2}$. In the remaining cases the splitting field $L$ of the maximal torus $T$ is either of degree $3$ over $k$ or of degree $6$ over $k$ with the $\Gal(L/k)\cong S_3$ in the latter case. The type of $G$ in these cases is denoted by $\Delta(G)={}^3\mr{D}_4$ and by $\Delta(G)={}^6\mr{D}_4$ respectively, with the corresponding simply connected groups denoted by $G^{sc}({}^3\mr{D}_4)$ and by $G^{sc}({}^6\mr{D}_4)$. In both cases the Galois group acts irreducibly on $\mc{X}^*(Z)\cong \Z/2\Z\oplus \Z/2\Z$ for the center $Z\le G^{sc}({}^3\mr{D}_4)$ (resp. $Z\le G^{sc}({}^6\mr{D}_4)$), yielding $Z\cong \mut{3}{2,2}$ (resp. $Z\cong \mut{6}{2,2}$) and the non-split quasi-split groups are given by $G^{sc}({}^3\mr{D}_4)$ and by $G^{ad}(^3\mr{D}_4):=G^{sc}({}^3\mr{D}_4)/\mut{3}{2,2}$ (resp. $G^{sc}(^6\mr{D}_4)$ and $G^{ad}(^6\mr{D}_4):=G^{sc}({}^6\mr{D}_4)/\mut{6}{2,2}$) in the notation of Section~\ref{sec:multiplicative_groups}.
		
		$\mathbf{E_6.}$ The action of $\Gal(k^{sep}/k)$ on the Dynkin diagram factors through an action of $\Gal(L/k)$ for a degree $2$ Galois field extension $L/k$ which is the splitting field of the maximal torus $T$. The type of $G$ in this case is denoted by $\Delta(G)={}^2\mr{E}_6$, the corresponding quasi-split simply connected group is denoted by $G^{sc}({}^2\mr{E}_6)$. We use the following notation for roots and weights (see \cite[Planche~V]{Bou81}):
		\begin{itemize}
			\item $\mc{X}_{\mathbb{R}}:= \{x_1e_1+x_2e_2+\hdots+x_8e_8\in \mathbb{R}^{8}\,|\, x_6=x_7=-x_8 \} \le \mathbb{R}^8$,
			\item simple roots: $\alpha_1 := \frac{1}{2}(e_1-e_2-e_3-e_4-e_5-e_6-e_7+e_8)$, $\alpha_2:=e_1+e_2$, $\alpha_3:=e_2-e_1$, $\alpha_4:=e_3-e_2$, $\alpha_5:=e_4-e_3$, $\alpha_6:=e_5-e_4$,
			\item
			fundamental weights: 
			$\varpi_1:=\frac{2}{3}(e_8-e_7-e_6)$, $\varpi_2:=\frac{1}{2}(e_1+e_2+e_3+e_4+e_5-e_6-e_7+e_8)$, $\varpi_3:=\frac{1}{2}(-e_1+e_2+e_3+e_4+e_5)+\frac{5}{6}(e_8-e_7-e_6)$, $\varpi_4:=e_3+e_4+e_5-e_6-e_7+e_8$, $\varpi_5:=e_4+e_5+\frac{2}{3}(e_8-e_7-e_6)$, $\varpi_6:=e_5+\frac{1}{3}(e_8-e_7-e_6)$.
		\end{itemize} 
		In the simply connected case we have $\mc{X}^*(T)=\bigoplus_{i=1}^6 \mathbb{Z} \varpi_i$ with the action of the nontrivial element $\tau\in \Gal(L/k)$ given by
        \[
        \tau(\varpi_1)= \varpi_6,\quad \tau(\varpi_2)= \varpi_2,\quad \tau(\varpi_3)= \varpi_5,\quad \tau(\varpi_4)= \varpi_4,\quad \tau(\varpi_5)= \varpi_3,\quad \tau(\varpi_6)= \varpi_1.
        \]
        For the center $Z\le G^{sc}(^2\mr{E}_6)$ we have
		\[
		\mc{X}^*(Z)= \mc{X}^*(T) / \bigoplus_{i=1}^6 \mathbb{Z}\alpha_i \cong \mathbb{Z}/3\mathbb{Z}
		\]
		and the embedding $Z\le T$ corresponds to the projection $ \mc{X}^*(T)\to	\mc{X}^*(Z)$ given by
		\[
		\varpi_2,\varpi_4\mapsto \bar{0}\in \mathbb{Z}/3\mathbb{Z}, \quad \varpi_1,\varpi_5\mapsto \bar{1}\in \mathbb{Z}/3\mathbb{Z},\quad \varpi_3,\varpi_6\mapsto \bar{2}\in \mathbb{Z}/3\mathbb{Z}.
		\]
		The action of $\Gal(L/k)$ on $\mc{X}^*(Z)$ is compatible with the projection, thus it is given by $\tau(x)=-x$ and $Z\cong \mut{2}{3}$ in the notation of Section~\ref{sec:multiplicative_groups}. We denote by $G^{ad}(^2\mr{E}_6):=G^{sc}(^2\mr{E}_6)/\mut{2}{3}$ the corresponding adjoint group. In particular, for a quasi-split simple group $G$ with $\Delta(G)={}^2\mr{E}_6$ one has $G\cong G^{sc}(^2\mr{E}_6)$ or $G\cong G^{ad}(^2\mr{E}_6)$.
         \end{expos}  

        \begin{expos} \label{expos:parabolic}
        Let $G$ be a quasi-split semisimple group over a field $k$ with a maximal torus and a Borel subgroup $T\le B\le G$ and let $I\subseteq \Delta(G_{k^{sep}})$ be a subset of its Dynkin diagram stable under the action of $\Gal(k^{sep}/k)$. Then the corresponding parabolic subgroup of $G_{k^{sep}}$ \cite[Theorem~21.91]{Mi17} is defined over $k$, and we denote it by $P_I\le G$. In particular, $P_\emptyset = B$ while maximal proper $\Gal(k^{sep}/k)$-stable subsets $I\subseteq \Delta(G_{k^{sep}})$ correspond to maximal proper parabolic subgroups $P_I\le G$. For $1\le i_1<i_2<\hdots <i_j \le \rank (G_{k^{sep}})$ such that $J:=\{\alpha_{i_1},\alpha_{i_2},\hdots,\alpha_{i_j}\} \subseteq \Delta(G_{k^{sep}})$ is $\Gal(k^{sep}/k)$-stable we put
        \[
        P_{i_1,i_2,\hdots,i_j} := P_{\Delta(G_{k^{sep}})\setminus J}.
        \]
        \end{expos}

        \subsection{Vector bundles on homogeneous varieties} \label{sec:vector_bundles}
        \begin{expos}
        Let $G$ be an algebraic group over a field $k$ and let $H\le G$ be a closed subgroup. Descent gives rise to an equivalence between the category of representations of $H$ and the category of $G$-equivariant vector bundles over $G/H$,
        \[
        \mc{R}ep(H)\xrightarrow{\simeq} \mc{V}ect_G(G/H),\quad V\mapsto (G\times V)/H,
        \]
        where $H$ acts on $G$ via $(h,g)\mapsto gh^{-1}$. Forgetting about the equivariant structure we obtain a vector bundle $\mc{V}$ over $G/H$.
        \end{expos}

        \begin{expos}
		Let $G$ be a quasi-split reductive group over a field $k$ with a maximal torus and a Borel subgroup $T\le B\le G$. According to~\refbr{expos:rep_of_weights}, a finite $\Gal(k^{sep}/k)$-invariant subset $\mc{Q}\subseteq \mc{X}^*(T)$ gives rise to $V(\mc{Q})\in \mc{R}ep(T)$ such that 
        \[
        [V(\mc{Q})_{k^{sep}}] = \sum_{\lambda\in \mc{Q}} x^{\lambda}.
        \]
        Since the embedding $T\le B$ has a canonical splitting, it follows that $V(\mc{Q})$ can be viewed as a representation of $B$ in a canonical way. We denote by $\mc{V}(\mc{Q})$ the associated vector bundle over $G/B$,
		\[
		\mc{V}(\mc{Q}):= (G\times V(\mc{Q}))/B.
		\]
	If $\mc{Q}=\{\varpi\}$ consists of a single element, we usually denote by $L(\varpi):=V(\{\varpi\})$ the corresponding linear representation and by $\mc{L}(\varpi):=\mc{V}(\{\varpi\})$ the corresponding line bundle. 

    Suppose $G$ is semisimple of rank $n$, let $f\colon \tilde{G}\to G$ be the simply connected cover and put $\tilde{T}:=f^{-1}(T)$, $\tilde{B}:=f^{-1}(B)$. Then for a finite $\Gal(k^{sep}/k)$-invariant subset $\mc{Q}\subseteq \mc{X}^*(\tilde{T})$ the vector bundle $\mc{V}(\mc{Q})$ over $\tilde{G}/\tilde{B}$ can be considered as a vector bundle over $G/B$ via the isomorphism $\tilde{G}/\tilde{B}\xrightarrow{\simeq} G/B$ induced by $f$. For a $\Gal(k^{sep}/k)$-invariant fundamental weight $\varpi_i\in \mc{X}^*(\tilde{T})$ and a $\Gal(k^{sep}/k)$-invariant set  $\{\varpi_{i_1},\varpi_{i_2}, \hdots, \varpi_{i_j}\} \subseteq \{\varpi_{1},\varpi_{2},\hdots, \varpi_{n}\}\subseteq \mc{X}^*(\tilde{T})$ of fundamental weights we put
    \[
    \mc{L}_i:=\mc{L}(\varpi_i),\quad \mc{V}_{i_1,i_2,\hdots,i_j}:=\mc{V}(\{\varpi_{i_1},\varpi_{i_2},\hdots,\varpi_{i_j}\})
    \]
    viewed as vector bundles over $G/B$.
    
    Let $1\le i_1<i_2<\hdots <i_j \le n$ be such that $J:=\{\alpha_{i_1},\alpha_{i_2},\hdots,\alpha_{i_j}\} \subseteq \Delta(G_{k^{sep}})$ is $\Gal(k^{sep}/k)$-stable, put $P:=P_{i_1,i_2,\hdots,i_j}\le G$ in the notation of~\refbr{expos:parabolic} and $\tilde{P}:=f^{-1}(P)\le \tilde{G}$. For a $\Gal(k^{sep}/k)$-invariant $\varpi\in \mathbb{Z} \varpi_{i_1}\oplus \mathbb{Z} \varpi_{i_2}\oplus \hdots \oplus \mathbb{Z} \varpi_{i_j}\subseteq \mc{X}^*(\tilde{T})$  it follows from \cite[Theorem~21.91]{Mi17} that there exists a unique linear representation $V_{\tilde{P}}(\varpi)$ of $\tilde{P}$ such that $[V_{\tilde{P}}(\varpi)_{k^{sep}}|_{\tilde{T}_{k^{sep}}}] = x^{\varpi}$. We put
    \[
    \mc{L}_P(\varpi):= (\tilde{G}\times V_{\tilde{P}}(\varpi))/\tilde{P}
    \]
    for the associated line bundle over $\tilde{G}/\tilde{P}$ which may also be viewed as a line bundle over $G/P$ via the isomorphism $\tilde{G}/\tilde{P}\xrightarrow{\simeq} G/P$ induced by $f$. For the projection $g\colon G/B\to G/P$ one has $g^*\mc{L}_P(\varpi)\cong\mc{L}(\varpi)$.
    \end{expos}  

    \begin{expos}\label{expos:line_over_group}
        Let $G$ be a semisimple group over a field $k$. Then a $\Gal(k^{sep}/k)$-invariant character $\lambda\in \mc{X}^*(\pi_1(G))$ gives rise to a line bundle
        \[
        \mc{L}(\lambda) := (\tilde{G} \times V(\{\lambda\}))/\pi_1(G)
        \]
        over $G\cong \tilde{G}/\pi_1(G)$. Suppose $G$ is quasi-split and let $T\le B\le G$ be a maximal torus and a Borel subgroup, let $f\colon \tilde{G}\to G$ be the simply connected cover and put $\tilde{T}:=f^{-1}(T)$. Consider $\varpi\in \mc{X}^*(\tilde{T})$ such that its image $\bar{\varpi}\in \mc{X}^*(\pi_1(G)) \cong \mc{X}^*(\tilde{T})/\mc{X}^*(T)$ is  $\Gal(k^{sep}/k)$-invariant. Then we have a line bundle $\mc{L}(\bar{\varpi})$ over~$G$. 
    \end{expos}
 
    \subsection{Fundamental group and Picard group of a quasi-split simple group} \label{sec:fundamental_and_Picard}
    \begin{expos} \label{expos:fund}
    Let $G$ be a quasi-split simple group over a field $k$. Then its fundamental group $\pi_1(G)$ depending on its type $\Delta(G)$ is one of the following.
		
		\begin{center}
			\def\arraystretch{1.5}
			\begin{tabular}{c|c|c|c|c|c|c}	
				$\Delta(G)$ &$\mathrm{A}_n$ & $\mathrm{B}_n, \mathrm{C}_n, \mathrm{E}_7$ &  $\mathrm{D}_{2r+1},\, r\ge 1$ & $\mathrm{D}_{2r},\, r\ge 2$ &  $\mathrm{F}_4, \mathrm{G}_2, \mathrm{E}_8$ & $\mathrm{E}_6$  \\
				\hline
				$\pi_1(G)$ & $\mu_l,\, l\mid (n+1)$ &  $1$, $\mu_2$ &  $1$, $\mu_2$, $\mu_4$ & $1$, $\mu_2^{so}$, $\mu_2^{hs}$, $\mu_2\times \mu_2$ &  $1$ &  $1$, $\mu_3$
				\vspace{10pt}			
			\end{tabular}			
			\begin{tabular}{c|c|c|c|c|c|c}	
				$\Delta(G)$ &  $\Att_n$, $n\ge 2$  & ${}^2\mathrm{D}_{2r+1},\, r\ge 1$  & ${}^2\mathrm{D}_{2r},\, r\ge 2$ & ${}^3\mathrm{D}_{4}$ & ${}^6\mathrm{D}_{4}$ & ${}^2\mathrm{E}_6$ \\
				\hline
				$\pi_1(G)$ & $\mut{2}{l}, \, l\mid (n+1)$  & $1$, $\mu_2$, $\mut{2}{4}$ & $1$, $\mu_2$, $\mut{2}{2,2}$ & $1$, $\mut{3}{2,2}$ & $1$, $\mut{6}{2,2}$ & $1$, $\mut{2}{3}$ 
				\vspace{10pt}			
			\end{tabular}
		\end{center}
        Here we use the notation of Sections~\ref{sec:multiplicative_groups} and~\ref{sec:quasi-split_simple}. The list arises from the fact that $\pi_1(G)$ is a subgroup of the center $Z(\tilde{G})$ of the simply connected cover $\tilde{G}\to G$, from the identification of $\mc{X}^*(Z(\tilde{G}))$ with the weight lattice of $\tilde{G}_{k^{sep}}$ modulo its root lattice \cite[Proposition~21.8, Proposition~23.59]{Mi17} and from the description of the corresponding lattices \cite[Planche~I-IX]{Bou81}. In the non-split quasi-split case see also the discussion  of the Galois action on $\mc{X}^*(Z(\tilde{G}))$ given in Section~\ref{sec:quasi-split_simple}.
        
        Let $G$ be a split simply connected simple group with $\Delta(G)=\mr{D}_{2r}$ and $T\le G$ be a split maximal torus. Then there are three embeddings $\mu_2\to \mu_2\times \mu_2 \cong  Z(G)\le T$ with the dual homomorphisms $\mc{X}^*(T)\to \mc{X}^*(\mu_2)\cong\Z/2\Z$ given by
		\begin{gather*}
			\varpi_i \mapsto \bar{0} \in \Z/2\Z,\quad 1\le i\le 2r-2, \qquad \varpi_{2r-1},\varpi_{2r}\mapsto \bar{1} \in \Z/2\Z,\\
			\varpi_{2i} \mapsto \bar{0} \in \Z/2\Z,\quad 1\le i\le r, \qquad \varpi_{2i-1} \mapsto \bar{1}\in\Z/2\Z,\quad 1\le i\le r,\\
			\varpi_{2i}, \varpi_{2r-1} \mapsto \bar{0} \in \Z/2\Z,\quad 1\le i\le r-1, \qquad \varpi_{2i-1},\varpi_{2r} \mapsto \bar{1}\in\Z/2\Z,\quad 1\le i\le r-1.
		\end{gather*}
		For the last two embeddings the quotient groups $G/\mu_2$ are isomorphic. We denote in the table the first and the second embeddings as $\mu_2^{so}\le T\le G$ and $\mu_2^{hs}\le T\le G$ respectively.
	\end{expos}

        \begin{expos} \label{expos:Picard}
	Let $G$ be a semisimple group over a field $k$ and adopt the notation of~\refbr{expos:line_over_group}. The homomorphism
	\[
	\mc{X}^*(\pi_1(G))^{\Gal(k^{sep}/k)} \xrightarrow{} \Pic(G),\quad \lambda\mapsto \mc{L}(\lambda),
	\]
	is an isomorphism \cite[Corollary~18.26]{Mi17}. Thus the Picard group of a quasi-split simple group $G$ is given by the following table, and if $G$ is not in the table, then $\Pic(G)=0$. 
	
	\begin{center}
		\def\arraystretch{1.5}
		\begin{longtable}{l|l|l}
			\caption{Picard group of a quasi-split simple algebraic group}\\			
			$\Delta(G)$ & $\pi_1(G)$ & $\Pic(G)$  \\
			\hline
			${}^{\hphantom{2}}\mr{A}_{n}$ & $\mu_l$, $l\mid (n+1)$ & $\Z/l\Z \cdot [\mc{L}(\bar{\varpi}_1)]$ \\
			
			$\Att_{2r-1}$ & $\mut{2}{2m}$, $m\mid r$ & $\Z/2\Z\cdot [\mc{L}(\bar{\varpi}_m)]$\\
			
			${}^{\hphantom{2}}\mr{B}_{n}$ & $\mu_2$& $\Z/2\Z\cdot [\mc{L}(\bar{\varpi}_n)]$  \\
			
			${}^{\hphantom{2}}\mr{C}_{n}$ & $\mu_2$ & $\Z/2\Z\cdot [\mc{L}(\bar{\varpi}_1)]$  \\

            ${}^{\hphantom{2}}\mr{D}_{2r}$, $r\ge 2$ & $\mu_2^{so}$ & $\Z/2\Z \cdot [\mc{L}(\bar{\varpi}_{2r-1})]$ \\				
   
            ${}^{\hphantom{2}}\mr{D}_{2r+1}$, $r\ge 1$ & $\mu_2$ & $\Z/2\Z \cdot [\mc{L}(\bar{\varpi}_{2r})]$ \\				
   			
			${}^2\mr{D}_{n}$, $n\ge 3$ & $\mu_2$ & $\Z/2\Z \cdot [\mc{L}(\bar{\varpi}_{n-1})]$ \\
			${}^{\hphantom{2}}\mr{D}_{2r}$, $r\ge 2$	& $\mu_2^{hs}$ & $\Z/2\Z \cdot [\mc{L}(\bar{\varpi}_{1})]$ \\					
			${}^{\hphantom{2}}\mr{D}_{2r+1}$, $r\ge 1$ & $\mu_4$ & $\Z/4\Z\cdot [\mc{L}(\bar{\varpi}_{2r})]$ \\
			${}^{2}\mr{D}_{2r+1}$, $r\ge 1$ & $\mut{2}{4}$ & $\Z/2\Z\cdot [\mc{L}(\bar{\varpi}_{1})]$ \\			
			${}^{\hphantom{2}}\mr{D}_{2r}$, $r\ge 2$ & $\mu_2\times \mu_2$ & $\Z/2\Z\cdot [\mc{L}(\bar{\varpi}_{2r-1})]\oplus \Z/2\Z\cdot [\mc{L}(\bar{\varpi}_{2r})]$ \\
			
			${}^{2}\mr{D}_{2r}$, $r\ge 2$ & $\mut{2}{2,2}$ & $\Z/2\Z\cdot [\mc{L}(\bar{\varpi}_{1})]$ \\
			
			${}^{\hphantom{2}}\mr{E}_6$ & $\mu_3$ & $\Z/3\Z \cdot [\mc{L}(\bar{\varpi}_{1})]$   \\			
			
			${}^{\hphantom{2}}\mr{E}_7$ & $\mu_2$ & $\Z/2\Z \cdot [\mc{L}(\bar{\varpi}_{2})]$   \\
		\end{longtable}
	\end{center}
	Here we use the numbering of the fundamental weights as in \cite[Planche~I-IX]{Bou81}. Other possible generators for $\Pic(G)$ can be easily read off from the description of the projection $\mc{X}^*(\tilde{T})\to \mc{X}^*(\pi_1(G))$.
        \end{expos}

        \begin{lemma} \label{lem:Pic_pull_zero}
            Let $G$ be a simply connected quasi-split simple group over a field $k$ and let $N_1\le N_2 \le Z(G)$ be central subgroups with $N_1\neq N_2$. Then the pullback
            \[
            \pi^*\colon \Pic(G/N_2) \to \Pic(G/N_1)
            \]
            for the projection $\pi\colon G/N_1\to G/N_2$ is zero unless $(\Delta(G),N_1,N_2)$ is one of the following 
            \begin{itemize} \itemsep0pt
                \item 
                $(\mr{A}_n,\mu_{l_1},\mu_{l_2})$ with $1\neq l_1\mid l_2\mid n+1$,
                \item
                $(\Att_{2r-1},\mut{2}{2m_1},\mut{2}{2m_2})$ with $m_1\mid m_2 \mid r$ and $\frac{m_2}{m_1}$ being odd, 
                \item 
                $(\mr{D}_{2r},\mu^{so}_2,\mu_2\times \mu_2)$, $(\mr{D}_{2r},\mu_2^{hs},\mu_2\times \mu_2)$, $(\mr{D}_{2r+1},\mu_2,\mu_4)$.
            \end{itemize}
        \end{lemma}
        \begin{proof}
        Straightforward from the above.            
        \end{proof}

       \subsection{Chow ring of a split simple group} \label{sec:Chow_ring_classic}
        \begin{expos} \label{expos:Chow_rational}
        Let $G$ be a reductive group over a field $k$. It is well-known that $\CH^*(G)\otimes \mathbb{Q} \cong \mathbb{Q}$. This could be seen via the isomorphisms
        \[
        \mathbb{Q}\cong \CH^*(G_K)\otimes \mathbb{Q} \cong \CH^*(G)\otimes \mathbb{Q},
        \]
        with $K$ being the splitting field of $G$, where the first isomorphism is given by \cite[Remark~2 on p.~21]{Gro58} while the second one then follows from the standard
        pushforward-pullback argument (as in Lemma~\ref{lem:pull_inj}). Alternatively, one has isomorphisms
        \[
        \CH^*(G)\otimes \mathbb{Q}\cong K_0(G)\otimes \mathbb{Q} \cong \mathbb{Q}
        \]
        with the first one given by the Chern character \cite[Example~15.2.16]{Fu98} and the second one given by \cite[Corollary~7.8]{Me97}.
        \end{expos}

    \begin{expos} \label{expos:Chow_ring_split}
	Let $G$ be a split reductive group over a field $k$. It follows from the Bruhat decomposition and \cite[Proposition~1]{To14a} that the K\"unneth formula $\CH^*(G)\otimes \CH^*(G)\xrightarrow{\simeq} \CH^*(G\times G)$ holds, endowing $\CH^*(G)$ with the structure of a Hopf algebra. Let $p\in \NN$ be a prime and put $\Ch^*(G):=\CH^*(G)\otimes \FF_p$. Then it follows from a Borel's result on the structure of Hopf algebras \cite[Theorem 7.11 and Proposition 7.8]{MM65} that 
	\[
	\Ch^*(G) \cong \FF_p[e_1,e_2,\hdots,e_s]/(e^{p^{k_1}}_1,e^{p^{k_2}}_2,\hdots, e^{p^{k_s}}_s),\quad \deg e_i=d_i,
	\]
	for some integers $s,d_i,k_i\in \NN$.
 
 Suppose $G$ is a simple group. For $k=\mathbb{C}$ the integers $s,d_i,k_i$ were determined in \cite[Theorem~6]{Kac85}, and the answer does not depend on $k$ e.g. because of the combinatorial presentation for $\Ch^*(G)$ given by the characteristic sequence. One has $\Ch^*(G) \not\cong\FF_p$ only for a finite list of primes $p$. These primes and the respective parameters are summarized in the following table, where $v_p$ is the $p$-adic valuation. 
	
	\begin{center}
		\def\arraystretch{1.5}
		\begin{longtable}{l|l|l|l|l|l}
			\caption{$\Ch^*(G)$ for a split simple algebraic group $G$}\\
                \label{tab:Kac}
                $\Delta(G)$ & $\pi_1(G)$ & $p$ & $s$ & $d_i$, $i=1,2,\ldots, s$ & $k_i$, $i=1,2,\ldots, s$ \\
			\hline
			$\mr{A}_{n}$ & $\mu_l$, $l\,|\, (n+1)$ & $p\,|\, l$ & $1$ & $1$ & $v_p(n+1)$ \\
			\hline
			$\mr{B}_{n}$ & $1$ & $2$ & $[\frac{n-1}{2}]$ & $2i+1$ & $[\log_2 \frac{2n}{2i+1}]$ \\
			
			& $\mu_2$& $2$ & $[\frac{n+1}{2}]$ & $2i-1$ & $[\log_2 \frac{2n}{2i-1}]$ \\
			\hline
			$\mr{C}_{n}$ & $\mu_2$ & $2$ & $1$ & $1$ & $v_2(n)+1$ \\
			\hline
			$\mr{D}_{n}$, $n\ge 3$ & $1$ & $2$ & $[\frac{n}{2}]-1$ & $2i+1$ & $[\log_2 \frac{2n-1}{2i+1}]$ \\						
			$\mr{D}_{2r}$, $r\ge 2$	& $\mu_2^{so}$ & $2$ & $r$ & $2i-1$ & $[\log_2 \frac{4r-1}{2i-1}]$ \\	
            
            & $\mu_2^{hs}$ & $2$ & $r$ & $1$, $i=1$ & $v_2(r)+1$, $i=1$ \\			
            
			&  & & & $2i-1$, $i\ge 2$ & $[\log_2 \frac{4r-1}{2i-1}]$, $i\ge 2$ \\						
            & $\mu_2\times \mu_2$ & $2$ & $r+1$ & $1$, $i=1$& $v_2(r)+1$, $i=1$\\						
			&  & & & $2i-3$, $i\ge 2$ & $[\log_2 \frac{4r-1}{2i-3}]$, $i\ge 2$ \\						

            $\mr{D}_{2r+1}$, $r\ge 1$	& $\mu_2$, $\mu_4$ & $2$ & $r$ & $2i-1$ & $[\log_2 \frac{4r+1}{2i-1}]$ \\	
   
			\hline
			$\mr{E}_6$ & $1,\mu_3$& $2$ & $1$ & $3$ & $1$ \\			
			& $1$ & $3$ & $1$ & $4$ & $1$ \\	
			& $\mu_3$ & $3$ & $2$ & $1,4$ & $2,1$ \\			
			\hline
			$\mr{E}_7$ & $1$ & $2$ & $3$ & $3,5,9$ & $1,1,1$ \\
			
			& $\mu_2$ & $2$ & $4$ & $1,3,5,9$ & $1,1,1,1$ \\
			
			& $1,\mu_2$ & $3$ & $1$ & $4$ & $1$ \\
			\hline
			$\mr{E}_8$& $1$  & $2$ & $4$ & $3,5,9,15$ & $3,2,1,1$  \\
			
			& & $3$ & $2$ & $4,10$ & $1,1$ \\
			
			& & $5$ & $1$ & $6$ & $1$ \\
			\hline
			$\mr{F}_4$ & $1$ &  $2$ & $1$ & $3$ & $1$ \\
			& &  $3$ & $1$ & $4$ & $1$ \\			
			\hline
			
			$\mr{G}_2$ & $1$ & $2$ & $1$ & $3$ & $1$ \\
		\end{longtable}
	\end{center}
\end{expos}

\Addresses
 
\end{document}